\documentclass[a4paper,12pt]{article}
\usepackage[latin1]{inputenc}
\usepackage{amsthm}
\usepackage{amsfonts}
\usepackage[dvips]{graphicx}
\usepackage{latexsym}
\usepackage{amsmath}
\usepackage{color}

\usepackage{amssymb}

\def\sergio #1{{\color{black}#1}}
\def\sergiotwo #1{{\color{black}#1}}

\newtheorem{lem}{Lemma}[section]
\newtheorem{thm}{Theorem}[section]
\newtheorem{defn}{Definition}[section]
\newtheorem{prop}{Proposition}[section]

\newtheorem{oss}{Remark}[section]

\newcommand{\uvec}{\mathbf{u}}
\newcommand{\vvec}{\mathbf{v}}
\newcommand{\hvec}{\mathbf{h}}
\newcommand{\nvec}{\mathbf{n}}
\newcommand{\wvec}{\mathbf{w}}
\newcommand{\e}{\epsilon}
\newcommand{\phie}{\varphi_{\epsilon}}

\newcommand{\mue}{\mu_{\epsilon}}
\newcommand{\Fe}{F_{\epsilon}}

\newcommand{\me}{m_{\epsilon}}
\newcommand{\etae}{\eta_\epsilon}

\newcommand{\wph}{\widetilde{\varphi}}

\allowdisplaybreaks[4]

\begin{document}

\title{Nonlocal Cahn-Hilliard-Hele-Shaw systems with \\ singular potential and degenerate mobility}

\author{
{Cecilia Cavaterra\thanks{
Dipartimento di Matematica ``F. Enriques'',
Universit\`{a} degli Studi di Milano, Via C. Saldini 50, 20133 Milano, Italy.
Istituto di Matematica Applicata e Tecnologie Informatiche ``Enrico Magenes'', CNR,
Via Ferrata 1, 27100 Pavia, Italy.
E-mail: \textit{cecilia.cavaterra@unimi.it}, ORCID ID 0000-0002-2754-7714}}
\and
{Sergio Frigeri \thanks{
Dipartimento di Matematica ``F. Enriques'',
Universit\`{a} degli Studi di Milano, Via C. Saldini 50, 20133 Milano, Italy.
E-mail: \textit{sergio.frigeri@unimi.it}, ORCID ID 0000-0002-7582-5205}}
\and
{Maurizio Grasselli\thanks{Dipartimento di Matematica,
Politecnico di Milano, Via E. Bonardi 9,
20133 Milano, Italy. E-mail: \textit{mau\-ri\-zio.grasselli@polimi.it}, ORCID ID 0000-0003-2521-2926}}}

\maketitle

\begin{abstract}\noindent
We study a Cahn-Hilliard-Hele-Shaw (or Cahn-Hilliard-Darcy) system for an incompressible mixture of two fluids. The relative concentration difference $\varphi$ is governed by
a convective nonlocal Cahn-Hilliard equation with degenerate mobility and logarithmic potential. The volume averaged fluid velocity $\mathbf{u}$ obeys a Darcy's law
depending on the so-called Korteweg force $\mu\nabla \varphi$, where $\mu$ is the nonlocal chemical potential. In addition, the kinematic viscosity $\eta$ may depend on $\varphi$. We establish first the existence of a global weak solution which satisfies
the energy identity. Then we prove the existence of a strong solution. Further regularity results on the pressure and on $\mathbf{u}$ are also obtained. Weak-strong uniqueness is demonstrated in the two dimensional case. In the three-dimensional case, uniqueness of weak solutions holds if $\eta$ is constant. Otherwise, weak-strong uniqueness is shown by assuming that the pressure of the strong solution is $\alpha$-H\"{o}lder continuous in space for $\alpha\in (1/5,1)$.
\bigskip



\end{abstract}

\section{Introduction}
\setcounter{equation}{0}

The behavior of an incompressible binary fluid flow in a Hele-Shaw cell occupying a bounded domain $\Omega \subset \mathbb{R}^d$, $d=2,3$,  can be described through a diffuse interface model which reduces to the following system in the Boussinesq approximation  (see \cite{LLG1,LLG2})
\begin{align}
& \eta(\varphi)\mathbf{u} + \nabla\pi=\mu\nabla\varphi  \label{Sy001}\\
& \mbox{div}(\mathbf{u})=0 \label{Sy002}\\
& \varphi_t+ \mathbf{u} \cdot\nabla\varphi=\mbox{div}(m(\varphi)\nabla\mu) \label{Sy003} \\
& \mu=-\Delta \varphi + F^\prime(\varphi) - \vartheta_0 \varphi \label{Sy004}
\end{align}
in $Q_T := \Omega \times (0,T)$.  Here $\varphi: \Omega \times [0,T] \to [-1,1]$ is the relative concentration difference
$\mathbf{u}: \Omega \times [0,T] \to \mathbb{R}^d$ the volume averaged fluid velocity, and $\eta(\cdot)$ is the kinematic viscosity given by
\begin{equation}
\eta(s)= \nu_1 \frac{1+s}{2} + \nu_2\frac{1-s}{2}, \quad s \in [-1,1], \label{visc}
\end{equation}
where $\nu_1>0$ and $\nu_2>0$ are the (constant) viscosities of the two fluids. The function $F$ is the mixing entropy density, namely
\begin{equation}
F(s)=\frac{\vartheta}{2}\big((1+s)\log(1+s)+(1-s)\log(1-s)\big), \quad s \in (-1,1), \label{mixentr}
\end{equation}
where $\vartheta>0$ is the absolute temperature and $\vartheta_0>\vartheta$ is the critical temperature.
Moreover, $\pi$ is the pressure and $m(\cdot) \geq 0$ is the mobility. Some other constants have been set equal to unity and gravity has been neglected for the sake of simplicity (see \cite[Sec.7]{Gio}). It is worth recalling that the original model is typically two dimensional. Nonetheless, in three dimensions, the system can model fluid flow in a porous medium and, in particular, it is used in solid tumor growth modeling (see \cite{DGL,GaLa,Gio} and references therein).

In this setting, i.e., without approximating \eqref{mixentr} with a regular double well potential and taking no-flux boundary conditions, the only theoretical results available so far were proven in \cite{Gio} for constant mobility. More precisely, in two spatial dimensions, the author
established the existence of a weak solution, its conditional uniqueness as well as the existence and uniqueness of strong solutions. Instead, in three spatial
dimensions, existence and uniqueness of a strong solution were proven locally in time or for small initial data. The case $\eta$ constant was formerly analyzed in \cite{GGW}.
Previous results for $\eta$ non-constant were only known for regular potentials, that is, smooth approximations defined on $\mathbb{R}$ of the singular potential $W(s) = F(s) - \frac{\theta_0}{2}s^2$ (see \cite{WW,WZ} and also \cite{DGL}). For a detailed analysis of contributions in the case $\eta$ constant with regular potential, we refer the reader to \cite{Gio}. However, as is well known, in such cases it is not possible to ensure the physical requirement $\varphi \in [-1,1]$.
Open issues for system \eqref{Sy001}-\eqref{Sy004} are the uniqueness of weak solutions and the existence of global strong solutions in dimension three (even in the case of constant $\eta$).

An alternative system is based on the nonlocal Cahn-Hilliard equation. In this case, the standard local free energy
$$
E(\varphi) = \int_\Omega \left(\frac{\vert \nabla \varphi\vert^2}{2} + W(\varphi)\right) dx,
$$
whose functional derivative is the chemical potential $\mu$, is replaced by the nonlocal free energy
$$
\mathcal{E}(\varphi) = -\frac{1}{2}\int\int_{\Omega\times\Omega} J(x-y)\varphi(x)\varphi(y)dxdy + \int_\Omega F(\varphi)dx,
$$
where $J:\mathbb{R}^d \to \mathbb{R}$ is a suitable interaction kernel such that
\begin{equation}
J(x)=J(-x). \label{kernel}
\end{equation}
Note that the nonlocal term represents the demixing effects which compete with the entropy mixing
(see \cite{GL1,GL2,GL3} for a macroscopic derivation from a microscopic model in a periodic context, see also the discussion in \cite{Fr}).
Then, taking $\mu$ as the functional derivative of $\mathcal{E}$ we obtain the nonlocal version of the
Cahn-Hilliard-Hele-Shaw system
\begin{align}
& \eta(\varphi)\mathbf{u} + \nabla\pi=\mu\nabla\varphi  \label{Sy01}\\
& \mbox{div}(\mathbf{u})=0 \label{Sy02}\\
& \varphi_t+\mathbf{u}\cdot\nabla\varphi=\mbox{div}(m(\varphi)\nabla\mu) \label{Sy03} \\
& \mu=-J\ast\varphi + F^\prime(\varphi) \label{Sy04}
\end{align}
in $Q_T$.
This system was analyzed in \cite{DPGG} in the case of constant viscosity and mobility. In particular, the global
well-posedness of weak solutions and the existence of global strong solutions were established also in dimension three.
An improvement with respect to what is known for the corresponding system \eqref{Sy001}-\eqref{Sy004}.

In the present contribution we take a step further by considering non-constant viscosity and degenerate mobility,
that is,
\begin{equation}
m(s)=1-s^2, \qquad s\in [-1,1]. \label{degmob}\\
\end{equation}
More precisely, our goal is to analyze \eqref{Sy01}-\eqref{Sy04} equipped with following boundary and initial conditions
\begin{align}
&\mathbf{u}\cdot\mathbf{n} = 0, \quad m(\varphi)\frac{\partial\mu}{\partial \mathbf{n}}=0, \qquad\mbox{on }
\Gamma \times (0,T),\label{sy5}\\
& \varphi(0)=\varphi_0, \qquad\mbox{in }\Omega, \label{sy6}
\end{align}
where $\mathbf{n}$ is the outward normal to $\Gamma:=\partial\Omega$ and $\varphi_0$ is a given initial condition.

We first prove the existence of a global weak solution which satisfies an energy identity (see Section \ref{sec:weakex}).
The existence of a global strong solution is then analyzed in Section \ref{sec:strongex}. As we shall see, the combination of degenerate mobility and singular potential
will play a basic role (cf. \cite{FGG2, FGGS} and references therein). In order to carry out our existence argument, we need an unexpected ingredient, that is, the spatial H\"{o}lder regularity of the pressure. This is obtained by means of a celebrated De Giorgi's result. Moreover, we need an existence result on the convective nonlocal Cahn-Hilliard equation which is a refinement of a previous one contained in \cite{FGGS,FGR} (see Section \ref{cnCH}).
Section \ref{sec:regpiu} is devoted to establish further regularity properties for $\pi$ and $\mathbf{u}$. These technical results are helpful, in particular,
to prove a conditional weak-strong uniqueness in Section \ref{sec:unique} for the three-dimensional case. The uniqueness issue is open for weak solutions even in the two-dimensional case (cf. \cite{FGG1} for the nonlocal Cahn-Hilliard-Navier-Stokes system). We can prove weak-strong uniqueness in dimension two. In three dimensions the result is conditional. More precisely, we need to require that the pressure of the strong solution is $\alpha$-H\"{o}lder continuous in space with $\alpha\in (1/5,1)$. On the other hand, if $\eta$ is constant then uniqueness of weak solutions holds. Section \ref{sec:conrem} is devoted to some comments on possible further investigations.
Section \ref{sec:GNineq} is an appendix containing some Gagliardo-Nirenberg type estimates which are mostly used in Section \ref{sec:regpiu}.

\section{Notation and useful results}
\setcounter{equation}{0}
\label{sec:notation}

Here we introduce some notation and we report some results which will be used in the sequel.
From now on $\Gamma$ will be smooth enough.

We set
\begin{align}
&\mathcal{V} := \{\mathbf{v} \in C^\infty_c(\overline{\Omega})^d:\mbox{ div}(\mathbf{v})=0\}, \\
&G_{div}:=\overline{\mathcal{V}}^{L^2(\Omega)^d}, \quad L^r_{div}(\Omega)^d= G_{div} \cap L^r(\Omega)^d, \; r>2,\\
&V_{div}:= G_{div} \cap H^1(\Omega)^d,\\
&V_{0,div}:= G_{div} \cap H^1_0(\Omega)^d, \quad V^r_{0,div}(\Omega)^d= L^r_{div}(\Omega)^d \cap W^{1,r}_0(\Omega)^d, \; r> 2.
\end{align}
For the sake of brevity, we also set $H:=L^2(\Omega)$, $H^d:=L^2(\Omega)^d$, $H^{d\times d}:=L^2(\Omega)^{d\times d}$,
denoting by $\|\cdot\|$ and $(\cdot,\cdot)$ the norm and the
scalar product, respectively, on $H$, $H^d$ or $H^{d\times d}$.
Moreover, we define  $V:=H^1(\Omega)$ and $V^d:=H^1(\Omega)^d$.

If $X$ is a (real) Banach space, $X'$ indicates its dual
and $\langle\cdot,\cdot\rangle$ stands
for the duality pairing between $X$ and $X'$.
For every $f\in V'$ we denote by $\overline{f}$ the average of $f$
over $\Omega$, i.e., $\overline{f}:=|\Omega|^{-1}\langle
f,1\rangle$. Here $|\Omega|$ is the Lebesgue measure of $\Omega$.
Let us introduce also the spaces $H_0:=\{v\in
H: (v,1)=0\}$, $V_0:= V \cap H_0$ and $ V_0':=\{f\in V':\langle f,1\rangle=0\}$. We note that
the dual space $(V_0)^\prime$ can be proven to be linearly isomorphic to $V_0'$.
The linear operator $A:V\to V'$, $A\in\mathcal{L}(V,V')$, is defined by
$$\langle Au,v\rangle:=\int_{\Omega}\nabla u\cdot\nabla v,\qquad\forall u,v\in V.$$
We recall that $A$ maps $V$ onto $V_0'$ and the restriction of $A$
to $V_0$ maps $V_0$ onto $V_0'$ isomorphically. Let us denote by
$\mathcal{N}:V_0'\to V_0$ the inverse map defined by
$$A\mathcal{N}f=f,\quad\forall f\in V_0'\qquad\mbox{and}\qquad\mathcal{N}Au=u,\quad\forall u\in V_0.$$
As is well known, for every $f\in V_0'$, $\mathcal{N}f$ is the
unique solution with zero mean value of the Neumann problem
\begin{equation*}
\left\{\begin{array}{ll}
-\Delta u=f,\qquad\mbox{in }\Omega\\
\frac{\partial u}{\partial \mathbf{n}}=0,\qquad\mbox{on }\Gamma.
\end{array}\right.
\end{equation*}
Furthermore, the following relations hold
\begin{align}
&\langle Au,\mathcal{N}f\rangle=\langle f,u\rangle,\qquad\forall u\in V,\quad\forall f\in V_0',\nonumber\\
&\langle f,\mathcal{N}g\rangle=\langle
g,\mathcal{N}f\rangle=\int_{\Omega}\nabla(\mathcal{N}f)
\cdot\nabla(\mathcal{N}g),\qquad\forall f,g\in V_0'.\nonumber
\end{align}
Recall that $A$ can be also viewed as an unbounded operator
$A:D(A)\subset H\to H$ where $D(A)=\{\phi\in
H^2(\Omega):\frac{\partial\phi}{\partial \mathbf{n}}=0\mbox{ on }\Gamma\}$.
The operator $A$ has a non-decreasing sequence of eigenvalues $\{\lambda_j\}_{j\in \mathbb{N}_0}$ such that $\lambda_1=0$ and $\lambda_j>0$
for all $j\geq 2$. The corresponding eigenfunctions $\{w_j\}_{j\in \mathbb{N}_0}$ form an orthonormal basis of $H$ and they are orthogonal
in $V$. Moreover, any non-zero constant is an eigenfunction associated with $\lambda_1$. We can take $w_1= \vert \Omega \vert^{-1/2}$ so that $\Vert w_1\Vert=1$.

Here below we report a crucial result for our analysis, namely a general result on the H\"{o}lder regularity of solutions to the Neumann problem:
\begin{align}
&-\big(a_{ij}u_{x_i}\big)_{x_j}=\mbox{div}\,\mathbf{f}-f\,,\qquad\mbox{in }\Omega\,,\label{Neu1}\\[1mm]
&(a_{ij}u_{x_i}+f_j)n_j=\psi\,,\qquad\mbox{on }\Gamma\,,\label{Neu2}
\end{align}
where $\Omega\subset\mathbb{R}^d$ is a bounded smooth domain, $(a_{ij})$ is a $d\times d$ symmetric matrix
with entries $a_{ij}\in L^\infty(\Omega)$ satisfying the ellipticity condition
\begin{align}
&\Lambda_\ast|\xi|^2\leq a_{ij}(x)\xi_i\xi_j\leq \Lambda^\ast|\xi|^2\,,\qquad\forall\xi\in\mathbb{R}^d
\quad\mbox{and for a.e. }x\in\Omega\,,\label{ell-cond}
\end{align}
for some $0<\Lambda_\ast\leq\Lambda^\ast$.
Referring to the general theory of linear
elliptic equations with measurable coefficients we first recall that  (see, e.g., \cite[Chapter 9, Theorem 10.1]{DiBen})
\begin{prop}
\label{DiBenedetto_0}
If $\mathbf{f}\in L^2(\Omega)^d$, $f\in L^q(\Omega)$, $\psi\in L^{\hat{q}}(\Gamma)$,
where $q,\hat{q}>1$ if $d=2$ and $q=6/5$, $\hat{q}=4/3$ if $d=3$,
and the compatibility condition $\int_\Omega f=\int_{\Gamma}\psi$ is satisfied, then problem
\eqref{Neu1}-\eqref{Neu2} admits a weak solution $u\in H^1(\Omega)$, which is unique up to a constant.
\end{prop}
Moreover, we have the following H\"{o}lder continuity result, which also gives an estimate
of the H\"{o}lder norm of the solution in terms of the data of the problem (see \cite[Chapter 9, Theorem 18.3]{DiBen})
\begin{prop}\label{DiBenedetto}
Let $u\in H^1(\Omega)$ be a solution to the Neumann problem \eqref{Neu1}-\eqref{Neu2}, with
$\mathbf{f}$ and $f$ satisfying
\begin{align}\label{DiBen1}
&\mathbf{f}\in L^{d+\varepsilon}(\Omega)^d\,,\quad f\in L^{\frac{d+\varepsilon}{2}}(\Omega)\,,\quad\mbox{for some }\varepsilon>0\,,
\end{align}
and $\psi\in L^{d-1+\sigma}(\Gamma)$, for some $\sigma\in(0,1)$.
Then $u$ is H\"{o}lder continuous in $\overline{\Omega}$, and there exist  constants $\Theta>0$ and
$\alpha\in(0,1)$, depending on
$\Vert\mathbf{f}\Vert_{L^{d+\varepsilon}(\Omega)^d}, \Vert f\Vert_{L^{\frac{d+\varepsilon}{2}(\Omega)}},
\Vert\psi\Vert_{L^{d-1+\sigma}(\Gamma)}, \Lambda_\ast$, $\Lambda^\ast,\varepsilon,d$ and on the
$C^1$-smoothness of $\Gamma$, such that
\begin{align}\label{DiBen2}
&\Vert u\Vert_{C^\alpha(\overline{\Omega})}\leq
\Theta \big(\Vert\mathbf{f}\Vert_{L^{d+\varepsilon}(\Omega)^d}, \Vert f\Vert_{L^{\frac{d+\varepsilon}{2}}(\Omega)},
\Vert\psi\Vert_{L^{d-1+\sigma}(\Gamma)}, \Lambda_\ast,\Lambda^\ast,\varepsilon,\Gamma,d\big)\,.
\end{align}
\end{prop}

We will also make use of a general $W^{s,p}(\Omega)-$regularity result for the elliptic problem
\begin{align}
 \Delta u=&f\,,\quad\mbox{in }\Omega\,,\label{ell-reg1}\\
  \frac{\partial u}{\partial \mathbf{n}}=&g\,,\quad\mbox{on }\Gamma\,.\label{ell-reg2}
\end{align}

\begin{prop}
\label{preg}
 Assume that $t,r,s\in\mathbb{R}$, $1<p<\infty$, and either $p=2$ or $s-1/p$ is not an integer, and let $f\in W^{r,p}(\Omega)$.
 If $g\in W^{t,p}(\Gamma)$, $r+1\geq 1/p$, and $s=\min\{r+2,t+1+1/p\}$, then any solution $u$ to \eqref{ell-reg1}-\eqref{ell-reg2}
 belongs to $W^{s,p}(\Omega)$. Moreover, if the set of solutions
 is not empty, the following estimate holds
 $$\inf \Vert u\Vert_{W^{s,p}(\Omega)}\leq C(\Vert f\Vert_{W^{r,p}(\Omega)}+\Vert g\Vert_{W^{t,p}(\Gamma)})\,,$$
 where the infimum is taken over all solutions $u$ and $C>0$ is independent of $f$ and $g$.
\end{prop}

Finally, the following notation will turn out to be convenient. If $A$ is a real number, we denote by $A^-$
any constant $B$ arbitrarily close to $A$, such that $B<A$, and by $A^+$ any
constant $B$ arbitrarily close to $A$, such that $B>A$.

\section{Existence of weak solutions}\setcounter{equation}{0}
\label{sec:weakex}
In this section we first introduce the basic assumptions which are needed to prove the existence of a global weak solution.
Then we define the weak formulation and we state and prove the first existence result.

Our assumptions on $\eta$, $J$, $m$, and $F$ read as follows. Note that the ones on $\eta$, $m$, and $F$ are slightly more general than the ones
specified in the Introduction.

\begin{description}
\item[(H1)] $\eta \in C^{0,1}([-1,1])$ and there exists $\eta_1>0$ such that
\begin{equation*}
\eta (s) \geq \eta _{1}\,,\qquad {\color{black}\forall s\in [-1,1]\,.}
\end{equation*}

\item[(H2)] \sergio{$J\in W^{1,1}_{loc}(\mathbb{R}^d)$}, $J(x)=J(-x)$ for almost any $x\in\Omega$, \sergio{and $J$ satisfies
\begin{equation*}
a:=\sup_{x\in \Omega }\int_{\Omega }|J(x-y)|\,dy<\infty \,,\qquad
b:=\sup_{x\in \Omega }\int_{\Omega }|\nabla J(x-y)|\,dy<\infty \,.
\end{equation*}
}
\item[(H3)] $m\in C^{0,1}([-1,1])$ is nonnegative and
$m(s)=0$ if and only if $s=\pm 1$. Moreover, there exists $\sigma_0>0$ such that $m$ is
nonincreasing in $%
[1-\sigma_0,1]$ and nondecreasing in $[-1,-1+\sigma_0]$.

\item[(H4)] $F\in C^{2}((-1,1))$ and $\lambda :=mF^{\prime \,\prime }\in C([-1,1])$.

\item[(H5)] There exists some $\sigma _{0}>0$ such that $%
F^{\,\prime \prime }$ is nondecreasing in $[1-\sigma _{0},1)$ and %
nonincreasing in $(-1,-1+\sigma _{0}]$.

\item[(H6)] There exists some $c_{0}>0$ such that
\begin{equation*}
F^{\,\prime \prime }(s)\geq c_{0}\,,\qquad \forall s\in (-1,1)\,\,.
\end{equation*}

\item[(H7)] There exists some $\alpha _{0}>0$ such that
\begin{equation*}
m(s)F^{\,\prime \prime }(s)\geq \alpha _{0}\,,\qquad \forall s\in \lbrack
-1,1]\,.
\end{equation*}
\end{description}

\begin{oss}
\label{lambdaconst}
An interaction kernel which satisfies (\textbf{H2}) is the Newtonian
kernel, namely $J(x)=j_3|x|^{-1}$ if $d=3$, and
$J(x)=-j_2\ln|x|$, if $d=2$, where $j_2$ and $j_3$ are positive constants.
Moreover, in (\textbf{H4}), $\lambda$
must be understood as continuously extended at the endpoints.
It is worth observing that the assumptions (\textbf{H3})-(\textbf{H7}) are satisfied, for instance,
by \eqref{mixentr} and \eqref{degmob}. Note that, in this case, $\lambda$ is constant.
\end{oss}

The notion of weak solution is defined by

\begin{defn}
\label{wfdef} Let $\varphi _{0}\in L^\infty(\Omega)$ with $F(\varphi _{0})\in L^{1}(\Omega )$ and $0<T<\infty $ be given.
A triplet $[\boldsymbol{u},\pi, \varphi]$ is called weak solution to \eqref{Sy01}-\eqref{Sy04} and \eqref{sy5}-\eqref{sy6} on $[0,T]$ if
\begin{align*}
& \mathbf{u}\in L^2(0,T;G_{div})\,, \\
& \pi\in L^2(0,T;V_0)\,,\\
& \varphi \in L^{\infty }(0,T;L^p(\Omega))\cap L^{2}(0,T;V)\,, \quad \forall\,p\in [2,\infty)\,, \\
& \varphi _{t}\in L^{2}(0,T;V^{\,\prime })\,,\\
& \varphi \in L^{\infty }(Q_{T})\,,\quad |\varphi (x,t)|\leq 1\quad \mbox{ for a.e. }(x,t)\in Q_{T}\,,
\end{align*}
and
\begin{align}
& \eta(\varphi)\mathbf{u}= -\nabla\pi -(J\ast \varphi )\nabla \varphi,\quad\hbox{ a.e. in }\,Q_T \,, \label{E1}\\
& \langle \varphi _{t},\psi \rangle _{V}+ (m(\varphi )F^{\prime\prime }(\varphi )\nabla \varphi,\nabla \psi)
-(m(\varphi )\nabla J\ast \varphi,\nabla \psi)=(\mathbf{u}\varphi
,\nabla \psi ), \nonumber \\
&\qquad\qquad\qquad\qquad\qquad\qquad\qquad\forall\, \psi \in V \hbox{ and a.e. in }\,(0,T)\,, \label{E2}\\
&\varphi (0)=\varphi _{0}, \quad\hbox{ a.e. in }\,\Omega. \label{I1}
\end{align}
\end{defn}

\begin{oss}
Observe that the regularity properties of the weak solution entail that
$\varphi\in C([0,T];H)$. Therefore the initial condition $\varphi(0)=\varphi_0$ makes sense.
Moreover, note that any weak solution is such that the total mass is conserved, namely
$$
\overline{\varphi}(t)=\overline{\varphi }_{0},
$$
for any $t\in [0,T]$.
\end{oss}

\begin{oss}
Looking at \eqref{E2} and recalling (\textbf{H7}), it is clear that the combination of degenerate mobility and singular potential helps since one deals
with a non-local but non-degenerate parabolic equation for $\varphi$ (see also \cite{FGGS}). On the contrary, if the mobility is constant the analysis requires more care (cf. \cite{DPGG}).
\end{oss}

Let $M\in C^{2}((-1,1))$ be the solution to $m(s)M^{\,\prime \prime }(s)=1$ for all $s\in (-1,1)$ with $M(0)=M^{\,\prime }(0)=0$.
Then the existence of a weak solution is given by

\begin{thm}
\label{existweak} Assume that \textbf{(H1)}-\textbf{(H7)} hold.
Let $\varphi _{0}\in L^{\infty }(\Omega )$ be such that $F(\varphi _{0})\in
L^{1}(\Omega )$ and $M(\varphi _{0})\in L^{1}(\Omega )$, where $M$ is defined as above.
Then, for every $T>0$, there exists a weak solution $[\mathbf{u},\pi,\varphi]$
to  \eqref{Sy01}-\eqref{Sy04} and \eqref{sy5}-\eqref{sy6} on $[0,T]$. This
weak solution satisfies the following energy identity
\begin{align}
& \frac{1}{2}\frac{d}{dt}\Vert \varphi
\Vert ^{2} +2\,\Vert \sqrt{\eta (\varphi )}\mathbf{u}\Vert
^{2}+\int_{\Omega }m(\varphi )F^{\,\prime \prime }(\varphi )|\nabla \varphi
|^{2}dx  \notag \\
& =\int_{\Omega }m(\varphi )\nabla J\ast \varphi \cdot \nabla \varphi
dx+\int_{\Omega }(-J\ast \varphi )\,\mathbf{u}\cdot \nabla \varphi
dx\,,  \label{energeq}
\end{align}%
for almost any $t>0$.

\end{thm}

The strategy to prove Theorem \ref{existweak} is the following. We consider suitable approximations of $\eta$, $m$, and $F$. Then
we formulate an approximating problem by adding a viscous term $-\nu\Delta \mathbf{u}$ to the Darcy's law (i.e. we consider its
Brinkman approximation, see \cite{CG} and references therein) for a given $\nu>0$ (see Subsec. \ref{nondeg}).
We solve this problem by means of a Galerkin scheme (see Subsec. \ref{proofexweakapp}). Then, in Subsection \ref{proofexwea}, we get first a weak solution
to the Brinkman-Cahn-Hilliard problem with the original $\eta$, $m$, and $F$. Finally, we pass to the limit as $\nu$ goes to $0$.

\subsection{Approximating problem}
\label{nondeg}

Let $\e \in (0,1]$ be fixed. Consider the following approximations of $\eta$, $m$, and $F$.
\begin{description}
\item[(A1)] Approximating viscosity
\begin{equation*}
\etae(s)=\left\{\begin{array}{lll}
\eta(1-\e),\qquad s\geq 1-\e\\
\eta(s),\qquad|s|\leq 1-\e\\
\eta(-1+\e),\qquad s\leq -1+\e.
\end{array}\right.
\end{equation*}

\item[(A2)] Approximating mobility
\begin{equation*}
\me(s)=\left\{\begin{array}{lll}
m(1-\e),\qquad s\geq 1-\e\\
m(s),\qquad|s|\leq 1-\e\\
m(-1+\e),\qquad s\leq -1+\e.
\end{array}\right.
 \end{equation*}

\item[(A3)] Approximating potential

\begin{equation}
F_\e(s)=\left\{\begin{array}{lll}
F(1-\e)+F^{\prime}(1-\e)\big(s-(1-\e)\big)+\frac{1}{2}F^{\,\prime\prime}(1-\e)\big(s-(1-\e)\big)^2\\[1mm]
+\big(s-(1-\e)\big)^3,\qquad s\geq 1-\e\,,\\[1mm]
F(s)\,,\qquad|s|\leq 1-\e\,,\\[1mm]
F(-1+\e)+F^{\prime}(-1+\e)\big(s-(-1+\e)\big)+\frac{1}{2}F^{\prime\prime}(-1+\e)\big(s-(-1+\e)\big)^2\\[1mm]
+\big|s-(-1+\e)\big|^3,\qquad s\leq -1+\e\,.
\end{array}\right.\nonumber
\end{equation}
\end{description}

It is easy to check that $F_\e\in C^{2,1}_{loc}(\mathbb{R})$ and that,
thanks also to \textbf{(H6)},
there exist two constants $k_1>0$ and $k_2\geq 0$, which do not depend on $\e$, such that
\begin{align}
&F_\e(s)\geq k_1|s|^3-k_2\,,\qquad\forall s\in\mathbb{R}\,.\label{auxlem1}
\end{align}
Moreover, as a consequence of \textbf{(H6)}, we still have that
\begin{align}
&F_\e''(s)\geq c_0\,,\qquad\forall s\in\mathbb{R}\,,
\label{auxlem2}
\end{align}
and \textbf{(H5)} implies that there exists $\e_0>0$ such that
\begin{align}
&F_{\e}(s)\leq F(s)+\,\e^3\,,\qquad\forall s\in (-1,1)\,,\quad\forall\e\in(0,\e_0]\,.\label{boundF}
\end{align}
Also, note that
\begin{align}
&\vert F^\prime_{\e}(s)\vert \leq k_3 s^2 + k_4\,,\qquad\forall s\in \mathbb{R}\,,\label{boundFprime}
\end{align}
for some positive constants $k_3,k_4$.

Fix $\nu>0$ and consider the Brinkman approximation
\begin{align}
& - \nu\Delta \mathbf{u} + \etae(\varphi)\mathbf{u} + \nabla\pi=\mu\nabla\varphi, \qquad \hbox{in }Q_T,\label{Sy01a}\\
& \mbox{div}(\mathbf{u})=0,\qquad \hbox{in }Q_T,\label{Sy02a}\\
& \varphi_t+u\cdot\nabla\varphi=\mbox{div}(\me(\varphi)\nabla\mu),\qquad \hbox{in }Q_T,\label{Sy03a} \\
& \mu=-J\ast\varphi + F_{\e}^\prime(\varphi),\qquad \hbox{in }Q_T, \label{Sy04a}\\
&\mathbf{u}= \mathbf{0}, \quad \me(\varphi)\frac{\partial\mu}{\partial \mathbf{n}}=0, \qquad\mbox{on }
\Gamma\times (0,T),\label{sy5a}\\
& \varphi(0)=\varphi_0, \qquad\mbox{in }\Omega. \label{sy6a}
\end{align}

Then we introduce the notion of weak solution which reads

\begin{defn}
\label{wfdfa}
Let $\varphi_0\in H$ be such that $F_\e(\varphi_0)\in L^1(\Omega)$ and $0<T<\infty$ be given. Then
$[u,\pi, \varphi]$ is a weak solution to \eqref{Sy01a}-\eqref{sy6a} on $[0,T]$ if
\begin{align*}
& \mathbf{u}\in L^2(0,T;V_{0,div})\,, \\
& \pi\in L^2(0,T;H_0)\,,\\
& \varphi \in L^{\infty }(0,T;L^3(\Omega))\cap L^{2}(0,T;V)\,, \\
& \varphi_{t}\in L^{2}(0,T;V^\prime)\,,\\
& \mu \in L^2(0,T;V)\,,
\end{align*}
and
\begin{align*}
& (\nu \nabla\mathbf{u},\nabla\mathbf{v})+(\etae(\varphi)\mathbf{u},\mathbf{v})= (\pi, \mbox{{\rm div}}(\mathbf{v})) - (\varphi\nabla \mu,\mathbf{v}),
\quad\forall\,\mathbf{v}\in C^\infty_c(\Omega)^d\,, \\
& \langle \varphi _{t},\psi \rangle _{V}+ (m_\e(\varphi )\nabla\mu, \nabla \psi)
=(\mathbf{u}\varphi,\nabla \psi ), \quad\forall\, \psi \in V \hbox{ and a.a. }\, t\in(0,T)\,,\\
&\mu = - J*\varphi + F^\prime_\e(\varphi)\,,
\end{align*}
almost everywhere in $(0,T)$ with
\begin{equation*}
\varphi (0)=\varphi _{0}, \quad\hbox{ a.e. in }\,\Omega.
\end{equation*}
\end{defn}

We now prove the following

\begin{thm}
\label{existweakapp} Assume that \textbf{(A1)}-\textbf{(A3)} hold.
Let $\varphi _{0}\in L^{\infty }(\Omega )$ with $F_\e(\varphi _{0})\in L^{1}(\Omega )$.
Then, for every $T>0$, there exists a weak solution $[\mathbf{u},\pi,\varphi]$
to  \eqref{Sy01a}-\eqref{sy6a} on $[0,T]$. Moreover, this solution satisfies the energy inequality
\begin{equation}
\mathcal{E}_\e(\varphi(t)) + \int_0^t \left(\|\sqrt{m_\e(\varphi(\tau))}\nabla\mu(\tau)\|^2 + \nu\|\nabla \mathbf{u}(\tau) \|^2
+ \Vert \sqrt{\etae(\psi(\tau))} \mathbf{u}(\tau)\Vert^2\right) d\tau \leq \mathcal{E}_\e(\varphi_0)\,, \label{eappr2}
\end{equation}
for all $t\in [0,T]$, where
\begin{align}
\mathcal{E}_\e(\varphi(t))=-\frac{1}{2} \int_{\Omega}\int_{\Omega}J(x-y)\varphi(x,t)\varphi(y,t) dxdy+\int_{\Omega}F_\e(\varphi(t))dx.
\label{endef}
\end{align}
\end{thm}

\begin{oss}
Actually, it can be proven that \eqref{eappr2} is an identity (cf. \cite[Proof of Cor.2]{FGR}).
\end{oss}

\subsection{Proof of Theorem \ref{existweakapp}}
\label{proofexweakapp}
We shall use a Galerkin approximation scheme. Let $\{\lambda_j\}_{j\in \mathbb{N}}$ and $\{w_j\}_{j\in\mathbb{N}}$ be the eigenvalues and the eigenvectors of the Laplace operator with homogeneous Neumann boundary conditions (cf. Section \ref{sec:notation}). Then, set
\begin{equation*}
W_n := \langle w_1,\dots,w_n\rangle
\end{equation*}
and denote by $\Pi_n: H \to W_n$ the usual linear bounded orthogonal projector.

Fix $n\in \mathbb{N}$ and introduce
\begin{equation*}
\varphi_n(t) = \sum_{j=1}^n g_{nj}(t)w_j,
\end{equation*}
where $\mathbf{g}_n=(g_{n1}, \dots, g_{nn})$ has to be determined. For any $\mathbf{g}_n \in C([0,T])^n$ there exists a unique $\mathbf{w}\in C([0,T];V_{0,div})$
(depending on $\mathbf{g}_n$) which solves
\begin{equation}
(\nu \nabla\mathbf{w},\nabla\mathbf{v})+(\etae(\varphi_n)\mathbf{w},\mathbf{v})= -(\varphi_n\nabla \mu_n,\mathbf{v}),
\quad\forall\, \mathbf{v}\in V_{0,div}, \,  \hbox{ in }\, (0,T), \label{Brinkappr}
\end{equation}
where
\begin{equation}
\mu_n = \Pi_n(-J*\varphi_n + F^\prime_\e(\varphi_n))\,. \label{chemappr}
\end{equation}
Moreover, it is easy to prove that the mapping $\mathbf{F}: C([0,T])^n \to C([0,T];V_{0,div})$, defined by setting $\mathbf{F}(\mathbf{g}_n) = \mathbf{w}$,
is continuous with respect to the standard Lagrangian norm in $C([0,T])^n$. Let us suppose
\begin{equation}
\varphi_0 \in D(A) \label{regic}
\end{equation}
and consider the following problem: find $\mathbf{g}_n$ solution to
the Cauchy problem
\begin{align}
&\int_\Omega\varphi_n^\prime \psi dx+\int_{\Omega }m_\e(\varphi_n)\nabla\mu_n\cdot \nabla \psi dx
=\int_\Omega \mathbf{F}(\mathbf{g}_n)\varphi_n \cdot\nabla \psi dx, \nonumber \\
&\qquad\qquad\qquad\qquad\qquad\qquad\qquad\forall\, \psi \in W_n, \hbox{ in }\, (0,T)\,, \label{CHappr}\\
&\varphi (0)=\varphi _{0n}:=\Pi_n\varphi_{0} \label{icappr}.
\end{align}
Taking $\psi = w_j$ for each $j\in\{1,\dots,n\}$ and using the orthogonality properties of the eigenfunctions, we can write down a first-order system
of ODEs in normal form for the unknown vector-valued function $\mathbf{g}$ with a locally Lipschitz continuous right-hand side. Therefore, the Cauchy-Lipschitz
theorem entails the existence of a unique solution $\mathbf{g}\in C^1([0,T_n])^n$ for some $T_n\in (0,T]$. We thus have found a unique approximating pair $(\mathbf{w}_n,\varphi_n) \in C([0,T_n];V_{0,div}) \times C^1([0,T_n];W_n)$ for each $n\in \mathbb{N}$.

Take now $\mathbf{v}=\mathbf{w}$ in \eqref{Brinkappr} and $\psi=\mu_n$ in \eqref{CHappr}. Adding together the identities and taking \eqref{chemappr} into account,
it is not difficult to obtain the energy identity
\begin{equation}
\frac{d}{dt} \mathcal{E}_\e(\varphi_n(t)) + \|\sqrt{m_\e(\varphi_n(t))}\nabla\mu_n(t)\|^2 + \nu\|\nabla \mathbf{w}_n(t) \|^2
+ \Vert \sqrt{\etae(\psi_n(t))} \mathbf{w}_n(t)\Vert^2 = 0,\label{eappr}
\end{equation}
for every $t\in [0,T_n]$.

Integrating \eqref{eappr} with respect to time in $(0,t)$ and recalling  \textbf{(H1)} and \textbf{(A1)}-\textbf{(A3)}, we get
\begin{equation}
\mathcal{E}_e(\varphi_n(t)) + \int_0^t \left(C\|\nabla\mu_n(\tau) \|^2 + \nu\|\nabla \mathbf{w}_n(\tau) \|^2
+ \eta_1\Vert \mathbf{w}_n(\tau)\Vert^2 \right)d\tau \leq \mathcal{E}_\e(\varphi_{0n}).\label{diseappr}
\end{equation}
Here and in the sequel of this proof $C>0$ stands for a generic constant which possibly depends on $\nu$ and $\e$ but is independent of $n$.
This constant may vary also within the same line.

Due to the convergence $\varphi_{0n} \to \varphi_0$ in $H^2(\Omega)$ (cf. \eqref{regic}), on account of \textbf{(A3)} and \eqref{endef}, we find
\begin{equation}
\mathcal{E}_e(\varphi_n(t)) + \int_0^t \left(C\|\nabla\mu_n(\tau) \|^2 + \nu\|\nabla \mathbf{w}_n(\tau) \|^2
+ \eta_1\Vert \mathbf{w}_n(\tau)\Vert^2 \right)d\tau \leq C.\label{diseapprunif}
\end{equation}
From the above estimate we deduce first that we can extend our approximating solution up to $T$ for each $n\in \mathbb{N}_0$.
Indeed we have $\vert \mathbf{g}_n(t)\vert_2 = \Vert \varphi_n(t)\Vert$, $\vert \cdot \vert_2$ being the Euclidean norm. Moreover, using \eqref{auxlem1} and arguing as in \cite{FGR}, we obtain the following uniform estimates
\begin{align}
&\|\mathbf{w}_n\|_{L^{2}(0,T;V_{0,div})} \leq C,\label{est1}\\
&\|\varphi_n\|_{L^{\infty}(0,T;L^3(\Omega))}\leq C,\label{est2}\\
&\| F_\e(\varphi_n)\|_{L^{\infty}(0,T;L^1(\Omega))}\leq C,\label{est3}\\
&\|\nabla\mu_n\|_{L^2(0,T;H^d)}\leq C. \label{est4}
\end{align}
Observe now that
\begin{align*}
\overline{\mu} &= \vert\Omega\vert^{-1}(\Pi_n (-J*\varphi_n + F_\e^\prime(\varphi_n)),1) \\
&= (\Pi_n (-J*\varphi_n + F_\e^\prime(\varphi_n)), w_1) =  (-J*\varphi_n + F_\e^\prime(\varphi_n),w_1).
\end{align*}
This gives a uniform control of $\Vert \overline{\mu}\Vert_{L^\infty(0,T)}$. Hence, using Poincar\'{e}-Wirtinger inequality and recalling \eqref{boundFprime} and \eqref{est4}, we deduce
\begin{equation}
\|\mu_n\|_{L^2(0,T;V)}\leq C.\label{est5}
\end{equation}
Using the above estimates, by comparison in equation \eqref{CHappr} we deduce
\begin{equation}
\|\varphi^\prime_n\|_{L^{2}(0,T;V^\prime)}\leq C.\label{est6}
\end{equation}
On the other hand, multiplying \eqref{chemappr} by $-\Delta\varphi_n$, recalling (\textbf{H2}) and \eqref{auxlem2}, using \eqref{est4}, and arguing as in \cite{CFG} we find
\begin{equation}
\|\varphi_n\|_{L^2(0,T;V)}\leq C.\label{est7}
\end{equation}
In addition, \eqref{boundFprime} and \eqref{est2} yield
\begin{equation}
\|F^\prime(\varphi_n)\|_{L^\infty(0,T;L^{3/2}(\Omega))}\leq C.\label{est8}
\end{equation}

The above uniform bounds and a standard compactness result in vector valued Banach spaces imply the existence of a triplet $(\mathbf{u},\varphi,\mu)$ and a subsequence
$(\mathbf{w}_n,\varphi_n,\mu_n)$ (not relabeled) such that
\begin{align}
& \mathbf{w}_n\rightharpoonup \mathbf{u} \quad\mbox{weakly in }L^2(0,T;V_{0,div}),\label{con1}\\
& \varphi_n\rightharpoonup\varphi\quad\mbox{weakly}^{\ast}\mbox{ in }L^{\infty}(0,T;L^3(\Omega)),
\; \mbox{weakly in }L^2(0,T;V),\label{con2}\\
& \varphi_n\to\varphi\quad\mbox{strongly in }L^2(0,T;H),\quad\mbox{a.e. in }Q_T,\label{con3}\\
& \varphi_n'\rightharpoonup\varphi_t\quad\mbox{weakly in }L^{2}(0,T;V'), \label{con4}\\
& F'_\e(\varphi_n)\rightharpoonup \Phi_\e\quad\mbox{weakly}^{\ast}\mbox{ in } L^{\infty}(0,T;L^{3/2}(\Omega)),
\;\mbox{weakly in }L^2(0,T;L^3(\Omega)),\label{con5}\\
& \mu_n\rightharpoonup\mu\quad\mbox{weakly in }L^2(0,T;V).\label{con6}
\end{align}
Moreover, using a well-known result, from \eqref{con3} and \eqref{con5} we infer that $\Phi_\e=F^\prime_e(\varphi)$.

Recalling \eqref{Brinkappr}, observe that $\mathbf{w}_n$ satisfies, in particular, the following variational identity
\begin{equation}
(\nu \nabla\mathbf{w}_n,\nabla\mathbf{v})+(\etae(\varphi_n)\mathbf{w}_n,\mathbf{v})= -(\varphi_n\nabla \mu_n,\mathbf{v}),
\qquad\forall\, \mathbf{v}\in\mathcal{V},   \hbox{ in }\, (0,T). \label{Brinkappr2}
\end{equation}
The above convergences and \textbf{(A1)}-\textbf{(A2)} allow us to pass to the limit in equations
\eqref{chemappr}, \eqref{CHappr}-\eqref{icappr}, and \eqref{Brinkappr2}. Using a density argument, this gives
\begin{align}
&(\nu \nabla\mathbf{u},\nabla\mathbf{v})+(\etae(\varphi)\mathbf{u},\mathbf{v})= -(\varphi\nabla \mu,\mathbf{v}),
\qquad\forall\, \mathbf{v}\in\mathcal{V},   \hbox{ a.e. in }\, (0,T), \label{Brinkappr3}\\
&\langle \varphi^\prime, \psi \rangle + (m_\e(\varphi)\nabla\mu,\nabla \psi)
=(\mathbf{u}\varphi, \nabla \psi), \,\,\,\qquad\forall\, \psi \in V, \hbox{ a.e. in }\, (0,T)\,, \label{CHappr2}\\
&\mu = -J*\varphi + F^\prime_\e(\varphi), \qquad\hbox{ a.e. in }\, Q_T\,, \label{chemappr2}\\
&\varphi (0)=\varphi _{0}, \qquad\hbox{ a.e. in}\, \Omega. \label{icappr2}
\end{align}
Also, using a semicontinuity argument (see \cite{FGR}), we can prove that $(\mathbf{u},\varphi)$ satisfies the energy estimate \eqref{eappr2}.

If $\varphi_0\in H$ with $F(\varphi_0)\in L^1(\Omega)$ then we can argue as in \cite{CFG} (see also \cite{FGR}).
We first approximate $\varphi_0$ with $\varphi_{0m}\in D(A)$ given by
$\varphi_{0m}:=(I+\frac{A}{m})^{-1}\varphi_0$. This sequence satisfies $\varphi_{0m}\to\varphi_0$ in $H$. The
corresponding approximating solutions $(\mathbf{u}_m,\varphi_m)$ satisfy the energy estimate \eqref{eappr2}
with $\varphi_0=\varphi_{0m}$. On the other hand, on account of \eqref{auxlem2}, we can use a convexity argument (cf. \cite{CFG}) to deduce
\begin{equation*}
\mathcal{E}_\e(\varphi_{0m}) \leq \mathcal{E}_\e(\varphi_0).
\end{equation*}
Hence the sequence $\{(\mathbf{u}_m,\varphi_m)\}$ satisfies \eqref{eappr2} with $\varphi_0$ in place of $\varphi_{0m}$.
Then, arguing as above, we can find that it converges, up to a subsequence, to a pair $\{(\mathbf{u},\varphi)\}$
satisfying \eqref{Brinkappr3}-\eqref{icappr2}.
Finally, owing to De Rham's theorem (see, e.g., \cite[Chap.IV, Sec.2]{BF}), we can find a unique $\pi\in L^2(0,T;H_0)$ such that
\begin{equation*}
(\nu \nabla\mathbf{u},\nabla\mathbf{v})+(\etae(\varphi)\mathbf{u},\mathbf{v})= -(\pi_0,\mbox{div}(\mathbf{v}))
-(\varphi\nabla \mu,\mathbf{v}), \quad\forall\, \mathbf{v}\in C^\infty_c(\Omega)^d,   \hbox{ a.e. in }\, (0,T).
\end{equation*}
This concludes the proof.

\subsection{Proof of Theorem \ref{existweak}}
\label{proofexwea}
Following a strategy devised in \cite{DPGG}, we will pass to the limit first as $\epsilon$ goes to $0$. Then we will let $\nu\to 0$.
Thus, let us consider first a weak solution $[\mathbf{u}_\epsilon, \varphi_\epsilon]$ to \eqref{Sy01a}-\eqref{sy6a}, keeping $\nu>0$ fixed and find
suitable uniform estimates. We omit the dependence on $\nu$ for the sake of simplicity.

From \eqref{eappr2}, recalling \textbf{(A3)} and \eqref{auxlem1}, we deduce the uniform bounds
\begin{align}
&\sqrt{\nu}\Vert \mathbf{u}_\e \Vert_{L^2(0,T;V_{0,div})}  +  \|\sqrt{\etae(\varphi_\e)}\mathbf{u}_\e\|_{L^{2}(0,T;G_{div})}\leq C\,,\label{bdd1}\\
&\|\phie\|_{L^{\infty}(0,T;L^3(\Omega))}\leq C\,,\label{bdd2}\\
&\| \sqrt{\me(\phie)}\nabla\mue\|_{L^2(0,T;H^d)}\leq C\,.\label{bdd3}
\end{align}
Here and in the sequel of this proof, $C>0$ indicates a generic constant which is independent of $\e$ and $\nu$.

Arguing as in \cite{FGR}, we now test equation \eqref{CHappr2} by $\psi=M_\e'(\phie)$, where $M_\e$ is a $C^2$
function such that $\me(s)M_\e''(s)=1$ and $M_\e(0)=M_\e'(0)=0$. This gives
\begin{equation}
\frac{d}{dt}\int_\Omega
M_\e(\phie) dx +\int_\Omega\me(\phie)\nabla\mue\cdot
M_\e''(\phie)\nabla\phie dx =\int_\Omega\mathbf{u}_\e\phie\cdot
\nabla M_\e'(\phie)dx. \label{entrop}
\end{equation}
Being $\mathbf{u}_\e$ divergence free, we have
\begin{equation*}
\int_\Omega\mathbf{u}_\e\phie\cdot \nabla M_\e'(\phie)dx = - \int_\Omega\mathbf{u}_\e\nabla \phie\cdot M_\e'(\phie)dx =
 - \int_\Omega\mathbf{u}_\e\nabla M_\e(\phie)dx =0\,.
\end{equation*}
Therefore \eqref{entrop} yields
\begin{equation*}
\frac{d}{dt}\int_\Omega M_\e(\phie)dx +\int_\Omega\nabla\mue\cdot \nabla\phie dx=0\,,
\end{equation*}
that is
\begin{equation}
\frac{d}{dt}\int_\Omega M_\e(\phie)dx +
\int_\Omega\left(\Fe''(\phie)|\nabla\phie|^2-(\nabla J\ast\phie)\cdot\nabla\phie\right)dx=0\,. \nonumber
\end{equation}
Thus, on account of \eqref{auxlem2}, we get
\begin{equation}
\frac{d}{dt}\int_\Omega M_\e(\phie) dx +\frac{c_0}{2}\Vert\nabla\phie\Vert^2\leq C\,.\nonumber
\end{equation}
Arguing as in \cite{FGR}, for $\e$ small enough we have a uniform control of $\int_\Omega M_\e(\phie)dx$ with $\int_\Omega M(\varphi_0)dx$.
Hence, recalling \eqref{bdd2}, we get the uniform bounds
\begin{align}
&\Vert\phie\Vert_{L^2(0,T;V)}\leq C\,,\label{bdd4}\\
&\Vert M_\e(\phie)\Vert_{L^\infty(0,T;L^1(\Omega))}\leq
C\,.\label{bdd5}
\end{align}
Then, on account of \eqref{bdd1}-\eqref{bdd3}, by comparison in \eqref{CHappr2} we also obtain
\begin{equation}
\Vert(\phie)_t\Vert_{L^2(0,T;V^\prime)}\leq C\,. \label{bdd6}
\end{equation}

Bounds \eqref{bdd1}-\eqref{bdd6} and standard compactness results
entail the existence of $\mathbf{u}\in  L^2(0,T;V_{div})$ and
$\varphi\in L^{\infty}(0,T;H)\cap L^2(0,T;V)$ such that, for some sequence $\e_n \to 0$, we have
\begin{align}
& \mathbf{u}_{\epsilon_n}\rightharpoonup \mathbf{u} \quad \mbox{weakly in }L^2(0,T;V_{0,div}),\label{conve1}\\
& \varphi_{\epsilon_n}\rightharpoonup\varphi\quad\mbox{weakly}^{\ast}\mbox{ in }L^{\infty}(0,T;L^3(\Omega))
\;\mbox{weakly in }L^2(0,T;V),\label{conve2}\\
& \varphi_{\epsilon_n}\to\varphi \quad\mbox{strongly in }L^2(0,T;L^s(\Omega)),\; s\in [2,6),\;\mbox{and a.e. in }Q_T,\label{conve3}\\
& (\varphi_t)_{\epsilon_n} \rightharpoonup\varphi_t\quad \mbox{weakly in }L^{2}(0,T;V'). \label{conve4}
\end{align}
In order to show that $|\varphi|\leq 1$ almost everywhere in $Q_T$ we can argue as in \cite[Sec.7]{FGR}.

Observe now that the weak formulation \eqref{Brinkappr3}-\eqref{chemappr2} can be rewritten as follows
\begin{align}
&(\nu \nabla\mathbf{u}_\e,\nabla\mathbf{v})+(\etae(\varphi_\e)\mathbf{u}_\e,\mathbf{v})= (\varphi(\nabla J \ast \varphi),\mathbf{v}), \nonumber\\
&\qquad\qquad\qquad\qquad\qquad\qquad\qquad\forall\, \mathbf{v}\in\mathcal{V},   \hbox{ a.e. in }\, (0,T), \label{Brinkappr4}\\
&\langle(\phi_t)_\e,\psi\rangle+( \me(\phie)\Fe''(\phie)\nabla\phie,\nabla\psi)
-(\me(\phie)(\nabla J\ast\phie),\nabla\psi)=(\mathbf{u}_\e\phie,\nabla\psi), \nonumber\\
&\qquad\qquad\qquad\qquad\qquad\qquad\qquad\forall\, \psi \in V, \hbox{ a.e. in }\, (0,T)\,. \label{CHappr3}
 \end{align}
Recalling (\textbf{H2}), (\textbf{A3}), and \eqref{conve3}, up to a subsequence, we obtain (cf. also the essential boundedness of $\varphi$)
\begin{equation}
m_{\e_n}(\varphi_{\e_n})F_{\e_n}''(\varphi_{\e_n})\to m(\varphi)F''(\varphi),\qquad\mbox{ strongly in }\:\: L^s(Q_T),\quad\forall\,s\in[2,\infty)\,.
\label{conve5}
\end{equation}
Using now \eqref{conve2}, \eqref{conve3}, and the embedding
$L^\infty(0,T;L^3(\Omega))\cap L^2(0,T;L^6(\Omega))\hookrightarrow L^{4}(Q_T)$, we deduce
\begin{equation}
\varphi_{\e_n}\to\varphi\qquad\mbox{strongly in }L^r(Q_T),\quad\forall\,r\in [2,4\,).
\label{conve6}
\end{equation}
In addition, thanks to Lebesgue's dominated convergence theorem, we find
\begin{equation}
\eta_{\e_n}(\varphi_{\e_n})\to \eta(\varphi),\quad m_{\e_n}(\varphi_{\e_n})\to m(\varphi) \qquad\mbox{ strongly in }\:\: L^s(Q_T),\quad\forall\,s\in[2,\infty)\,.
\label{conve7}
\end{equation}

Convergences \eqref{conve1}-\eqref{conve4} and \eqref{conve5}-\eqref{conve7} allow us to pass to the limit
in \eqref{Brinkappr4}-\eqref{CHappr3} for $\e=\e_n$ as $n$ goes to $\infty$
(cf. \cite{FGR}). Also, by integrating \eqref{CHappr3} in time over $(0,t)$ and then taking the limit as before,
we recover the initial condition $\varphi(0)=\varphi_0$ almost everywhere in $\Omega$.

Summing up, for any $\nu>0$, there is a pair $[\mathbf{u}_\nu,\varphi_\nu]$ such that
\begin{align}
& \mathbf{u}_\nu\in L^2(0,T;V_{0,div})\,, \nonumber\\
& \varphi_\nu \in L^{\infty }(0,T;L^p(\Omega))\cap L^{2}(0,T;V) \qquad \forall\,p\in[2,\infty)\,, \nonumber \\
& \varphi_\nu \in L^{\infty }(Q_{T})\,,\quad |\varphi_\nu (x,t)|\leq 1\quad \mbox{ for a.e. }(x,t)\in Q_{T}\,,\nonumber \\
& (\varphi_\nu)_{t}\in L^{2}(0,T;V^\prime)\,, \nonumber\\
&(\nu \nabla\mathbf{u}_\nu,\nabla\mathbf{v})+(\eta(\varphi)\mathbf{u}_\nu,\mathbf{v})= (\varphi_\nu(\nabla J \ast \varphi_\nu),\mathbf{v}), \nonumber\\
&\qquad\qquad\qquad\qquad\qquad\qquad\qquad\forall\, \mathbf{v}\in V_{0,div},   \hbox{ a.e. in }\, (0,T), \label{Brinkappr5} \\
&\langle(\varphi_\nu)_t,\psi\rangle+( m(\varphi_\nu)F''(\varphi_\nu)\nabla\varphi_\nu,\nabla\psi)
-(m(\varphi_\nu)(\nabla J\ast\varphi_\nu),\nabla\psi)=(\mathbf{u}_\nu\varphi_\nu,\nabla\psi), \nonumber\\
&\qquad\qquad\qquad\qquad\qquad\qquad\qquad\forall\, \psi \in V, \hbox{ a.e. in }\, (0,T)\,, \label{CHappr4}\\
&\varphi_\nu (0)=\varphi _{0}, \quad\hbox{ a.e. in }\,\Omega. \nonumber
\end{align}
From \eqref{bdd1}, using a semicontinuity argument, we get the uniform (with respect to $\nu$) bound
\begin{equation}
\sqrt{\nu}\Vert \mathbf{u}_\nu \Vert_{L^2(0,T;V_{0,div})}  +  \|\mathbf{u}_\nu\|_{L^{2}(0,T;G_{div})}\leq C\,.\label{bdd7}
\end{equation}

Recalling \cite[Rem.3.3]{GGG}, we have that $t\mapsto \Vert \varphi_\nu(t)\Vert_{L^\infty(\Omega)}$ is measurable, essentially bounded, and such that
\begin{equation*}
\vert(\varphi_\nu(t),f)\vert \leq \Vert f(t) \Vert_{L^1(\Omega)}, \qquad\mbox{ for a.a. } t\in (0,T),
\end{equation*}
for any $f\in L^1(0,T;L^1(\Omega))$. Therefore, we have
\begin{equation*}
\vert (\mathbf{u}_\nu(t)\varphi_\nu(t),\nabla\psi)\vert = \vert (\varphi_\nu(t),\mathbf{u}_\nu(t)\cdot\nabla\psi)\vert \leq \Vert \mathbf{u}_\nu(t)\cdot\nabla\psi\Vert_{L^1(\Omega)}
\leq \Vert \mathbf{u}_\nu(t)\Vert \Vert \nabla \psi\Vert,
\end{equation*}
for almost any $t\in(0,T)$. This entails that $\mathbf{u}_\nu\varphi_\nu \in L^2(0,T;(V^d)^\prime)$ and
\begin{equation}
\Vert\mathbf{u}_\nu\varphi_\nu\Vert_{L^2(0,T;(V^d)^\prime)} \leq C\,. \label{bdd8}
\end{equation}
We can now take $\psi=\varphi(t)$ in \eqref{CHappr4}. Thanks to \eqref{bdd8} and (\textbf{H7}), we can easily find a bound
\begin{equation}
\Vert\varphi_\nu\Vert_{L^2(0,T;V)} \leq C\,. \label{bdd9}
\end{equation}
Then, by comparison, we also get
\begin{equation}
\Vert(\varphi_\nu)_t\Vert_{L^2(0,T;V^\prime)} \leq C\,. \label{bdd10}
\end{equation}
In addition, we have
\begin{equation}
\Vert\varphi_\nu\Vert_{L^\infty(0,T;L^p(\Omega))} \leq \vert \Omega \vert^{1/p}, \quad \forall\, p\in [2,\infty)\,. \label{bdd11}
\end{equation}
On account of\eqref{bdd7}-\eqref{bdd11} and well-known compactness results, we can find a pair $[\mathbf{u},\varphi]$
and a sequence $\nu_n \to 0$ as $n$ goes to $\infty$ such that
\begin{align}
&\nu_n \mathbf{u}_{\nu_n} \to 0  \quad \mbox{strongly in }L^2(0,T;V_{0,div}),\label{conv1}\\
& \mathbf{u}_{\nu_n}\rightharpoonup \mathbf{u}\quad \mbox{weakly in }L^2(0,T;G_{div}),\label{conv2}\\
& \phi_{\nu_n}\rightharpoonup\varphi\quad\mbox{weakly}^{\ast}\mbox{ in }L^{\infty}(0,T;L^p(\Omega))\cap L^\infty(Q_T),
\;\mbox{weakly in }L^2(0,T;V),\label{conv3}\\
& \varphi_{\nu_n}\to\varphi\quad\mbox{strongly in }L^3(Q_T),\;\mbox{and a.e. in }Q_T,\label{conv4}\\
& (\varphi_t)_{\nu_n} \rightharpoonup \varphi_t\quad \mbox{weakly in }L^{2}(0,T;V'). \label{conv5}
\end{align}
Arguing as above (cf. \eqref{conve7}) and using \eqref{conv1}-\eqref{conv2}, by means of standard techniques, we can pass to the limit in \eqref{Brinkappr5} and find
\begin{equation}
(\eta(\varphi)\mathbf{u},\mathbf{v})= (\varphi(\nabla J \ast \varphi),\mathbf{v}), \quad \forall\, \mathbf{v}\in V_{0,div},
\hbox{ a.e. in }\, (0,T)\,,
\end{equation}
which can be rewritten as
\begin{equation}
(\eta(\varphi)\mathbf{u},\mathbf{v})= -((J \ast \varphi) \nabla\varphi,\mathbf{v}), \quad \forall\, \mathbf{v}\in V_{0,div},
\hbox{ a.e. in }\, (0,T)\,.
\end{equation}
Then, using density and De Rham's theorem, we find a unique $\pi\in L^2(0,T;V_0)$ such that \eqref{E1} holds.

In order to pass to the limit in equation  \eqref{CHappr4}, observe first that (cf. \eqref{conv2} and \eqref{conv4})
\begin{equation}
\int_0^T (\mathbf{u}_{\nu_n}(\tau), \varphi_{\nu_n}(\tau)\nabla\psi)d\tau \to \int_0^T (\mathbf{u}(\tau), \varphi(\tau)\nabla\psi)d\tau, \quad \forall \psi\in D(A).
\label{conv6}
\end{equation}
On account of \eqref{conv3}-\eqref{conv6} and recalling \eqref{conve5}, \eqref{conve7} (which now hold with respect to $\nu_n$), standard techniques give
\begin{equation}
\langle\varphi_t,\psi\rangle+( m(\varphi)F''(\varphi)\nabla\varphi,\nabla\psi)
-(m(\varphi)(\nabla J\ast\varphi),\nabla\psi)=(\mathbf{u}\varphi,\nabla\psi),
\end{equation}
for all $\psi \in D(A)$ and almost everywhere in $(0,T)$. Thus equation \eqref{E2} holds thanks to the density of $D(A)$ in $V$. Initial
condition \eqref{I1} can be recovered as usual. Summing up, we have proven that problem \eqref{Sy01}-\eqref{Sy04} and \eqref{sy5}-\eqref{sy6} has a weak solution
$[\mathbf{u},\pi,\varphi]$ in the sense of Definition \ref{wfdef}.

\section{Existence of strong solutions}
\setcounter{equation}{0}
\label{sec:strongex}

In this section we state and prove the existence of strong solutions to \eqref{Sy01}-\eqref{Sy04}, \eqref{sy5}-\eqref{sy6}. However
equations \eqref{Sy03}-\eqref{Sy04} need to be suitably rewritten in the form
\begin{equation}
\varphi _{t}+\uvec\cdot \nabla \varphi =\Delta B(\varphi )-%
\mbox{div}\big(m(\varphi )(\nabla J\ast \varphi )\big)\,,\label{e25}
\end{equation}%
where we have set
\begin{equation}
\label{primitive}
B(s)=\int_{0}^{s}\lambda (\sigma )d\sigma \,,\qquad \forall s\in \lbrack -1,1]\,.
\end{equation}%
Notice that we have $\nabla B(\varphi )=\lambda (\varphi )\nabla \varphi \,$.
Hence, the boundary condition $m(\varphi )\nabla \mu \cdot \boldsymbol{n}=0$
becomes
\begin{equation}
\big[\nabla B(\varphi)-m(\varphi )(\nabla J\ast \varphi )%
\big]\cdot \nvec=0\,.  \label{BCbis}
\end{equation}%
Thus, the equivalent weak formulation \eqref{E2} of equations \eqref{Sy03}-\eqref{Sy04}
is
\begin{equation}
\label{e174}
\langle \varphi _{t},\psi \rangle _{V}+\big(\nabla
B(\varphi), \nabla \psi\big)-\big(m(\varphi )(\nabla J\ast
\varphi ),\nabla \psi\big)=(\uvec\,\varphi,\nabla \psi )\,,
\end{equation}
for every $\psi \in V$ and for almost any $t\in (0,T)$. Moreover, we rewrite the Darcy's law
\eqref{E1} in the form
\begin{equation}
\eta(\varphi)\mathbf{u} + \nabla\pi=(\nabla J\ast\varphi)\varphi\,.\label{e26}
\end{equation}
Therefore, we can give our definition of strong solution
\begin{defn}
Let $\varphi _{0}\in V\cap L^{\infty }(\Omega )$ and $0<T<\infty$ be given. A weak solution
$[\uvec,\pi,\varphi]$ to \eqref{Sy01}-\eqref{Sy04}, \eqref{sy5}, \eqref{sy6} on $[0,T]$ corresponding to $\varphi_0$
is called strong solution if
\begin{align}
& \uvec\in L^2(0,T;V_{div})\,, \\
& \pi\in L^2(0,T;H^2(\Omega)\cap V_0)\,,\\
&\varphi \in L^{\infty }(0,T;V)\cap L^{2}(0,T;H^{2}(\Omega ))\cap H^{1}(0,T;H)\,,
\end{align}
and if \eqref{Sy02}, \eqref{e25},  \eqref{e26} hold almost everywhere in $Q_T$, and \eqref{sy5}$_1$, \eqref{BCbis}
hold almost everywhere on $\Gamma\times(0,T)$.
\end{defn}

In order to establish regularity results, we shall need the kernel $J$
to be more regular. For instance, we could suppose $J\in
W_{loc}^{2,1}(\mathbb{R}^{d})$. However, this assumption excludes, for instance,
Newtonian and Bessel potential kernels which are physically relevant.
Thus, in order to include them, we recall the definition of admissibile kernel
(see \cite[Definition 1]{BRB}).


\begin{defn}
A kernel $J\in W_{loc}^{1,1}(\mathbb{R}^{d})$
is admissible if the following conditions are satisfied:

\begin{description}
\item[(J1)] $J\in C^{3}(\mathbb{R}^{d}\backslash \{0\})$;

\item[(J2)] $J$ is radially symmetric, $J(x)=\tilde{J}(|x|)$ and $\tilde{J}$
is non-increasing;

\item[(J3)] $\tilde{J}^{\prime \prime }(r)$ and $\tilde{J}^{\prime }(r)/r$
are monotone on $(0,r_{0})$ for some $r_{0}>0$;

\item[(J4)] $|D^{3}J(x)|\leq C_{d}|x|^{-d-1}$ for some $C_d>0$.
\end{description}
\end{defn}

The advantage of this assumption is related to the following lemma which allows, in particular,
to control the $W^{2,p}(\Omega)-$norm of the convolution operator term without assuming $J\in W_{loc}^{2,1}(\mathbb{R}^{d})$.


\begin{lem}
\label{admiss} \cite[Lemma 2]{BRB}  Let $J$ be admissible. Then, for
every $p\in (1,\infty )$, there exists $C_{p}>0$ such that
\begin{equation*}
\Vert \nabla v\Vert _{L^{p}(\Omega )^{d\times d}}\leq C_{p}\Vert \psi \Vert
_{L^{p}(\Omega )}\,,\qquad \forall \psi \in L^{p}(\Omega )\,,
\end{equation*}%
where $v=\nabla J\ast \psi $. Here, $C_{p}=C^{\ast }p$ for $p\in \lbrack
2,\infty )$ and $C_{p}=C^{\ast }p/\left( p-1\right) $ for $p\in \left(
1,2\right) $, for some constant $C^{\ast }>0$ independent of $p.$
\end{lem}

Moreover, we also need the following lemma to handle the boundary condition
\eqref{BCbis} . Its proof immediately follows from the definition
of the seminorm in the space $W^{s,p}(\Gamma)$, with $s\in(0,1)$, and $1<p<\infty$ (cf. \cite[Chapter IX, Section 18]{DiB}),
namely,
\begin{align*}
&[u]_{W^{s,p}(\Gamma)}^p=\int_{\Gamma}\int_{\Gamma}
\frac{|u(x)-u(y)|^p}{|x-y|^{d-1+sp}}d\Gamma(x)\,d\Gamma(y)\,,
\end{align*}
where $d\Gamma$ is the surface measure on $\Gamma$.
\begin{lem}
\label{trace-product} Let $\varphi ,\psi \in W^{s,p}(\Gamma )\cap
L^{\infty }(\Gamma )$, where $s\in(0,1)$, $1<p<\infty$, and $d=2,3$.
Then $\varphi\, \psi \in W^{s,p}(\Gamma )\cap L^{\infty }(\Gamma )$ and
\begin{equation*}
\Vert \varphi\, \psi \Vert _{W^{s,p}(\Gamma )}\leq \Vert \varphi
\Vert _{L^{\infty }(\Gamma )}\Vert \psi \Vert _{W^{s,p}(\Gamma )}+\Vert \psi \Vert _{L^{\infty }(\Gamma )}\Vert \varphi
\Vert _{W^{s,p}(\Gamma )}\,.
\end{equation*}
\end{lem}

We also need to strengthen assumption \textbf{(H4)} by replacing it with
\begin{description}
\item[(H8)] $F\in C^{3}(-1,1)$ and $\lambda :=mF^{\prime \,\prime }\in
C^{1}([-1,1])$.
\end{description}
Note that this assumption is certainly satisfied in the case \eqref{degmob} and \eqref{mixentr}.

The main result of this section is

\begin{thm}\label{strong-sols}
Suppose that $d=2,3$, that assumptions \textbf{(H1)}-\textbf{(H3)} and \textbf{(H5)}-\textbf{(H8)}
are satisfied, and that %
$J\in W_{loc}^{2,1}(\mathbb{R}^{d})$ or that $J$ is
admissible. Let $\varphi _{0}\in V\cap L^{\infty }(\Omega )$ with 
$M(\varphi _{0})\in L^{1}(\Omega )$. Then, for every $T>0$, problem \eqref{Sy01}-\eqref{Sy04}, \eqref{sy5}-\eqref{sy6}
admits a strong solution $[\uvec,\pi,\varphi]$ on $[0,T]$ such that
\begin{align}
& \uvec\in L^{4(1-\theta)}(0,T;V_{div})
\,\cap\, L^{4(1-\theta)/\theta}(0,T;L^4(\Omega)^d) \,\cap\, L^\infty(0,T;G_{div})\,,\label{str-reg-1} \\
& \pi\in L^{4(1-\theta)}(0,T;H^2(\Omega))
 \,\cap\, L^{4(1-\theta)/\theta}(0,T;W^{1,4}(\Omega))
 \,\cap\, L^\infty(0,T;V_0)\,,\label{str-reg-2}\\
&\varphi \in L^{\infty }(0,T;V)\cap L^{2}(0,T;H^{2}(\Omega ))\cap H^{1}(0,T;H)\,,\label{str-reg-3}
\end{align}
\sergiotwo{for some $\theta\in (0,1/2)$}. In addition $\pi\in L^\infty(0,T;C^\alpha(\overline{\Omega})))$ for some $\alpha\in (0,1)$.
\end{thm}

\begin{oss}
\label{continuity}
We also have $\varphi \in C([0,T];V)$ because of \eqref{str-reg-3} (see, e.g., \cite[Section 5.9, Theorem 4]{EV}).
\end{oss}

In two dimensions a stronger regularity result can be proven, namely,
\begin{thm}\label{strong-sols-2d}
Suppose that $d=2$ and let the assumptions of Theorem \ref{strong-sols} hold.
If, in addition, $\varphi _{0}\in H^2(\Omega)$ and
the following compatibility condition is satisfied
\begin{equation}
\nabla B(\varphi _{0})\cdot \nvec=m(\varphi _{0})(\nabla J\ast
\varphi _{0})\cdot \nvec\,,\quad \mbox{ a.e. on }\Gamma \,,
\label{comp}
\end{equation}
then, for every $T>0$, problem \eqref{Sy01}-\eqref{Sy04}, \eqref{sy5}-\eqref{sy6}
admits a strong solution $[\uvec,\pi,\varphi]$ on $[0,T]$ satisfying, besides \eqref{str-reg-1} and \eqref{str-reg-2},
the further regularity properties
\begin{align}
&\uvec_t\in L^2(0,T;G_{div})\,,\label{str-reg-1-2d}\\
&\varphi\in L^\infty(0,T;H^2(\Omega))\cap H^1(0,T;V)\cap W^{1,\infty}(0,T;H). \label{str-reg-2-2d}
\end{align}
\end{thm}

\begin{oss}
The strong solution given by Theorem \ref{strong-sols} can be viewed as a strong solution to the equations \eqref{Sy01}, \eqref{Sy03}-\eqref{Sy04}
and boundary condition \eqref{sy5}$_2$ if, for instance, $\varphi$ satisfies the so-called strong separation property,
namely $\varphi$ is uniformly away from the pure states $\pm 1$ (see \cite[Rem.4.3]{FGGS}, see also \cite{FGG2} and references therein).
\end{oss}

\subsection{Proof of Theorem \ref{strong-sols}}


The proof is carried out by first providing existence of a strong solution on
a sufficiently small time interval. This is achieved by means of a fixed point argument based
on the Schauder's theorem. Then, by relying on suitable higher order estimates,
the local in time solution will be extended to an arbitrary time interval $[0,T]$, $T>0$.
A key tool for this proof is a regularity result for the convective nonlocal Cahn-Hilliard
equation with a given divergence-free velocity field (see Theorem \ref{reg-thm-bis}  in Section \ref{cnCH}).

Let us outline our Schauder's fixed point argument.
We first introduce the functional spaces $X_T$ and $Y_T$ given by
\begin{align}
&X_T:=L^\infty(0,T;H)\cap L^2(0,T;V)\cap H^1(0,T; V^\prime)\,,\nonumber\\
&Y_T:=L^\infty(0,T;V)\cap L^2(0,T;H^2(\Omega))\cap H^1(0,T;H)\,,
\end{align}
where $T>0$ will be fixed later on.

For every given $\varphi\in Y_T$, with $|\varphi|\leq 1$, we consider the following (formal) problem
\begin{align}
& \eta(\varphi)\mathbf{u} + \nabla\pi=(\nabla J\ast\varphi)\varphi, \qquad  \mbox{ in } Q_{T},\label{pb1}\\
& \mbox{div}(\mathbf{u})=0\label{pb2}, \qquad  \mbox{ in } Q_{T},\\
& \widetilde{\varphi}_t+\uvec\cdot\nabla\widetilde{\varphi}=\Delta B(\widetilde{\varphi})
- \mbox{div}(m(\widetilde{\varphi})(\nabla J\ast\widetilde{\varphi})), \qquad  \mbox{ in } Q_{T},\label{pb3}\\
&\mathbf{u}\cdot\mathbf{n} = 0\,, \quad
\big[\nabla B(\widetilde{\varphi})-m(\widetilde{\varphi})(\nabla J\ast\widetilde{\varphi})\big]\cdot\mathbf{n} = 0
\qquad\mbox{on }\Gamma\times (0,T),\label{pb4}\\
& \widetilde{\varphi}(0)=\varphi_0, \qquad\mbox{in }\Omega\,.\label{pb5}
\end{align}
We then divide the argument into four steps. These steps are carried out for $d=2$ or $\lambda$ constant. In the case
$d=3$ and non-constant $\lambda$ we shall also need to regularize \eqref{pb1}-\eqref{pb5} (see the end of the proof).

In Step 1 we study problem \eqref{pb1}, \eqref{pb2}, \eqref{pb4}$_1$, proving that, for every $\varphi\in Y_T$,
with $|\varphi|\leq 1$, it admits a unique solution $[\pi,\uvec]$. We also establish some crucial higher order estimates
for $\pi$ and $\uvec$ in terms of $\varphi$.
The estimates in Step 1 are purely elliptic and time is tacitly omitted.

In Step 2 we address the nonlocal convective Cahn-Hilliard
system \eqref{pb3}, \eqref{pb4}$_2$, \eqref{pb5}, with the velocity $\uvec$
given by the solution to \eqref{pb1}, \eqref{pb2}, \eqref{pb4}$_1$. We exploit Theorem \ref{reg-thm-bis}
to get a unique strong solution $\wph$ to this problem.
By virtue of the estimates derived in Step 1, we shall then conclude that, for
every given $\varphi\in Y_T$, with $|\varphi|\leq 1$, \eqref{pb1}-\eqref{pb5}
admits a unique solution $[\uvec,\widetilde{\varphi}]\in \big(L^\infty(0,T;G_{div})\cap L^2(0,T;V_{div})\big)
\times Y_T$,
with $|\widetilde\varphi|\leq 1$. This allows us to introduce the map $\mathcal{F}:\varphi\mapsto\wph$, which
 is well defined from the set
$\{\varphi\in Y_T\,:\,|\varphi|\leq 1\}$ into itself. The goal of Step 2 is to identify a suitable convex set of $Y_T$,
which is compact in $X_T$,
such that $\mathcal{F}$ is also a map from this set into itself. However, $\mathcal{F}$ cannot be defined
if $d=3$ and $\lambda$ non-constant. In this case we need to regularize $\uvec$ in \eqref{pb3}
and then pass to the limit in the regularization parameter to conclude (see below).

Step 3 will be devoted to prove that
$\mathcal{F}$ is continuous on $X_T$. The existence of a local in time strong solution will
then follow from Schauder's theorem.

In the final Step 4, we shall show that the local in time
solution constructed in the previous steps is indeed global.

We point out that all the estimates in the first three steps will be derived for both cases $d=2,3$.
We also remind once more that $\mathcal{F}$ cannot be defined if $d=3$ and $\lambda$ is not constant.
In this case we shall use a regularization argument.

In the sequel of this section we will indicate by $C$ a generic positive constant which only depends
on main constants of the problem (see (\textbf{H1})-(\textbf{H8})) and on $\Omega$ at most. This constant may vary also
within the same line. Any other dependency will be explicitly pointed out.

\bigskip

\textbf{Step 1.}
We first study the elliptic system \eqref{pb1}, \eqref{pb2}, \eqref{pb4}$_1$, with $\varphi$ given in $H^2(\Omega)$ (or in $V$)
such that $|\varphi|\leq 1$. First, we observe that problem \eqref{pb1}, \eqref{pb2}, \eqref{pb4}$_1$ is equivalent to the following
\begin{align}
&\mbox{div}\Big(\frac{1}{\eta(\varphi)}\nabla\pi\Big)=\mbox{div}\Big(\frac{(\nabla J\ast\varphi)\varphi}{\eta(\varphi)}\Big), \qquad  \mbox{ in } Q_{T},
\label{pb1-eq}\\
&\frac{\partial\pi}{\partial\mathbf{n}}=(\nabla J\ast\varphi)\varphi\cdot\mathbf{n}, \qquad  \mbox{ on } \Gamma \times (0,T),\label{pb2-eq}\\
&\uvec=-\frac{1}{\eta(\varphi)}\nabla\pi+\frac{1}{\eta(\varphi)}(\nabla J\ast\varphi)\varphi, \qquad  \mbox{ in } Q_{T}.\label{pb3-eq}
\end{align}
More precisely, for $\varphi\in V$ fixed, with $|\varphi|\leq1$, we can easily check that $[\pi,\mathbf{u}]\in V_0\times G_{div}$ is a solution to
\eqref{pb1}, \eqref{pb2}, \eqref{pb4}$_1$ if and only if $\pi\in V_0$
is a weak solution to \eqref{pb1-eq}, \eqref{pb2-eq}, namely $\pi$ satisfies
\begin{align}\label{weakfor-Darcy}
&\int_\Omega\frac{1}{\eta(\varphi)}\nabla\pi\cdot\nabla\psi=
\int_\Omega\frac{(\nabla J\ast\varphi)\varphi}{\eta(\varphi)}\cdot\nabla\psi\,,\qquad\forall \psi\in V\,,
\end{align}
and $\mathbf{u}\in G_{div}$
is given by \eqref{pb3-eq}. Indeed, let $\pi\in V_0$ satisfy \eqref{weakfor-Darcy} and let $\uvec\in H^d$ be given by
\eqref{pb3-eq}. Then, \eqref{pb1} trivially holds almost everywhere in $Q_T$, and we have that $\int_\Omega\uvec\cdot\nabla\psi=0$,
for all $\psi\in C^\infty_0(\Omega)$. This entails
that \eqref{pb2} holds in the sense of distributions. Hence, recalling that
the trace operator $\gamma_{\nvec}$ (which satisfies $\gamma_{\nvec}(\vvec)=\vvec\cdot\nvec$ on $\Gamma$,
for all $\vvec\in \sergiotwo{C^\infty_0(\mathbb{R}^2)^d}$) is a well defined linear and continuous operator
from the space $\{\vvec\in\sergiotwo{L^2(\Omega)^d}\,:\,\mbox{div}(\vvec)\in L^2(\Omega)\}$ into $H^{-1/2}(\Gamma)$,
by applying the generalized Stokes formula (see, e.g., \cite[Chapter I, Theorem 1.2]{T}) we get
$\langle\gamma_{\nvec}(\uvec),\psi_{|\Gamma}\rangle_{H^{1/2}(\Gamma)}=0$, for all $\psi\in V$,
which means that \eqref{pb4}$_1$ holds (in the generalized sense), and also that $\uvec\in G_{div}$.
Therefore, the equivalence of problem \eqref{pb1}, \eqref{pb2}, \eqref{pb4}$_1$ with problem
\eqref{pb1-eq}-\eqref{pb3-eq} is proven.

A straightforward application of the
Lax-Milgram theorem yields that, for every $\varphi\in V$, with $|\varphi|\leq 1$, problem \eqref{pb1-eq}-\eqref{pb3-eq}
(and hence also problem \eqref{pb1}, \eqref{pb2}, \eqref{pb4}$_1$) admits a unique solution
$[\pi, \uvec]\in V_0\times G_{div}$. Moreover, the following estimates hold
\begin{align}
&\Vert\nabla\pi\Vert\leq\frac{\eta_\infty}{\eta_1}\Vert(\nabla J\ast\varphi)\varphi\Vert
\leq\frac{\eta_\infty}{\eta_1}\,b\,\Vert\varphi\Vert\leq \frac{\eta_\infty}{\eta_1}\,b\,|\Omega|^{1/2}\leq C\,,
\label{e5}\\
&\Vert\uvec\Vert\leq \frac{1}{\eta_1}\Vert\nabla\pi\Vert+\frac{b}{\eta_1}\Vert\varphi\Vert
\leq\frac{b}{\eta_1}\Big(1+\frac{\eta_\infty}{\eta_1}\Big)\Vert\varphi\Vert
\leq \frac{b}{\eta_1}\Big(1+\frac{\eta_\infty}{\eta_1}\Big)|\Omega|^{1/2}\leq C\,,\label{e8}
\end{align}
where $\eta_\infty:=\Vert\eta\Vert_{L^\infty(-1,1)}$ (see (\textbf{H1})).

Assume now that $\varphi\in H^2(\Omega)$, with $|\varphi|\leq 1$.
Then, problem \eqref{pb1-eq}-\eqref{pb2-eq} is equivalent to the elliptic problem given by
\begin{align}\label{ellipt-nondiv}
&\Delta\pi=\frac{\eta^\prime(\varphi)}{\eta(\varphi)}\,\nabla\varphi\cdot\nabla\pi
+\eta(\varphi)\,\mbox{div}\Big(\frac{(\nabla J\ast\varphi)\varphi}{\eta(\varphi)}\Big)\,,
\end{align}
together with the boundary condition \eqref{pb2-eq}. Indeed, it is easy to check that
the weak formulation of \eqref{ellipt-nondiv} subject to \eqref{pb2-eq} is satisfied if and only if
\eqref{weakfor-Darcy} is satisfied. To this aim it is enough to observe that, being $\varphi\in H^2(\Omega)$ with $|\varphi|\leq 1$,
recalling \textbf{(H1)}, we have that $\psi=\eta(\varphi)\chi\in V$ if and only if $\chi\in V$.
Hence, by taking $\psi=\eta(\varphi)\chi$ in \eqref{weakfor-Darcy} 
we can deduce the weak formulation of \eqref{ellipt-nondiv} subject to \eqref{pb2-eq} (with $\chi\in V$ as test function) from \eqref{weakfor-Darcy}, and conversely.

Thus we consider problem \eqref{pb1-eq}-\eqref{pb2-eq}, written as \eqref{ellipt-nondiv} with \eqref{pb2-eq},
and we apply classical elliptic regularity theory, together with a bootstrap argument, to deduce that $\pi\in H^2(\Omega)$.
Indeed, we begin by noting that \sergiotwo{the right hand side of \eqref{pb2-eq} belongs to $H^{1/2}(\Gamma)$,
and the right hand side of \eqref{ellipt-nondiv} belongs to
$L^{2^-}(\Omega)$, if $d=2$, and to $L^{3/2}(\Omega)$, if $d=3$.} 
Hence, \sergiotwo{by a classical elliptic regularity result} (remember that $\Gamma$ is smooth enough), we have that $\pi\in W^{2,2^-}(\Omega)$,
\sergiotwo{if $d=2$, and $\pi\in W^{2,3/2}(\Omega)$, if $d=3$}. Thus
$\nabla\pi\in W^{1,2^-}(\Omega)^2\hookrightarrow L^4(\Omega)^2$, \sergiotwo{if $d=2$, and
$\nabla\pi\in W^{1,3/2}(\Omega)^3\hookrightarrow L^3(\Omega)^3$,  if $d=3$. This entails that} the right hand side of \eqref{ellipt-nondiv}
is in $H$, \sergiotwo{and hence that $\pi\in H^2(\Omega)$, for both cases $d=2,3$. From \eqref{pb3-eq}}
we also get $\uvec\in V_{div}$.

Let us now derive the estimates for the $H^2(\Omega)-$norm of $\pi$ and for the $V^d-$norm of $\uvec$ in terms of the $H^2(\Omega)-$norm
of $\varphi$. \sergiotwo{To this aim we first derive
an estimate that controls the $L^4(\Omega)^d-$norm of $\mathbf{u}$ in terms of the $H^2(\Omega)-$norm of $\varphi$.}
This estimate, which is obtained by relying on the H\"{o}lder continuity property of the pressure $\pi$,
will turn out to be a key ingredient in our fixed point argument.
First, observe that, by applying Proposition \ref{Bre-Mir} \sergiotwo{for $d=2,3$}, the following interpolation inequality holds
\begin{align}
&\Vert\pi\Vert_{W^{1,4}(\Omega)}\leq C\Vert\pi\Vert_{H^2(\Omega)}^\theta\Vert\pi\Vert_{W^{4/\rho,\rho}(\Omega)}^{1-\theta}\,,
\label{e9}
\end{align}
where $4<\rho<\infty$ and $\theta\in(0,1)$ is given by $\theta=\theta_\rho:=\frac{1}{2}\frac{\rho-4}{\rho-2}$. Indeed, by taking
$r=1$, $q=4$, $s_1=4/\rho$, $p_1=\rho$, $s_2=2$, $p_2=2$ in Proposition \ref{Bre-Mir} (and replacing $\theta$ by $1-\theta$), from \eqref{Bre-Mir-cond} we get $\theta=\theta_\rho$, and $s=r=1$. Since \eqref{GN-ex} is not satisfied, then we obtain \eqref{e9}.
We point out that $\theta=\theta_\rho<1/2$, for every $4<\rho<\infty$
\sergiotwo{(notice that
$\theta$ does not depend on $d$)}. Next, we fix $\rho$ such that $4/\rho<\alpha$, where $\alpha\in(0,1)$,
and this ensures the embedding $C^\alpha(\overline{\Omega})\hookrightarrow
W^{4/\rho,\rho}(\Omega)$ \sergiotwo{for both cases $d=2,3$}.
Hence, from \eqref{e9} we deduce the following inequality
\begin{align}
&\Vert\pi\Vert_{W^{1,4}(\Omega)}\leq
C\Vert\pi\Vert_{H^2(\Omega)}^\theta\Vert\pi\Vert_{C^\alpha(\overline{\Omega})}^{1-\theta}\,,\qquad\mbox{with }\,\,
\theta<1/2\,.\label{e10}
\end{align}
With this interpolation inequality at our disposal, we now turn back to
the elliptic problem \eqref{pb1-eq}-\eqref{pb2-eq}, which is a special case of \eqref{Neu1}-\eqref{Neu2}
with
$$a_{ij}=\frac{1}{\eta(\varphi)}\delta_{ij}\,,\qquad\mathbf{f}=-\frac{(\nabla J\ast\varphi)\varphi}{\eta(\varphi)}\,,
\qquad f=0\,,\qquad \psi=0\,.$$
The ellipticity condition \eqref{ell-cond} is satisfied with $\Lambda_\ast=\eta_\infty$
(here we use $|\varphi|\leq 1$), and $\Lambda^\ast=\eta_1$. Moreover, we can immediately
check that condition \eqref{DiBen1} holds (taking, for simplicity, $\varepsilon=1$)
$$
\Vert\mathbf{f}\Vert_{L^{d+1}(\Omega)^d}
=\left\Vert\frac{(\nabla J\ast\varphi)\varphi}{\eta(\varphi)}\right\Vert_{L^{d+1}(\Omega)^d}\leq
\frac{b|\Omega|^{1/(d+1)}}{\eta_1}\,.
$$
Hence, from Proposition \ref{DiBenedetto} we
infer that $\pi$ is H\"{o}lder continuous in $\overline{\Omega}$, and that there exist constants $\Theta$ and
$\alpha\in(0,1)$ depending only on $\eta_1$, $\eta_\infty$, $b$, $|\Omega|$, \sergiotwo{$d$}, and on the $C^1$ structure of $\Gamma$, such that
\begin{align}
&\Vert\pi\Vert_{C^\alpha(\overline{\Omega})}\leq \Theta(\eta_1, \eta_\infty, b, |\Omega|, \sergiotwo{d},  \Gamma)\,.
\label{e53}
\end{align}
By exploiting this estimate, we can now apply \eqref{e10}
(with the same exponent $\alpha$ as given by Proposition \ref{DiBenedetto}) to obtain the bound
\begin{align}
&\Vert\pi\Vert_{W^{1,4}(\Omega)}\leq C\Vert\pi\Vert_{H^2(\Omega)}^\theta\,,\qquad\mbox{with }\,\,
\theta<1/2\,.\label{e11}
\end{align}
 Therefore, from \eqref{pb3-eq}, by means 
of \eqref{e11}, we get the following estimate for the $L^4(\Omega)^d-$norm of $\uvec$
in terms of the $H^2(\Omega)-$norm of $\pi$
\begin{align}
&\Vert\uvec\Vert_{L^4(\Omega)^d}\leq \frac{1}{\eta_1}\Vert\pi\Vert_{W^{1,4}(\Omega)}
+\frac{b}{\eta_1}|\Omega|^{1/4}\leq C(\Vert\pi\Vert_{H^2(\Omega)}^\theta+1)\,.
\label{e12}
\end{align}
In order to get an estimate for the $L^4(\Omega)^d-$norm of $\uvec$
in terms of the $H^2(\Omega)-$norm of $\varphi$, we employ a classical elliptic regularity estimate, the following well-known
Gagliardo-Nirenberg inequality
(see, e.g., \eqref{GN-spe} for $p=4$)
$$\Vert\nabla\varphi\Vert_{L^4(\Omega)^d}\leq C\Vert\varphi\Vert_{L^\infty(\Omega)}^{1/2}
\Vert\varphi\Vert_{H^2(\Omega)}^{1/2}\,,$$
\sergiotwo{the control  \eqref{e11}}, and Lemma \ref{admiss}, \sergiotwo{to get,} from \eqref{ellipt-nondiv} with \eqref{pb2-eq},
\begin{align}
\Vert\pi\Vert_{H^2(\Omega)}&\leq  C\Big(\Vert\frac{\eta^\prime(\varphi)}{\eta(\varphi)}\,\nabla\varphi\cdot\nabla\pi\Vert+
\Vert\mbox{div}(\nabla J\ast\varphi)\Vert
+\Vert\eta(\varphi)(\nabla J\ast\varphi)\cdot\nabla\Big(\frac{\varphi}{\eta(\varphi)}\Big)\Vert\nonumber\\[1mm]
&+\Vert(\nabla J\ast\varphi)\varphi\cdot\nvec\Vert_{H^{1/2}(\Gamma)}\Big)\nonumber\\[1mm]
&\leq C\big(\sergiotwo{\Vert\varphi\Vert_{H^2(\Omega)}^{1/2}\Vert\pi\Vert_{H^2(\Omega)}^{\theta}}
+\Vert\varphi\Vert_V+\Vert(\nabla J\ast\varphi)\varphi\cdot\nvec\Vert_{H^{1/2}(\Gamma)}\big)\,.\label{e3}
\end{align}
As far as the boundary term in \eqref{e3} is concerned, invoking Lemma \ref{trace-product}, we have that
\begin{align}
&\Vert(\nabla J\ast\varphi)\varphi\cdot\nvec\Vert_{H^{1/2}(\Gamma)}=
\Vert \varphi\frac{\partial}{\partial\nvec}(J\ast\varphi)\Vert_{H^{1/2}(\Gamma)}
\leq \Vert\varphi\Vert_{L^\infty(\Gamma)}\Vert\frac{\partial}{\partial\nvec}(J\ast\varphi)\Vert_{H^{1/2}(\Gamma)}
\nonumber\\[1mm]
&+\Vert\frac{\partial}{\partial\nvec}(J\ast\varphi)\Vert_{L^\infty(\Gamma)}\Vert\varphi\Vert_{H^{1/2}(\Gamma)}
\leq C\Vert J\ast\varphi\Vert_{H^2(\Omega)}+C\Vert\nabla J\ast\varphi\Vert_{\sergiotwo{W^{1,4}(\Omega)^d}}\Vert\varphi\Vert_V\nonumber\\
&\leq C(\Vert\varphi\Vert+\Vert\varphi\Vert_{\sergiotwo{L^4(\Omega)}}\Vert\varphi\Vert_V)\leq C(1+\Vert\varphi\Vert_V)\,,\label{e4}
\end{align}
where Lemma \ref{admiss} has been used again, as well as the embedding
\sergiotwo{$W^{1,4}(\Omega)\hookrightarrow C(\overline{\Omega})$,
for $d=2,3$.}
Therefore, collecting 
\sergiotwo{\eqref{e3} and \eqref{e4},} we get
\begin{align}
&\Vert\pi\Vert_{H^2(\Omega)}
\leq  C\,\big(\sergiotwo{\Vert\varphi\Vert_{H^2(\Omega)}^{\frac{1}{2(1-\theta)}}+\Vert\varphi\Vert_V}+1\big)\,.
\label{e6}
\end{align}
\sergiotwo{The desired estimate of the $L^4(\Omega)^d-$norm of $\uvec$ in terms of the $H^2(\Omega)-$norm of $\varphi$ then follows from
\eqref{e12} and \eqref{e6}, namely
\begin{align}\label{e27}
&\Vert\uvec\Vert_{L^4(\Omega)^d}\leq C\,\big(\Vert\varphi\Vert_{H^2(\Omega)}^{\frac{\theta}{2(1-\theta)}}
+\Vert\varphi\Vert_V^\theta+1\big)\,.
\end{align}
}
\sergiotwo{We can also deduce an estimate for the $V^d-$norm of $\uvec$. Indeed, }
from \eqref{pb3-eq}, \sergiotwo{and again using \eqref{e11},} we have that
\begin{align}
\Vert\uvec\Vert_{V^d}&\leq C(\Vert\nabla\pi\Vert+\Vert\pi\Vert_{H^2(\Omega)}+\Vert\nabla\varphi\Vert_{\sergiotwo{L^4(\Omega)^d}}
\Vert\nabla\pi\Vert_{\sergiotwo{L^4(\Omega)^d}}+\Vert(\nabla J\ast\varphi)\varphi\Vert_{V^d})\nonumber\\[1mm]
&\leq C(\Vert\pi\Vert_{H^2(\Omega)}+\sergiotwo{\Vert\varphi\Vert_{H^2(\Omega)}^{1/2}
\Vert\pi\Vert_{H^2(\Omega)}^{\theta}}+\Vert\varphi\Vert_V+1)\,,\nonumber
\end{align}
and hence, on account of \eqref{e6}, we obtain
\begin{align}
&\Vert\uvec\Vert_{V^d}
\leq C\,\big(\sergiotwo{\Vert\varphi\Vert_{H^2(\Omega)}^{\frac{1}{2(1-\theta)}}
+\Vert\varphi\Vert_V}+1\big)\,.\label{e7}
\end{align}
Summing up, from the analysis of the problem \eqref{pb1-eq}-\eqref{pb3-eq} we
know that, for every $\varphi\in H^2(\Omega)$, with $|\varphi|\leq 1$,
system \eqref{pb1}, \eqref{pb2}, \eqref{pb4}$_1$ admits
a unique solution $[\pi,\uvec]\in \big(H^2(\Omega)\cap V_0\big)\times V_{div}$
such that estimates \sergiotwo{\eqref{e6}-\eqref{e7}} hold.

\bigskip

\textbf{Step 2.} We now consider problem \eqref{pb3}, \eqref{pb4}$_2$, \eqref{pb5}, where
$\uvec$ is the second component of the unique solution to \eqref{pb1}, \eqref{pb2}, \eqref{pb4}$_1$
with $\varphi$ given in $Y_T$ satisfying $|\varphi|\leq 1$.
Thanks to Theorem \ref{reg-thm-bis} we know that if $d=2$ or $\lambda$ is constant then \eqref{pb3}, \eqref{pb4}$_2$, \eqref{pb5}
admits a unique strong solution $\wph\in Y_T$ with $|\wph|\leq 1$ (see  \eqref{regNCH2} and \eqref{regNCH1}).
Indeed, from \eqref{e12} (or also from \eqref{e7}) and from $\varphi\in Y_T$, it is immediate to check
that condition \eqref{regvel-bis} holds with $r=4$.
Therefore, the map $\mathcal{F}:\varphi\mapsto\wph$, that associates to every $\varphi$, given in \eqref{pb1},
the unique solution $\wph$ to \eqref{pb3}, \eqref{pb4}$_2$, \eqref{pb5}, is well defined from the set $\{\psi\in Y_T:|\psi|\leq 1\}$ into itself.

Our goal is now to show that, provided that $T>0$ and $R>0$ are suitably chosen,
the map $\mathcal{F}$ satisfies $\mathcal{F}:B_{Y_T}(R)\to B_{Y_T}(R)$, namely it is also a map from $B_{Y_T}(R)$ into itself,
where $B_{Y_T}(R)$ is the closed convex set given by
\begin{align}
&B_{Y_T}(R):=\{\psi\in Y_T\,:\,
\Vert\psi\Vert_{Y_T}\leq R\,,\quad |\psi|\leq 1\}\,.\nonumber
\end{align}

Fix $\varphi\in  B_{Y_T}(R)$. From \sergiotwo{\eqref{e27}} we first obtain an estimate for $\uvec$
in $L^4(0,T;L^4_{div}(\Omega)^{\sergiotwo{d}})$.
More precisely, we find
\begin{align}
\Vert\uvec\Vert_{L^4(0,T;L^4(\Omega)^{\sergiotwo{d}})}&\leq
\sergiotwo{
C\big(\Vert\varphi\Vert_{L^2(0,T;H^2(\Omega))}^{\frac{\theta}{2(1-\theta)}} T^{\frac{1-2\theta}{4(1-\theta)}}
+\Vert\varphi\Vert_{L^{\infty}(0,T;V)}^\theta T^{1/4}+T^{1/4}\big)}\nonumber\\[1mm]
&\sergiotwo{
\leq C\big( R^{\frac{\theta}{2(1-\theta)}}\,T^{\frac{1-2\theta}{4(1-\theta)}}+R^\theta\, T^{1/4}+T^{1/4}\big)\,.}
\label{e16}
\end{align}
On the other hand, we know that $\wph$ satisfies the differential identity \eqref{diffid}. Therefore,
on account of \textbf{(H7)}, we get
\begin{align}
\frac{1}{2}\frac{d\Phi}{dt}+\alpha_0\Vert\wph_t\Vert^2&\leq\Vert\uvec\Vert_{L^4(\Omega)^{\sergiotwo{d}}}
\Vert\nabla B(\wph)\Vert_{L^4(\Omega)^d}\Vert\wph_t\Vert+(m_\infty+m_\infty^\prime)\,b\,\Vert\wph_t\Vert\Vert\nabla B(\wph)\Vert\,
\nonumber\\[1mm]
&\leq\frac{\alpha_0}{4}\Vert\wph_t\Vert^2+C\Vert\uvec\Vert_{L^4(\Omega)^{\sergiotwo{d}}}^2\Vert B(\wph)\Vert_{H^2(\Omega)}
+C\Vert\nabla\wph\Vert^2\,,\label{e13}
\end{align}
where $\Phi :=\Vert \nabla B(\wph)\Vert ^{2}-2\big(m(\wph )(\nabla J\ast
\wph ),\lambda (\wph )\nabla \wph \big)$ satisfies, for all $t\in [0,T]$,
\begin{align}
&K_1(\Vert\nabla\wph(t)\Vert^2-1)\leq\Phi(t)\leq K_2(\Vert\nabla\wph(t)\Vert^2+1)\,,\label{e15}
\end{align}
with two positive constants $K_1,K_2$ depending on $m$, $\lambda$, and $J$. Let us
estimate the $H^2(\Omega)-$norm of $B(\wph)$ in terms of the $H-$norm of $\wph_t$ and of the $L^4(\Omega)^d-$norm of $\uvec$.
To this aim, by relying on the elliptic estimate
\begin{equation*}
\Vert B(\wph )\Vert _{H^{2}(\Omega )}\leq C\left( \Vert \Delta B(\wph
)\Vert +\Vert B(\wph )\Vert _{V}+\left\Vert \nabla B(\wph )\cdot
\nvec\right\Vert _{H^{1/2}(\Gamma )}\right) \,,
\end{equation*}
and on \eqref{pb3}, 
we have that
\begin{align}
\Vert B(\wph)\Vert _{H^2(\Omega)}&\leq C \big(\Vert \wph_t\Vert
+\Vert\uvec\Vert_{L^4(\Omega)^{\sergiotwo{d}}}\Vert\nabla B(\wph)\Vert_{L^4(\Omega)^{\sergiotwo{d}}}
+\Vert \nabla B(\wph)\Vert+1\big)\nonumber\\[1mm]
&\leq C\big(\Vert\wph_t\Vert+\Vert\uvec\Vert_{L^4(\Omega)^{\sergiotwo{d}}}\Vert B(\wph)\Vert_{H^2(\Omega)}^{1/2}
+\Vert\nabla B(\wph)\Vert+1\big)\,. \nonumber
\end{align}
Hence we find that
\begin{align}
\Vert B(\wph)\Vert _{H^2(\Omega)}&\leq C \big(\Vert \wph_t\Vert
+\Vert\uvec\Vert_{L^4(\Omega)^{\sergiotwo{d}}}^2+\Vert \nabla B(\wph)\Vert+1\big)\,.\label{e14}
\end{align}
By inserting \eqref{e14} into \eqref{e13} and taking Young's inequality and \eqref{e15} into account, we easily get
\begin{align}
\frac{d\Phi}{dt}+\alpha_0\Vert\wph_t\Vert^2&\leq C\big(1+\Vert\uvec\Vert_{L^4(\Omega)^{\sergiotwo{d}}}^2\big)\Vert\nabla B(\wph)\Vert
+C\big(1+\Vert\uvec\Vert_{L^4(\Omega)^{\sergiotwo{d}}}^4\big)\nonumber\\[1mm]
&\leq C_1\big(1+\Vert\uvec\Vert_{L^4(\Omega)^{\sergiotwo{d}}}^4\big)\Phi
+C_2\big(1+\Vert\uvec\Vert_{L^4(\Omega)^{\sergiotwo{d}}}^4\big)\,.\label{e17}
\end{align}
We shall henceforth denote by $C_i$, $i\in\mathbb{N}$, some positive constants that depend on
the structural parameters of the problem, namely on $J,m,\lambda,\eta,\Omega,\Gamma$,
but are independent of $T$, $R$, and $\varphi_0$.
By applying Gronwall's lemma to \eqref{e17} and using \eqref{e16} and \eqref{e15}, we obtain
\begin{align}
\Vert\wph(t)\Vert_V^2&\leq C_3\, e^{C_1(T+\Vert\uvec\Vert_{L^4(0,T;L^4(\Omega)^{\sergiotwo{d}})}^4)}
\big(1+\Vert\nabla\varphi_0\Vert^2+C_1(T+\Vert\uvec\Vert_{L^4(0,T;L^4(\Omega)^{\sergiotwo{d}})}^4)\big)+C_3\nonumber\\[1mm]
&\leq C_3\, e^{\Lambda(R,T)}\big(1+\Vert\nabla\varphi_0\Vert^2+\Lambda(R,T)\big)+ C_3\,,\nonumber
\end{align}
where we have set
\begin{equation}
\label{Lambda}
\Lambda(R,T):=C_4\,( \sergiotwo{R^{\frac{2\theta}{1-\theta}}\,T^{\frac{1-2\theta}{1-\theta}}+
R^{4\theta}T+T} )\,.
\end{equation}
Therefore, we get
\begin{align}
&\Vert\wph\Vert_{L^\infty(0,T;V)}\leq \Gamma_1\big(\Lambda(R,T),\Vert\nabla\varphi_0\Vert\big)\,,
\label{e18}
\end{align}
where
$$\Gamma_1^2(\Lambda,\xi):=C_3\,e^\Lambda\,(1+\xi^2+\Lambda)+C_3\,.$$
By integrating in time \eqref{e17} and using \eqref{e18}, we deduce
\begin{align}
&\Vert\wph\Vert_{H^1(0,T;H)}^2\leq
C_5\,(1+\Vert\nabla\varphi_0\Vert^2)+C_6\,\Lambda(R,T)\,\Gamma_1^2\big(\Lambda(R,T),\Vert\nabla\varphi_0\Vert\big)\,.
\nonumber
\end{align}
Thus we have
\begin{align}
&\Vert\wph\Vert_{H^1(0,T;H)}\leq \Gamma_2\big(\Lambda(R,T),\Vert\nabla\varphi_0\Vert\big)\,,
\label{e20}
\end{align}
where
$$\Gamma_2^2(\Lambda,\xi):=C_5\,(1+\xi^2)+C_6\,\Lambda\,\Gamma_1^2(\Lambda,\xi)\,.$$
In order to estimate the norm of $\wph$ in $L^2(0,T;H^2(\Omega))$, we first consider
the identity
\begin{equation}\label{e43}
\partial _{ij}^{2}\wph=\frac{1}{\lambda (\wph )}\partial
_{ij}^{2}B(\wph )-\frac{1}{\lambda ^{2}(\wph )}\partial _{i}\lambda
(\wph )\partial _{j}B(\wph )\,,\qquad i,j=1,2\,,
\end{equation}%
from which we deduce
\begin{align}
\Vert\partial _{ij}^{2}\wph\Vert&\leq \frac{1}{\alpha_0}\Vert B(\wph)\Vert_{H^2(\Omega)}+
\frac{\lambda_\infty^\prime}{\alpha_0^2}\Vert\nabla\wph\Vert_{L^4(\Omega)^{\sergiotwo{d}}}\Vert\nabla B(\wph)\Vert_{L^4(\Omega)^{\sergiotwo{d}}}\nonumber\\[1mm]
&\leq \frac{1}{\alpha_0}\Vert B(\wph)\Vert_{H^2(\Omega)}+C\frac{\lambda_\infty^{1/2}\lambda_\infty^\prime}{\alpha_0^2}
\Vert\wph\Vert_{H^2(\Omega)}^{1/2}
\Vert B(\wph)\Vert_{H^2(\Omega)}^{1/2}\nonumber\\[1mm]
&\leq\delta\,\Vert\wph\Vert_{H^2(\Omega)}+C_\delta\,\Big(\frac{1}{\alpha_0}
+\frac{\lambda_\infty\lambda_\infty^{\prime^{\,2}}}{\alpha_0^4}\Big)
\Vert B(\wph)\Vert_{H^2(\Omega)}\,.\nonumber
\end{align}
Hence, taking $\delta>0$ small enough, we find
\begin{align}
\Vert\wph\Vert_{H^2(\Omega)}&\leq C_{\alpha_0,\lambda_\infty,\lambda_\infty^\prime}
\,\Vert B(\wph)\Vert_{H^2(\Omega)}
\nonumber\\[1mm]
&\leq C_{\alpha_0,\lambda_\infty,\lambda_\infty^\prime} \big(\Vert \wph_t\Vert
+\Vert\uvec\Vert_{L^4(\Omega)^{\sergiotwo{d}}}^2+\Vert \nabla B(\wph)\Vert+1\big)\,,\label{e19}
\end{align}
where, in the last inequality, \eqref{e14} has been used.
Therefore, on account of \eqref{e20}, \eqref{e18}, from \eqref{e19}  we infer
\begin{align}
\Vert\wph\Vert_{L^2(0,T;H^2(\Omega))}^2&\leq C_{\alpha_0,\lambda_\infty,\lambda_\infty^\prime}\,
\big(\Gamma_2^2+\Vert\uvec\Vert_{L^4(0,T;L^4(\Omega)^{\sergiotwo{d}})}^4+\Gamma_1^2+T\big)\nonumber\\[1mm]
&\leq C_7\,
\big(\Gamma_2^2+\Gamma_1^2+\Lambda(R,T)\big)\,.\nonumber
\end{align}
Thus we find
\begin{align}
&\Vert\wph\Vert_{L^2(0,T;H^2(\Omega))}\leq \Gamma_3\big(\Lambda(R,T),\Vert\nabla\varphi_0\Vert\big)\,,
\label{e21}
\end{align}
where we have set
$$\Gamma_3^2(\Lambda,\xi):=C_7
\,\big(\Gamma_2^2(\Lambda,\xi)+\Gamma_1^2(\Lambda,\xi)+\Lambda\big)\,.$$
Let us now choose $R$ in the following way
$$R:=3\,\max_{1\leq i\leq 3} \Gamma_i\big(1,\Vert\nabla\varphi_0\Vert\big)\,,$$
and observe that $R$ only depends on the $V-$norm of $\varphi_0$.
With this choice of $R$, we fix $T>0$ such that $\Lambda(R,T)\leq 1$. This is possible
thanks to the fact that $\theta<1/2$ (cf. \eqref{Lambda}). Therefore, \eqref{e18}, \eqref{e21}, and \eqref{e20}
yield
\begin{align}
\Vert\wph\Vert_{Y_T}&=\Vert\wph\Vert_{L^\infty(0,T;V)}+\Vert\wph\Vert_{L^2(0,T;H^2(\Omega))}+
\Vert\wph\Vert_{H^1(0,T;H)}\nonumber\\[1mm]
&\leq \sum_{i=1}^{3}\Gamma_i\big(\Lambda(R,T),\Vert\nabla\varphi_0\Vert\big)
\leq\sum_{i=1}^{3}\Gamma_i\big(1,\Vert\nabla\varphi_0\Vert\big)\leq R\,.\nonumber
\end{align}
Therefore $\mathcal{F}$ takes $B_{Y_T}(R)$ into itself.

\bigskip

\textbf{Step 3.} In this step we shall prove that $\mathcal{F}:B_{Y_T}(R)\to B_{Y_T}(R)$
is continuous with respect to the strong topology of $X_T$.
Take a sequence $\{\varphi_n\}\subset B_{Y_T}(R)$ such that $\varphi_n\to\overline{\varphi}$ in $X_T$.
We have (up to a subsequence) that $\varphi_n\rightharpoonup\overline{\varphi}$ weakly star in $Y_T$ and
$\overline{\varphi}\in B_{Y_T}(R) $.

Let us denote by $\mathcal{Q}_1$ and $\mathcal{Q}_2$ the maps defined by
$\pi=\mathcal{Q}_1(\varphi)$ and $\uvec=\mathcal{Q}_2(\varphi)$, respectively, where
$[\pi,\uvec]\in (H^2(\Omega)\cap V_0)\times V_{div}$ is the unique weak solution to \eqref{pb1-eq},
\eqref{pb2-eq}, \eqref{pb3-eq}. Set then $\pi_n:=\mathcal{Q}_1(\varphi_n)$ and $\uvec_n:=\mathcal{Q}_2(\varphi_n)$.
Thanks to \eqref{e5}, \eqref{e6} and to \eqref{e8}. \eqref{e7},  we have that (up to a subsequence)
\begin{align}
&\pi_n\rightharpoonup\pi^\ast\,,\qquad\mbox{ weakly star in }L^\infty(0,T;V_0)\cap L^2(0,T;H^2(\Omega))\,,
\nonumber\\[1mm]
&\uvec_n \rightharpoonup\uvec^\ast\,,\qquad\mbox{ weakly star in }L^\infty(0,T;G_{div})\cap L^2(0,T;V_{div})\,.
\label{e29}
\end{align}
Writing the weak formulation \eqref{weakfor-Darcy} with $\pi_n$ and $\varphi_n$, multiplying it
by a test function $\omega\in C^\infty_0(0,T)$, and passing to the limit as $n\to\infty$,
we can easily deduce that $\pi^\ast$ again satisfies \eqref{weakfor-Darcy}, and hence (thanks to uniqueness) that
$\pi^\ast=\overline{\pi}:=\mathcal{Q}_1(\overline{\varphi})$. Moreover, by passing to the limit in \eqref{pb3-eq},
written for $[\uvec_n$, $\pi_n$, $\varphi_n]$, we get also $\uvec^\ast=\overline{\uvec}:=\mathcal{Q}_2(\overline{\varphi})$.

Let us now denote by $\mathcal{G}$ the map that to each $\uvec\in L^4(0,T;L^4_{div}(\Omega)^d)$ associates
$\widetilde{\varphi}=\mathcal{G}(\uvec)$, where $\widetilde{\varphi}\in Y_T$ is the unique strong solution to
\eqref{pb3}, \eqref{pb4}$_2$, \eqref{pb5} given by Lemma \ref{reg-thm-bis}. Then set $\widetilde{\varphi}_n:=\mathcal{G}(\uvec_n)
=(\mathcal{G}\circ\mathcal{Q}_2)(\varphi_n)=\mathcal{F}(\varphi_n)$. From Step 2 we know that
$\{\widetilde{\varphi}_n\}\subset B_{Y_T}(R)$. Hence we have that (up to a subsequence)
$\widetilde{\varphi}_n\rightharpoonup \varphi^\ast$ weakly star in $Y_T$.
Writing the weak formulation of \eqref{pb3}, \eqref{pb4}$_2$, \eqref{pb5} for $\widetilde{\varphi}_n$,
with $\uvec_n$ given in the convective term, we obtain
$$
\langle\widetilde{\varphi}_{n,t},\psi\rangle_V+(\nabla B(\wph_n),\nabla\psi)=
(\uvec_n\wph_n,\nabla\psi)+(m(\wph_n)(\nabla J\ast\wph_n),\nabla\psi)\,,\qquad\forall\psi\in V\,.
$$
Multiplying the above identity
by a test function $\omega\in C^\infty_0(0,T)$ and passing to the limit, on account of the weak
and strong convergences for $\{\widetilde{\varphi}_n\}$ and for $\{\uvec_n\}$, it is not difficult
to see that the same weak formulation is satisfied also for $\varphi^\ast$, with $\overline{\uvec}$
in the convective term. Therefore, thanks to the uniqueness of the strong solution to problem
\eqref{pb3}, \eqref{pb4}$_2$, \eqref{pb5} (with $\uvec=\overline{\uvec}$ given), we have that
$\varphi^\ast=\mathcal{G}(\overline{\uvec})=(\mathcal{G}\circ\mathcal{Q}_2)(\overline\varphi)=\mathcal{F}(\overline\varphi)$.

We thus conclude that, up to a subsequence, $\mathcal{F}(\varphi_n)\rightharpoonup\mathcal{F}(\overline\varphi)$,
weakly star in $Y_T$ and strongly in $X_T$, due to the compact injection $Y_T\hookrightarrow\hookrightarrow X_T$.
The uniqueness of the limit $\mathcal{F}(\overline\varphi)$ entails the strong convergence
for the whole sequence $\{\mathcal{F}(\varphi_n)\}$. This concludes the proof of the continuity of $\mathcal{F}$.

Using the fact that the closed convex set $B_{Y_T}(R)$ is compact in $X_T$,
we can now apply Schauder's fixed point theorem to the map $\mathcal{F}:B_{Y_T}(R)\to B_{Y_T}(R)$ and obtain a fixed point $\varphi \in B_{Y_T}(R)$.
Thus, recalling also estimates \eqref{e6} and \eqref{e7}, we deduce that there exists a strong solution $[\uvec,\pi,\varphi]$ on $[0,T]$, for some $T>0$
small enough such that \eqref{str-reg-1}-\eqref{str-reg-3} hold.

\bigskip

\textbf{Step 4.} Our goal is now to prove that the local in time solution
can be extended to an arbitrary time interval $[0,T]$, for any $T>0$.
Let $T_m\in(0,\infty]$ be the maximal time of existence and let $[\uvec,\pi,\varphi]$ be
a maximal strong solution to \eqref{Sy01}-\eqref{Sy04}, \eqref{sy5}, \eqref{sy6}
on $[0,T_m)$. By maximal strong solution we mean, by definition, that:
\begin{itemize}
\item  $[\uvec,\pi,\varphi]$ is a local in time
strong solution on $[0,T_m)$,
namely
\begin{description}
\item[(i)] $[\uvec,\pi,\varphi]$ satisfies
\begin{align*}
& \uvec\in L_{loc}^2([0,T_m);V_{div})\,,\\
& \pi\in L_{loc}^2([0,T_m);H^2(\Omega)\cap V_0)\,,\\
&\varphi \in L_{loc}^{\infty }([0,T_m);V)\cap L_{loc}^{2}([0,T_m);H^{2}(\Omega ))\cap H_{loc}^{1}([0,T_m);H)\,,
\end{align*}
\item[(ii)] $[\uvec,\pi,\varphi]$ is a strong solution to \eqref{Sy01}-\eqref{Sy04}, \eqref{sy5}, \eqref{sy6}
on $[0,t]$, for all $t\in(0,T_m)$\,;
\end{description}
\item  there is no strict extension $[\hat{\uvec},\hat{\pi},\hat{\varphi}]:[0,T_m^\prime]\to
V_{div}\times (H^2(\Omega)\cap V_0)\times H^2(\Omega)$, with $T_m^\prime>T_m$,
such that $[\hat{\uvec},\hat{\pi},\hat{\varphi}]$ is a local in time strong solution on $[0,T_m^\prime)$, i.e., such that
$[\hat{\uvec},\hat{\pi},\hat{\varphi}]$ satisfies (i)-(ii) with $T_m^\prime$ in place of $T_m$.
\end{itemize}

We recall that we are in the case $d=2$ or $\lambda$ constant (so that the mapping $\mathcal{F}$ is well defined).

We shall prove that $T_m=\infty$. By exploiting Step 1 and Step 2, we need to
derive some estimates for the norm of the maximal strong solution (similar to \eqref{e18}, \eqref{e20},  \eqref{e21})
containing constants on the right hand side which depend only on $t\in(0,T_m)$ (and on $\Vert\varphi_0\Vert_V$), and
which are bounded for $t\in(0,T_m)$. Notice that \eqref{e18}, \eqref{e20}, \eqref{e21},
cannot be used since the constants on the right hand sides depend on $R$, i.e., they depend on the norm of the solution itself.

Let us first consider estimate \sergiotwo{\eqref{e27}}, which it can be also written as
\begin{align}
&\Vert\uvec\Vert_{L^4(\Omega)^{\sergiotwo{d}}}\leq
C\big(\sergiotwo{\Vert B(\varphi)\Vert_{H^2(\Omega)}^{\frac{\theta}{2(1-\theta)}}+\Vert B(\varphi)\Vert_V^\theta+1}\big)\,.
\label{e12-bis}
\end{align}
To get \eqref{e12-bis} it is enough to write $\nabla\varphi=\nabla B(\varphi)/\lambda(\varphi)$
in the first term on the right hand side of the first inequality of \eqref{e3}, and then proceed as for
\sergiotwo{\eqref{e4}-\eqref{e27}.}
By combining \eqref{e12-bis} with \eqref{e14}, and exploiting the fact that now $\wph=\varphi$,
we get
$$\Vert B(\varphi)\Vert_{H^2(\Omega)}\leq C\,\big(\Vert\varphi_t\Vert
+\Vert B(\varphi)\Vert_{H^2(\Omega)}^{\sergiotwo{\frac{\theta}{1-\theta}}}+\sergiotwo{\Vert B(\varphi)\Vert_V^{2\theta}}
+\Vert \nabla B(\varphi)\Vert+1\big)\,,$$
which yields (recall that $\theta<1/2$)
\begin{align}
&\Vert B(\varphi)\Vert_{H^2(\Omega)}\leq C\,\big(\Vert\varphi_t\Vert+\Vert\nabla B(\varphi)\Vert+1\big)\,.
\label{e22}
\end{align}
From \eqref{e12-bis} and \eqref{e22} we then get
\begin{align*}
\Vert\uvec\Vert_{L^4(\Omega)^{\sergiotwo{d}}}&\leq
\sergiotwo{C\,\big(\Vert\varphi_t\Vert^{\frac{\theta}{2(1-\theta)}}
+\Vert\nabla B(\varphi)\Vert^{\frac{\theta}{2(1-\theta)}}+
\Vert B(\varphi)\Vert_V^\theta+1\big)}\\[1mm]
&\leq C\,\big(\Vert\varphi_t\Vert^\theta+\Vert\nabla B(\varphi)\Vert^\theta+1\big)\,,
\end{align*}
and inserting the above estimate into \eqref{e17} (where $\wph=\varphi$) we obtain
\begin{align}
&\frac{d\Phi}{dt}+\alpha_0\Vert\varphi_t\Vert^2\leq C\,\big(\Vert\varphi_t\Vert^{2\theta}
+\Vert\nabla B(\varphi)\Vert^{2\theta}+1\big)\Vert\nabla B(\varphi)\Vert
+C\,\big(\Vert\varphi_t\Vert^{4\theta}
+\Vert\nabla B(\varphi)\Vert^{4\theta}+1\big)\nonumber\\[1mm]
&\qquad\leq\frac{\alpha_0}{2}\Vert\varphi_t\Vert^2+ C\,\big(\Vert\nabla B(\varphi)\Vert^{\frac{1}{1-\theta}}
+\Vert\nabla B(\varphi)\Vert^{2\theta+1}+\Vert\nabla B(\varphi)\Vert^{4\theta}+\Vert\nabla B(\varphi)\Vert
+1\big)\nonumber\\[1mm]
&\qquad\leq\frac{\alpha_0}{2}\Vert\varphi_t\Vert^2+ C\,\big(\Vert\nabla B(\varphi)\Vert^2+1\big)\,.
\nonumber
\end{align}
Hence, in view also of \eqref{e15}, we deduce that
\begin{align}\label{e23}
&\frac{d\Phi}{dt}+\frac{\alpha_0}{2}\Vert\varphi_t\Vert^2\leq C\,(1+\Phi)\,.
\end{align}
From this differential inequality, by Gronwall's lemma, and arguing in the same fashion as for \eqref{e18}-\eqref{e21},
we can obtain the desired estimate for the ($\varphi$ component of the) maximal strong solution, namely
\begin{align}
&\Vert\varphi\Vert_{L^\infty(0,t;V)}+\Vert\varphi\Vert_{L^2(0,t;H^2(\Omega))}+
\Vert\varphi\Vert_{H^1(0,t;H)}\leq C\big(t,\Vert\varphi_0\Vert_V\big)\,.
\label{e24}
\end{align}
This inequality holds for all $0<t<T_m$, with a constant $C$ on the right hand side which depends only on $t$ and on $\Vert\varphi_0\Vert_V$,
and which is locally bounded with respect to $t$ on $[0,\infty)$.

Let us suppose now that $T_m<\infty$ and consider, for simplicity, just the
$\varphi$ component of the maximal strong solution.
Observe that the constant $C$ on the right hand side of \eqref{e24} can be bounded
by $C\big(T_m,\Vert\varphi_0\Vert_V)$, and \eqref{e24} holds for all $0<t<T_m$. Thus we
deduce that $\varphi\in C([0,T_m];V)$ (cf. Remark \ref{continuity}).
This allows us to restart the system by taking $\varphi(T_m)$ as new initial datum, in place of $\varphi_0$ in
\eqref{sy6}. By applying again the Schauder's fixed point argument (see Steps 1, 2, and 3), we can then construct
a new local in time strong solution which is defined on an interval of the form $(T_m,T_m+\delta)$, for some $\delta>0$.
By means of the local in time solution on $(T_m,T_m+\delta)$ we can then define a strict extension of $\varphi$
on $[0,T_m+\delta)$, which is still a strong solution to \eqref{Sy01}-\eqref{Sy04}, \eqref{sy5}, \eqref{sy6}.
This contradicts the maximality of $\varphi$, and concludes the proof
for the case $d=2$ or $\lambda$ constant.

We are left to prove the theorem in the case $d=3$ and $\lambda$ non-constant.

We know that, in this case, uniqueness of the strong solution $\wph\in Y_T$ (with
$|\wph|\leq 1$) to problem \eqref{pb3}, \eqref{pb4}$_2$, \eqref{pb5}, with
$\uvec$ given in $L^4(0,T;L^4(\Omega)^3)$ is not known (see Theorem \ref{reg-thm-bis}).
However, Theorem \ref{reg-thm-bis} entails uniqueness of $\wph$
provided that the velocity field in the convective term of the nonlocal Cahn-Hilliard is divergence-free
and has an $L^2(0,T;L^\infty(\Omega)^3)-$regularity . We thus replace $\uvec$ in \eqref{pb3} by a suitable
regularization $\vvec$. A convenient choice turns out to be a Leray-$\alpha$ type regularization (see, for instance, \cite{CHOT}).
More precisely, in place of problem \eqref{pb1}-\eqref{pb5} we now address the following system
\begin{align}
& \eta(\varphi)\mathbf{u} + \nabla\pi=(\nabla J\ast\varphi)\varphi, \qquad  \mbox{ in } Q_{T}, \label{pb1-3d}\\
& \mbox{div}(\mathbf{u})=0, \qquad  \mbox{ in } Q_{T},\label{pb2-3d}\\
& \widetilde{\varphi}_t+\vvec\cdot\nabla\widetilde{\varphi}=\Delta B(\widetilde{\varphi})
- \mbox{div}(m(\widetilde{\varphi})(\nabla J\ast\widetilde{\varphi})), \qquad  \mbox{ in } Q_{T},\label{pb3-3d}\\
&(I+\alpha\, S)\,\vvec=\uvec, \qquad  \mbox{ in } Q_{T}\label{pb4-3d}\\
&\mathbf{u}\cdot\mathbf{n} = 0\,, \quad
\big[\nabla B(\widetilde{\varphi})-m(\widetilde{\varphi})(\nabla J\ast\widetilde{\varphi})\big]\cdot\mathbf{n} = 0
\qquad\mbox{on }\Gamma\times (0,T),\label{pb5-3d}\\
& \widetilde{\varphi}(0)=\varphi_0 \qquad\mbox{in }\Omega\,,\label{pb6-3d}
\end{align}
where $S$ is the Stokes operator with no-slip boundary condition (see, for instance, \cite[Chap.5]{BF}) and
$\alpha>0$ is a fixed regularization parameter.
 In order to reproduce the Schauder fixed point argument also for system
\eqref{pb1-3d}-\eqref{pb6-3d}, we need to control the $L^4(\Omega)^3-$norm of $\vvec$ by the $L^4(\Omega)^3-$norm of
$\uvec$, uniformly with respect to $\alpha$. This crucial control can be achieved by applying
a well-known result on the resolvent estimates in $L^p$ for the Stokes operator (with no-slip boundary condition)
in sufficiently smooth domains (e.g., of class $C^2$) which, for the reader's convenience, we report here below in the form suitable for our purposes (see \cite[Theorem 1]{Giga1,Giga2}).
\begin{lem}\label{Giga}
Let $\Omega$ be a bounded and smooth domain in $\mathbb{R}^d$, $d\geq 2$. 
If $\hvec\in L^p_{div}(\Omega)^d$ is given for some $p > 2$ then there exists a constant $C_p>0$
such that   the unique solution $\vvec \in D(S)=V^p_{0,div}(\Omega)\cap W^{2,p}(\Omega)^d$  to
$(I+\alpha\, S)\,\vvec=\hvec $
satisfies the estimate
\begin{align}
&\Vert\vvec\Vert_{L^p(\Omega)^d}\leq C_p\,\Vert\hvec\Vert_{L^p(\Omega)^d}\label{e177}
\end{align}
and $C_p$ is independent of $\hvec$ and $\alpha>0$.
\end{lem}

We are now ready to adapt the argument developed in Steps 1-4 to
\eqref{pb1-3d}-\eqref{pb6-3d} in order to establish existence of a strong solution to this system
for every fixed $\alpha>0$.
First, we point out that all estimates deduced in the previous steps also hold in the present case.
We thus consider problem \eqref{pb3-3d}, \eqref{pb4-3d}, \eqref{pb5-3d}$_2$, \eqref{pb6-3d} where
$\uvec$ is the (second component of the) unique solution $\varphi\in Y_T$ to \eqref{pb1-3d}, \eqref{pb2-3d}, \eqref{pb5-3d}$_1$
such that $|\varphi|\leq 1$.
Thanks to the fact that $\Vert\vvec\Vert_{D(S)}\leq C_\alpha \Vert\uvec\Vert$, which yields that
$\vvec\in L^2(0,T;L^\infty(\Omega)^3\cap V_{0,div})$ (actually, we have also a better
time regularity for $\vvec$, e.g., $\vvec\in L^{4(1-\theta)/\theta}(0,T;L^\infty(\Omega)^3\cap V_{0,div})$,
cf. \eqref{e27}),
and thanks to Theorem \ref{reg-thm-bis}, we know that there exists a unique solution $\wph\in Y_T$ to
\eqref{pb3-3d}, \eqref{pb4-3d}, \eqref{pb5-3d}$_2$, \eqref{pb6-3d} such that $|\wph|\leq 1$.
Therefore, the map $\mathcal{F}_\alpha:\varphi\mapsto\wph$ is still well defined from the set $\{\psi\in Y_T:|\psi|\leq 1\}$ into itself. \\

To proceed as in Step 2, we need a uniform (with respect to $\alpha$) control for
the $L^4(0,T;L^4(\Omega)^{3})-$ norm of $\vvec$ in terms of the same norm of $\uvec$.
This control can be achieved by applying Lemma \ref{Giga} to \eqref{pb4-3d} with $p=4$, namely
\begin{align*}
&\Vert\vvec\Vert_{L^4(0,T;L^4(\Omega)^{3})}\leq
C_4\Vert\uvec\Vert_{L^4(0,T;L^4(\Omega)^{3})}\,.
\end{align*}
By combining this estimate with \eqref{e27}, we get
\begin{align}\label{e28}
&\Vert\vvec\Vert_{L^4(0,T;L^4(\Omega)^{3})}
\leq C\big( R^{\frac{\theta}{2(1-\theta)}}\,T^{\frac{1-2\theta}{4(1-\theta)}}+R^\theta\, T^{1/4}+T^{1/4}\big)\,,
\end{align}
where $C>0$ is independent of $\alpha$. \\
Observe now that the differential inequality \eqref{e13},
with $\vvec$ in place of $\uvec$, still holds true and, by employing \eqref{e28} into this inequality, we can argue
exactly in the same fashion as in Step 2. Hence, we can still conclude that the map $\mathcal{F}_\alpha$
takes $B_{Y_T}(R)$ into itself, with $R>0$ and $T>0$ chosen as in Step 2 (independently of $\alpha$).

Concerning  the continuity of the map $\mathcal{F}_\alpha$ in the strong topology of $X_T$ (cf. Step 3),
the only modification is related to the map $\mathcal{G}=\mathcal{G}_\alpha$, which is now defined as
the map that to each $\uvec\in L^4(0,T;L^4_{div}(\Omega)^3)$ associates
$\widetilde{\varphi}=\mathcal{G}_\alpha(\uvec)$, where $\widetilde{\varphi}\in Y_T$ is the unique strong solution to
\eqref{pb3-3d}, \eqref{pb4-3d}, \eqref{pb5-3d}$_2$, \eqref{pb6-3d} given by Theorem \ref{reg-thm-bis}. Keeping the same
notation used in Step 3 and setting $\widetilde{\varphi}_n:=\mathcal{G}_\alpha(\uvec_n)
=(\mathcal{G}_\alpha\circ\mathcal{Q}_2)(\varphi_n)=\mathcal{F}(\varphi_n)$, from Step 2 we deduce again that
$\{\widetilde{\varphi}_n\}\subset B_{Y_T}(R)$. Hence, we have (up to a subsequence)
that $\widetilde{\varphi}_n\rightharpoonup \varphi^\ast$, weakly star in $Y_T$.
Let us now write the weak formulation of \eqref{pb3-3d}, \eqref{pb4-3d}, \eqref{pb5-3d}$_2$, \eqref{pb6-3d}
for $\widetilde{\varphi}_n$,
with $\wph_n$, $\vvec_n$, $\uvec_n$ in place of $\wph$, $\vvec$, $\uvec$, respectively. We have
\begin{align}
&\langle\widetilde{\varphi}_{n,t},\psi\rangle_V+(\nabla B(\wph_n),\nabla\psi)=
(\vvec_n\wph_n,\nabla\psi)+(m(\wph_n)(\nabla J\ast\wph_n),\nabla\psi)\,,\quad\forall\psi\in V\,,
\label{e30}\\
&\vvec_n=(I+\alpha\,S)^{-1}\uvec_n\,.\label{e31}
\end{align}Lemma \ref{Giga} yields that $(I+\alpha\,S)^{-1}\in\mathcal{L}(L^{10/3}(0,T;L^{10/3}(\Omega)^3)),L^{10/3}(0,T;L^{10/3}(\Omega)^3))$. Hence, from \eqref{e29} we deduce that
\begin{align*}
&\vvec_n \rightharpoonup\overline{\vvec}\,,\qquad\mbox{ weakly in }L^{10/3}(0,T;L^{10/3}(\Omega)^3)\,,
\end{align*}
where $\overline{\vvec}=(I+\alpha\,S)^{-1}\overline{\uvec}$. By means of this weak convergence and
on account of the weak/strong convergences for $\wph_n$ (see Step 3), we can then pass to the limit
in \eqref{e30}, \eqref{e31} and deduce that
the weak formulation of \eqref{pb3-3d}, \eqref{pb4-3d}, \eqref{pb5-3d}$_2$, \eqref{pb6-3d}
is satisfied also for $\varphi^\ast$, with $\overline{\vvec}$ and $\overline{\uvec}$ in place of $\vvec$ and $\uvec$, respectively.
By again invoking the uniqueness
of the strong solution to problem
\eqref{pb3-3d}, \eqref{pb4-3d}, \eqref{pb5-3d}$_2$, \eqref{pb6-3d} (with $\uvec=\overline{\uvec}$ given),
ensured by Theorem \ref{reg-thm-bis}, we have that
$\varphi^\ast=\mathcal{G}_\alpha(\overline{\uvec})=(\mathcal{G}_\alpha\circ\mathcal{Q}_2)(\overline\varphi)=\mathcal{F}(\overline\varphi)$.
The continuity of $\mathcal{F}_\alpha$ in the strong topology of $X_T$ then follows as in Step 3.

Schauder's fixed point theorem can be again applied to the map $\mathcal{F}_\alpha: B_{Y_T}(R)\to B_{Y_T}(R)$, as well as estimates
\eqref{e6} and \eqref{e7}. This yields the existence of a local in time strong solution $[\uvec,\vvec,\pi,\varphi]$
to \eqref{pb1-3d}-\eqref{pb6-3d} such that \eqref{str-reg-1}-\eqref{str-reg-3} hold.
This local in time strong solution can then be extended to an arbitrary time interval $[0,T]$, for all $T>0$,
by arguing exactly as in Step 4.

We have thus shown that, for every fixed $\alpha>0$, system \eqref{pb1-3d}-\eqref{pb6-3d} admits a global in time strong solution
$[\uvec^\alpha,\vvec^\alpha,\pi^\alpha,\varphi^\alpha]$ satisfying \eqref{str-reg-1}-\eqref{str-reg-3}.

We now need to recover suitable bounds for $[\uvec^\alpha,\vvec^\alpha,\pi^\alpha,\varphi^\alpha]$ which are uniform with respect to $\alpha$
in order to pass to the limit in \eqref{pb1-3d}-\eqref{pb6-3d} as $\alpha\to 0$.
These bounds can be obtained by observing that all constants in the estimates derived in the former Steps 1 to 4
are independent of $\alpha$. In particular, \eqref{e24} is satisfied also for $\varphi^\alpha$
yielding (up to a subsequence)
\begin{align}
&\varphi^\alpha\rightharpoonup\hat{\varphi}\,,\qquad\mbox{weakly star in }\,\,L^\infty(0,T;V)\cap L^2(0,T;H^2(\Omega))\cap H^1(0,T;H)\,,
\label{e35}\\
&\qquad\qquad\qquad\mbox{strongly in }\,\,C([0,T];L^{6^{-}}(\Omega))\,,\,\mbox{ and pointwise a.e. in }\,\,Q_T\,.
\label{e36}
\end{align}
On the other hand, from \eqref{e27}, \eqref{e7}, and \eqref{e24} we also have that
\begin{align}
&\uvec^\alpha\rightharpoonup\hat{\uvec}\,,\qquad\mbox{weakly in}\,\,
L^{4(1-\theta)/\theta}(0,T;L_{div}^4(\Omega)^3)\cap L^{4(1-\theta)}(0,T;V_{div})\,,\label{e32}
\end{align}
and, thanks to Lemma \ref{Giga} applied to $(I+\alpha\,S)^{-1}\uvec^\alpha=\vvec^\alpha$, we get
\begin{align}
&\vvec^\alpha\rightharpoonup\hat{\vvec}\,,\qquad\mbox{weakly in}\,\,
L^{4(1-\theta)/\theta}(0,T;L_{div}^4(\Omega)^3)\,.\label{e34}
\end{align}
Moreover, from \eqref{e6}, and \eqref{e24} there follows that
\begin{align}
&\pi^\alpha\rightharpoonup\hat{\pi}\,,\qquad\mbox{weakly in}\,\,
L^{4(1-\theta)}(0,T;H^2(\Omega)\cap V_0)\,.\label{e37}
\end{align}
It is easy to see that $\hat{\vvec}=\hat{\uvec}$. Indeed, setting $\mathcal{J}_\alpha:=(I+\alpha\,S)^{-1}$,
and observing that $\mathcal{J}_\alpha$ is self-adjoint, then, for every $\wvec\in L^2(Q_T)^3$, we have that
\begin{align}\label{e33}
&\int_{0}^{T}(\vvec^\alpha,\wvec)\,dt=\int_{0}^{T}(\mathcal{J}_\alpha\uvec^\alpha,\wvec)\,dt
=\int_{0}^{T}(\uvec^\alpha,\mathcal{J}_\alpha\wvec)\,dt\to\int_{0}^{T}(\hat{\uvec},\wvec)\,dt\,,
\end{align}
where we have used \eqref{e32} and the fact that $\mathcal{J}_\alpha\wvec\to\wvec$, strongly in $L^2(0,T;L^2(\Omega)^3)$.
This strong convergence follows from the general properties of the resolvent operator $\mathcal{J}_\alpha$
of the maximal monotone (linear) map $S$, namely, $\mathcal{J}_\alpha\wvec(t)\to\wvec(t)$, strongly in $L^2(\Omega)^3$,
for almost any $t\in (0,T)$ and $\Vert\mathcal{J}_\alpha\wvec(t)\Vert\leq\Vert\wvec(t)\Vert$, for all $\alpha>0$
(using also Lebesgue's theorem). Therefore, \eqref{e33} gives $\vvec^\alpha\rightharpoonup\hat{\uvec}$
in $L^2(0,T;L^2(\Omega)^3)$. Thus we deduce (see \eqref{e34})  $\hat{\vvec}=\hat{\uvec}$.

By means of \eqref{e35}-\eqref{e37} (with $\hat{\vvec}=\hat{\uvec}$),
a standard argument allows us to pass to the limit in system
\eqref{pb1-3d}-\eqref{pb6-3d} as $\alpha\to 0$ (up to a subsequence)
and find that $[\hat{\uvec},\hat{\pi},\hat{\varphi}]$ is a strong solution
to \eqref{Sy01}-\eqref{Sy04}, \eqref{sy5}, \eqref{sy6} satisfying \eqref{str-reg-1}-\eqref{str-reg-3}.
Once we have a strong solution, then it is easy to show that $\pi \in L^\infty(0,T;C^\alpha(\overline{\Omega}))$ for
some $\alpha\in (0,1)$ (see \eqref{e53}).
This concludes the proof of Theorem \ref{strong-sols}.

\subsection{Proof of Theorem \ref{strong-sols-2d}}
We proceed formally, for the sake of brevity. The argument below can be made rigorous by means of a Faedo-Galerkin scheme.
Indeed, only the time derivative of $\uvec$ and of $\varphi$ will be used as test functions. Alternatively,
a time discretization procedure can be used (see \cite[Proof of Theorem 3.6]{FGGS}).

To begin with, we take the time derivative of the Darcy's law \eqref{e26} and multiply the resulting identity by
$\uvec_t$. We get
\begin{align}\label{e38}
&\big(\eta(\varphi)\uvec_t,\uvec_t\big)+\big(\eta^\prime(\varphi)\varphi_t\,\uvec,\uvec_t\big)
=\big((\nabla J\ast\varphi_t)\varphi,\uvec_t\big)+
\big((\nabla J\ast\varphi)\varphi_t,\uvec_t\big)\,.
\end{align}
Setting $\eta_\infty^\prime:=\Vert\eta^\prime\Vert_{L^\infty(-1,1)}$,
the second term on the left hand side of \eqref{e38} can be estimated as follows
\begin{align}\label{e39}
\big|\big(\eta^\prime(\varphi)\varphi_t\,\uvec,\uvec_t\big)\big|&
\leq \eta_\infty^\prime\Vert\varphi_t\Vert_{L^4(\Omega)}\Vert\uvec\Vert_{L^4(\Omega)^2}\Vert\uvec_t\Vert\notag\\[1mm]
&\leq C\eta_\infty^\prime\,(\Vert\varphi_t\Vert+\Vert\varphi_t\Vert^{1/2}\Vert\nabla\varphi_t\Vert^{1/2})
\Vert\uvec\Vert_{L^4(\Omega)^2}\Vert\uvec_t\Vert\notag\\[1mm]
&\leq\delta\Vert\uvec_t\Vert^2+C_\delta\,{\eta_\infty^\prime}^{\!\!\!2}\,
(\Vert\varphi_t\Vert^2+\Vert\varphi_t\Vert\Vert\nabla\varphi_t\Vert)\,\Vert\uvec\Vert_{L^4(\Omega)^2}^2\notag\\[1mm]
&\leq \delta\Vert\uvec_t\Vert^2+\delta^\prime\Vert\nabla\varphi_t\Vert^2
+C_{\delta,\delta^\prime}\big({\eta_\infty^\prime}^{\!\!\!2}\Vert\uvec\Vert_{L^4(\Omega)^2}^2
+{\eta_\infty^\prime}^{\!\!\!4}\Vert\uvec\Vert_{L^4(\Omega)^2}^4\big)\Vert\varphi_t\Vert^2\,.
\end{align}
The estimates of the two terms on the right hand side of \eqref{e38} being straightforward, we can
then insert \eqref{e39} into \eqref{e38}, use assumptions
\textbf{(H1)}, \textbf{(H2)}, and take $\delta$ suitably small to obtain
\begin{align}\label{e40}
&\frac{\eta_1}{2}\Vert\uvec_t\Vert^2\leq\delta^\prime\Vert\nabla\varphi_t\Vert^2
+C_\delta \big(1+{\eta_\infty^\prime}^{\!\!\!2}\Vert\uvec\Vert_{L^4(\Omega)^2}^2
+{\eta_\infty^\prime}^{\!\!\!4}\Vert\uvec\Vert_{L^4(\Omega)^2}^4\big)\Vert\varphi_t\Vert^2\,.
\end{align}
Next, we take the time derivative of \eqref{e25} and test the resulting equation by $\varphi_t$ to get
\begin{align}
& \frac{1}{2}\frac{d}{dt}\Vert \varphi _{t}\Vert ^{2}+(\nabla B(\varphi
)_{t},\nabla \varphi _{t})  \notag \\
& =-\left( \uvec_{t}\cdot \nabla \varphi ,\varphi _{t}\right)
+\left( m^{\,\prime }(\varphi )\varphi _{t}\left( \nabla J\ast \varphi
\right) ,\nabla \varphi _{t}\right) +\left( m(\varphi )\left( \nabla J\ast
\varphi _{t}\right) ,\nabla \varphi _{t}\right) \,.  \label{e42}
\end{align}
In order to estimate the second term on the left side and the first one on the right
(the estimate of the last two terms on the right being straightforward), we can argue exactly as in the proof of \cite[Proposition 5.1]{FGGS}.
Indeed, on account of \eqref{e22}, of the
$L^\infty(0,T;V)$ bound for $\varphi$ (cf. \eqref{str-reg-3}), and of the Gagliardo-Nirenberg inequality in two dimensions,
we deduce the following differential inequality
\begin{align}\label{e41}
&\frac{d}{dt}\Vert \varphi _{t}\Vert ^{2}+\frac{\alpha _{0}}{4}\Vert \nabla
\varphi _{t}\Vert ^{2}\leq C\left( \Vert \varphi _{t}\Vert ^{4}+\Vert
\varphi _{t}\Vert ^{2}+\Vert \uvec_{t}\Vert ^{2}+1\right) \,.
\end{align}
By means of \eqref{e40}, taking $\delta^\prime$ small enough (i.e., $\delta^\prime\leq\alpha_0\eta_1/16\,C$),
from \eqref{e41} we infer
\begin{align*}
&\frac{d}{dt}\Vert \varphi _{t}\Vert ^{2}+\frac{\alpha _{0}}{8}\Vert \nabla
\varphi _{t}\Vert ^{2}\leq C\big(\Vert \varphi _{t}\Vert ^{4}+1\big)
+C\big(1+{\eta_\infty^\prime}^{\!\!\!2}\Vert\uvec\Vert_{L^4(\Omega)^2}^2
+{\eta_\infty^\prime}^{\!\!\!4}\Vert\uvec\Vert_{L^4(\Omega)^2}^4\big)\Vert\varphi_t\Vert^2\,.
\end{align*}
From this differential inequality, Gronwall's lemma and \eqref{str-reg-1} entail that $\varphi\in H^1(0,T;V)\cap W^{1,\infty}(0,T;H)$. 
The $L^\infty(0,T;H^2(\Omega))-$regularity for $\varphi$ follows as in the proof of \cite[Proposition 5.1]{FGGS},
by using \eqref{e22}, the fact that $\varphi_t\in L^\infty(0,T;H)$, implying $B(\varphi)\in L^\infty(0,T;H^2(\Omega))$,
and identity \eqref{e43}.

Once \eqref{str-reg-2-2d} is established, \eqref{str-reg-1-2d} follows from \eqref{e40}, by taking
\eqref{str-reg-1} into account. The proof is finished.

\section{Further regularity properties for $\pi$ and $\uvec$}
\setcounter{equation}{0}
\label{sec:regpiu}

The goal of this section is to develop a detailed analysis of the regularity properties
of the pressure and velocity fields of the strong solution derived in Theorem \ref{strong-sols}.
While in Theorem \ref{strong-sols} our main objective was just to rigorously establish existence of a
strong solution in some suitable regularity class, our main focus here is to address
more closely the regularity of $\pi$ and $\uvec$ that stems from the elliptic system
satisfied by the pressure field, as a consequence of the validity of the Darcy's law
(the regularity for $\varphi$ is essentially determined by the nonlocal Cahn-Hilliard structure, and it will always
be taken as given by \eqref{str-reg-3} in all this section).
This goal is achieved by applying elliptic regularity results to
problem \eqref{ellipt-nondiv}, \eqref{pb2-eq} with $\varphi$ satisfying \eqref{str-reg-3},
and making a careful use of suitable Gagliardo-Nirenberg-Sobolev interpolation inequalities (cf. Proposition \ref{Bre-Mir})
to gain, in particular, a $W^{2,p}(\Omega)-$regularity for $\pi$, for all $1<p<\infty$.
Concerning the time regularity, the delicate point and our main effort are to obtain ``optimal'' time integrability
exponent for $\pi$ with values, e.g., in $W^{2,p}(\Omega)$. Indeed, to the best of our knowledge,
there are no results in the literature that allow us to obtain such an optimal value for this exponent,
once the space-time regularity \eqref{str-reg-3} for $\varphi$ is assumed in equation \eqref{ellipt-nondiv}.

By comparing with other arguments and
with other ways to estimate in $L^p(\Omega)$ the principal term in the elliptic equation
for $\pi$ (see $\mathcal{F}_1$ below) by means of H\"{o}lder and Gagliardo-Nirenberg inequalities,
it turns out that the best exponents seem to be reached by suitably exploiting the H\"{o}lder continuity
property for $\pi$ (the question whether the exponents thus obtained are optimal or not
is, however, still open). Therefore, the H\"{o}lder continuity of $\pi$, which
revealed itself to be helpful to prove existence of a strong solution in Theorem \ref{strong-sols},
here plays a major role, meaning that, differently from the proof of Theorem \ref{strong-sols},
the value of the H\"{o}lder continuity exponent $\alpha\in(0,1)$ of $\pi$
is now crucial. Indeed, the time integrability exponents for $\pi$ will be expressed in terms
of $\alpha$. In the sequel the time dependence will be generally omitted for the sake of simplicity.

We point out that, recalling that $\pi$ satisfies the elliptic system
\begin{align}
&\mbox{div}\Big(\frac{1}{\eta(\varphi)}\nabla\pi\Big)=\mbox{div}\Big(\frac{(\nabla J\ast\varphi)\varphi}{\eta(\varphi)}\Big)\,,
\label{pb1-eq-bis}\\[1mm]
&\frac{\partial\pi}{\partial\mathbf{n}}=(\nabla J\ast\varphi)\varphi\cdot\mathbf{n}\,,\label{pb2-eq-bis}
\end{align}
with $\varphi$ satisfying \eqref{str-reg-3}, the well known De Giorgi's result (see, e.g., \cite{DiB})
ensures that $\pi\in C^\alpha(\overline{\Omega})$, with some $\alpha\in(0,1)$ and some
$\Vert\pi\Vert_{C^\alpha(\overline{\Omega})}$, which only depend on $\eta_1,\eta_\infty, b, d,\Omega$, and on the
geometrical properties of $\Gamma$ (cf. \eqref{e53}). Therefore, the exponent $\alpha$ and the norm
 $\Vert\pi\Vert_{C^\alpha(\overline{\Omega})}$ depend (or can be bounded by constants that depend)
 on structural parameters only (which are a priori known), and may be considered independent of the form of
 the $\varphi$-component of the strong solution (which is not a priori known).

We also observe that, in addition to providing a rather complete
picture of the regularity properties of $\pi$ and $\uvec$ for the strong solution of Theorem \ref{strong-sols},
the analysis of this section (especially in the case $d=3$)
and the effort in achieving the best time integrability exponents have another important
motivation. Indeed, these
 properties will be used in Section \ref{sec:unique} to prove weak-strong uniqueness results
 for the case of non-constant viscosity $\eta$. In particular, for $d=3$
 a conditional type result will be proven and the condition will depend on $\alpha$
 (hence, on an essentially structural constant, see the discussion above). This condition will allow us to
 guarantee a required regularity for the velocity field of one of the two solutions.
 Therefore, the higher the time integrability exponent for $\pi$ (with $\alpha$ given), the weaker the assumption
 on $\alpha$, namely, the smaller the lower bound for $\alpha$ ensuring weak-strong uniqueness
 with non-constant $\eta$ in dimension three will be.

\bigskip

We can now state the main result of this section.

\begin{thm} \label{str-sols-fur-reg}
Let all assumptions of Theorem \ref{strong-sols} be satisfied.
Then, for every $T>0$, the $\pi$ and $\uvec$ components of
the strong solution $[\uvec,\pi,\varphi]$ to problem \eqref{Sy01}-\eqref{Sy04}, \eqref{sy5}, \eqref{sy6},
in addition to \eqref{str-reg-1}-\eqref{str-reg-3}, satisfy the following regularity properties,
where $\alpha\in(0,1)$ is the H\"{o}lder continuity exponent of $\pi$.
\begin{itemize}
  \item If $d=2$ then we have that
  \begin{align}
       &\pi\in L^{\sigma_p}(0,T;W^{2,p}(\Omega))\cap L^{\hat{\sigma}_p}(0,T;W^{1,q}(\Omega))\,,\label{e175}\\[1mm]
       &\uvec\in L^{\sigma_p}(0,T;W^{1,p}(\Omega)^2)\cap L^{\hat{\sigma}_p}(0,T;L^{q}(\Omega)^2)\,,\label{e179}
  \end{align}
  with $\sigma_p$, $\hat{\sigma}_p$ and $q$ given according with the following cases
    \begin{description}
      \item[(i)] if $\,\,2\leq p<\infty$ then
       \begin{align}
       &\sigma_p:=\Big(\frac{2p}{(2-\alpha)p-2}\Big)^-\,,\quad
       \hat{\sigma}_p:=\Big(\frac{2q}{(1-\alpha)q-2}\Big)^-\,,\quad p\,\frac{2-\alpha}{1-\alpha}\leq q<\infty\,;
       \label{e176}
       \end{align}
      \item[(ii)] if $\,\,1<p<2$ then
      \begin{align}
      &\sigma_p=\left\{\begin{array}{ll}
      \left(\frac{p^\prime}{(1-\alpha)^2}\right)^-\,,\quad 1<p\leq p_\alpha:=\frac{2(2-\alpha)}{3-2\alpha}\,,\\[3mm]
      \left(\frac{2p}{(2-\alpha)p-2}\right)^-\,,\quad p_\alpha\leq p<2\,,
      \end{array}\right.
      \end{align}
      \begin{align}
      &\hat{\sigma}_p:=\frac{1}{(1-\alpha)^2}\,\Big(\frac{2q}{q-2}\Big)^-\,,
      \quad 2< q\leq 2\,\frac{2-\alpha}{1-\alpha}\,.
      \end{align}
   \end{description}
    If, in addition, $J\in W^{3,1}_{loc}(\mathbb{R}^2)$ then
\begin{align}
&\pi\in L^{\sigma_\infty}(0,T;H^3(\Omega))\,,\quad
\uvec\in L^{\sigma_\infty}(0,T;H^2(\Omega)^2)\,,\qquad\sigma_\infty:=\Big(\frac{2}{2-\alpha}\Big)^-\,.
\end{align}

  \item If $d=3$ then we have that
  \begin{align}
       &\pi\in L^{\mu_p}(0,T;W^{2,p}(\Omega))\cap L^{\hat{\mu}_p}(0,T;W^{1,q}(\Omega))\,,\label{e180}\\[1mm]
       &\uvec\in L^{\mu_p}(0,T;W^{1,p}(\Omega)^3)\cap L^{\hat{\mu}_p}(0,T;L^{q}(\Omega)^3)\,,\label{e165}
  \end{align}
  with $\mu_p$, $\hat{\mu}_p$ and $q$ given according with the following cases
  \begin{description}
    \item[(i)] if $\,\,2\leq p<3$ then
    \begin{align}
      &\mu_p=\left\{\begin{array}{ll}
      \Big(\frac{p}{(2-\alpha)p-3}\Big)^-\,,\quad 0< \alpha\leq\frac{2p-4}{p}\,,\\[2mm]
      \Big(\frac{2p}{(2-\alpha)p-2}\Big)^-\,,\quad\frac{2p-4}{p}\leq \alpha\leq\frac{2p-3}{p}\,,\\[2mm]
      \Big(\frac{6}{2-\alpha}\Big)^-\,,\quad\frac{2p-3}{p}<\alpha<1\,,
       \end{array}\right.
    \end{align}
    \begin{align}\label{e166}
      &\hat{\mu}_p=\left\{\begin{array}{ll}
      \Big(\frac{q}{(1-\alpha)q-3}\Big)^-\,,\quad p\,\frac{2-\alpha}{1-\alpha}\leq q\leq\frac{4p}{4-p}\,,
      \quad 0< \alpha\leq\frac{2p-4}{p}\,,\\[2mm]
      \Big(\frac{(2-\alpha)p-3}{(2-\alpha)p-2}\Big)^-\,\frac{2q}{(1-\alpha)q-3}\,,
      \quad p\,\frac{2-\alpha}{1-\alpha}\leq q\leq\frac{3p}{3-p}\,,\quad\frac{2p-4}{p}\leq \alpha\leq\frac{2p-3}{p}\,,\\[2mm]
      \Big(\frac{6}{1-\alpha}\Big)^-\,,\quad q=\frac{3p}{3-p}\,,\quad \frac{2p-3}{p} < \alpha<1\,;
      \end{array}\right.
    \end{align}
    \item[(ii)] if $\,\,3\leq p<4$ then
    \begin{align}
      &\mu_p=\left\{\begin{array}{ll}
      \Big(\frac{p}{(2-\alpha)p-3}\Big)^-\,,\quad \frac{p-3}{p}< \alpha\leq\frac{2p-4}{p}\,,\\[2mm]
      \Big(\frac{2p}{(2-\alpha)p-2}\Big)^-\,,\quad\frac{2p-4}{p}\leq \alpha<1\,,
       \end{array}\right.
    \end{align}
    \begin{align}
      &\hat{\mu}_p=\left\{\begin{array}{ll}
      \Big(\frac{q}{(1-\alpha)q-3}\Big)^-\,,\quad p\,\frac{2-\alpha}{1-\alpha}\leq q<\infty\,,
      \quad \frac{p-3}{p}< \alpha\leq\frac{2p-4}{p}\,,\\[2mm]
      \Big(\frac{(2-\alpha)p-3}{(2-\alpha)p-2}\Big)^-\,\frac{2q}{(1-\alpha)q-3}\,,
      \quad p\,\frac{2-\alpha}{1-\alpha}\leq q<\infty\,,\quad\frac{2p-4}{p}\leq \alpha<1\,;
       \end{array}\right.
    \end{align}
    \item[(iii)] if $\,\,4\leq p<6$, then
     \begin{align}
     &\mu_p=\Big(\frac{p}{(2-\alpha)p-3}\Big)^-\,,\quad\frac{p-3}{p}<\alpha<1\,,
     \end{align}
     \begin{align}
      &\hat{\mu}_p=\left\{\begin{array}{ll}
      \frac{q}{(1-\alpha)q-3}\,,\quad\frac{6p}{6-p}\leq q<\infty\,,\quad\frac{p-3}{p}<\alpha\leq\frac{2(p-3)}{p}\\[2mm]
      \frac{q}{(1-\alpha)q-3}\,,\quad p\,\frac{2-\alpha}{1-\alpha}\leq q<\infty\,,\quad\frac{2(p-3)}{p}\leq \alpha<1\,;
       \end{array}\right.
    \end{align}
    \item[(iv)] if $\,\,p=6$, then
    \begin{align}
    &\mu_6=\Big(\frac{2}{3-2\alpha}\Big)^-\,,\quad\hat{\mu}_6=\Big(\frac{1}{1-\alpha}\Big)^-\,,\quad
    q=\infty\,,\quad\frac{1}{2}<\alpha<1\,.
     \end{align}
    \end{description}
\end{itemize}

Finally, if $\eta$ is a positive constant and $J\in W_{loc}^{2,1}(\mathbb{R}^{d})$ or $J$ is
admissible, we have that \eqref{e175}, \eqref{e179}, \eqref{e180},
\eqref{e165} hold with $\sigma_p$, $\hat{\sigma}_p$, $\mu_p$, $\hat{\mu}_p$, and $q$ given according
with the following cases.
\begin{itemize}
  \item If $d=2$ then
  \begin{align}
  &\sigma_p=\frac{2p}{p-2}\,,\qquad 2\leq p<\infty\,,\label{e181}\\[1mm]
  &\hat{\sigma}_p=\infty\,,\quad 2\leq q<\infty\,,\quad\mbox{ if }\,\,p=2\,;\quad
  2\leq\hat{\sigma}_p<\infty\,,\quad  q=\infty\,,\quad\mbox{ if }\,\,3\leq p<\infty\,.
  \label{e182}
  \end{align}
\item If $d=3$ then
\begin{align}
      &\mu_p=\hat{\mu}_p=\left\{\begin{array}{lll}
      \frac{2p}{p-2}\,,\qquad 2\leq p\leq 4\,,\\[2mm]
      \frac{p}{p-3}\,,\qquad 4\leq p\leq 6
       \end{array}\right.\,,\quad
       q\,\left\{\begin{array}{lll}
      =\frac{3p}{3-p}\,,\quad 2\leq p<3\,,\\[2mm]
      \in[2,\infty)\,,\quad p=3\,,\\[2mm]
      =\infty\,,\quad 2<p\leq 6\,.
       \end{array}\right.
       \label{e183}
\end{align}
\end{itemize}
\end{thm}

\begin{oss}
For the sake of simplicity we have not reported the cases $1<p<2$ and $J\in W^{3,1}_{loc}(\mathbb{R}^3)$ when $d=3$.
\end{oss}

\begin{proof}
It is convenient to rewrite \eqref{ellipt-nondiv} and \eqref{pb2-eq} in the following form
\begin{align}
\Delta\pi&=\mathcal{F}(\varphi,\nabla\varphi,\nabla\pi), \qquad\textrm{ a.e. in } Q_T,\label{ell-1}\\[1mm]
\frac{\partial\pi}{\partial\mathbf{n}}&=\mathcal{G}(\varphi), \qquad\textrm{ a.e. in } \Gamma\times (0,T),\label{ell-2}
\end{align}
where
\begin{align}
&\mathcal{F}(\varphi,\nabla\varphi,\nabla\pi):=\mathcal{F}_1(\varphi,\nabla\varphi,\nabla\pi)+\mathcal{F}_2(\varphi,\nabla\varphi)\,,\label{SOURCE0}\\[1mm]
&\mathcal{F}_1:=\zeta(\varphi)\nabla\varphi\cdot\nabla\pi\,,
\quad\mathcal{F}_2:=\varphi\,\mbox{div}(\nabla J\ast\varphi)+(1-\varphi\,\zeta(\varphi))
(\nabla\varphi\cdot(\nabla J\ast\varphi))\,,  \label{SOURCE} \\[1mm]
&\mathcal{G}:=(\nabla J\ast\varphi)\varphi\cdot\mathbf{n}\,, \label{BDRY}
\end{align}
with $\zeta(\varphi):=\eta^\prime(\varphi)/\eta(\varphi)$. We recall that, as above, the explicit time dependence is omitted.
We start with dimension two.

\medskip

$\blacktriangleright\, d=2,\, J\, admissible$.

\medskip

  From \eqref{str-reg-2}-\eqref{str-reg-3}, we have that $\nabla\pi,\nabla\varphi\in L^r(\Omega)^2$,
  for all $r\in(1,\infty)$. Thus $\mathcal{F}\in L^p(\Omega)$, for all $p\in(1,\infty)$.
As far as the boundary term $\mathcal{G}$ is concerned, we deduce that $\varphi\in W^{1,p}(\Omega)$ for all $1<p<\infty$, as a
consequence of $\varphi\in H^2(\Omega)$. Thus we have $\varphi\in W^{1-1/p,p}(\Gamma)$. Moreover, by relying only on the condition that $J$ is admissible and
by applying Lemma \ref{admiss}, we get $J\ast\varphi\in W^{2,p}(\Omega)$, for all $1<p<\infty$,
and this implies that $(\nabla J\ast\varphi)\cdot\mathbf{n}=\partial_\mathbf{n}(J\ast\varphi)\in W^{1-1/p,p}(\Gamma)$,
for all $1<p<\infty$. 
Hence, we have that
 $\varphi, \partial_\mathbf{n}(J\ast\varphi)\in W^{1-1/p,p}(\Gamma)\cap L^\infty(\Gamma)$, and this also entails
 that $\mathcal{G}\in W^{1-1/p,p}(\Gamma)\cap L^\infty(\Gamma)$, for $1< p<\infty$.
 Using now Proposition \ref{preg} with $r=0$, $t=1-1/p$, so that we have $s=2$, and $s-1/p=2-1/p$ is not an integer, we find $\pi\in W^{2,p}(\Omega)$ and the following estimate holds (see also \eqref{SOURCE0})
 \begin{align}
 &\Vert\pi\Vert_{W^{2,p}(\Omega)}\leq C(\Vert\mathcal{F}_1\Vert_{L^p(\Omega)}+\Vert\mathcal{F}_2\Vert_{L^p(\Omega)}
 +\Vert\mathcal{G}\Vert_{W^{1-1/p,p}(\Gamma)})\,,\quad 1< p<\infty\,.\label{e52}
 \end{align}
As above, in the sequel of this proof we will indicate by $C$ a generic positive constant which only depends
on the main constants of the problem (see (\textbf{H1})-(\textbf{H8})) and on $\Omega$ at most. This constant may also vary within the same line.
Any other dependency will be explicitly pointed out.
\medskip

We now proceed to estimate the three norms on the right hand side of \eqref{e52}. To this aim it is convenient to
distinguish the two cases
$2\leq p<\infty$, and $1<p<2$.

(i) {\it Case $2\leq p<\infty$.} 
We have (see \ref{SOURCE})
\begin{align}
&\Vert\mathcal{F}_1\Vert_{L^p(\Omega)}\leq \zeta_\infty\Vert\nabla\varphi\Vert_{L^{p+\e}(\Omega)^2}\Vert\nabla\pi\Vert_{L^{q}(\Omega)^2}\,,
\quad q:=p\Big(1+\frac{p}{\e}\Big)\,,
\label{e133}
\end{align}
where $\e>0$ will be conveniently chosen later. Moreover, we take advantage
of the $\alpha-$H\"{o}lder continuity property of $\pi$ and of Proposition
\ref{Bre-Mir} to estimate the $L^{q}(\Omega)^2$-norm of $\nabla\pi$, namely,
\begin{align}
&\Vert\nabla\pi\Vert_{L^{q}(\Omega)^2}\leq C\Vert\pi\Vert_{W^{\frac{k}{\rho},\rho}(\Omega)}^{1-\beta}\Vert\pi\Vert_{W^{2,p}(\Omega)}^\beta
\leq C\Vert\pi\Vert_{C^\alpha(\overline{\Omega})}^{1-\beta}\Vert\pi\Vert_{W^{2,p}(\Omega)}^\beta\,,
\label{e134}
\end{align}
for some $\beta\in(0,1)$, $k\geq 0$ and $\rho\geq1$, with $\rho>k/\alpha$, so that the injection $C^\alpha(\overline{\Omega})\hookrightarrow
W^{\frac{k}{\rho},\rho}(\Omega)$ holds true and allows us to control the $W^{\frac{k}{\rho},\rho}(\Omega)-$norm of $\pi$ by a constant (see \eqref{e53})
which only depends on structural parameters.
By combining \eqref{e133} with \eqref{e134}, and by employing the classical two-dimensional Gagliardo-Nirenberg inequality
to estimate the $L^{p+\e}(\Omega)^2$-norm of $\nabla\varphi$,
we obtain
\begin{align}
\Vert\mathcal{F}_1\Vert_{L^p(\Omega)}&\leq  C\Vert\nabla\varphi\Vert_{L^{p+\e}(\Omega)^2}\Vert\pi\Vert_{W^{2,p}(\Omega)}^\beta
\leq \delta\Vert\pi\Vert_{W^{2,p}(\Omega)}+C_\delta\Vert\nabla\varphi\Vert_{L^{p+\e}(\Omega)^2}^{\frac{1}{1-\beta}}\notag\\[2mm]
&\leq \delta\Vert\pi\Vert_{W^{2,p}(\Omega)}
+C_\delta\Vert\varphi\Vert_{H^2(\Omega)}^{\frac{1}{1-\beta}(1-\frac{2}{p+\e})}\,,
\label{e143}
\end{align}
where we have also used the $L^\infty(0,T;V)-$regularity of $\varphi$.
Using now Proposition \ref{Bre-Mir}, the interpolation inequality \eqref{e134} holds, provided that
$\beta\in(0,1)$ is given by
\begin{align}
&\frac{1}{q}=\Big(\frac{1-\beta}{\rho}+\frac{\beta}{p}\Big)-\frac{s-1}{2}\,,\quad s:=(1-\beta)\frac{k}{\rho}+2\beta\,,
\label{e135}
\end{align}
with $k\geq 0$ and $\rho\geq 1$, and with $\rho>k/\alpha$ satisfying the following condition
\begin{align}
&(1-\beta)\frac{k}{\rho}+2\beta\geq 1\,.\label{e136}
\end{align}
Noting that $q>p\geq 2$, that $q/(q-2)>p/2(p-1)=p^\prime/2$, and assuming in addition
that $k>2$ (we can easily see that we can restrict to
$k>2$ in all our analysis)\footnote{Indeed, if $0\leq k\leq 2$, on account of \eqref{e138},
then we have $\beta\geq p(q-2)/2q(p-1)>\beta_\ast$, for $\beta_\ast$ given by \eqref{e141} below.
Moreover, we have also that $p(q-2)/2q(p-1)>1/2>\beta_\ast$, for the $\beta_\ast$ given by \eqref{e178}, since $\e<p$
implies $q>2p$.},
we can then check that \eqref{e135} admits a solution $\beta\in (0,1)$ if and only if
\begin{align}
&\rho>\frac{q}{q-2}\,(k-2)\,,\label{e137}
\end{align}
with $\beta$ given by
\begin{align}
&\beta=\beta(k,\rho):=\frac{p}{q}\,\frac{(q-2)\,\rho-q\,(k-2)}{2\,(p-1)\,\rho-p\,(k-2)}\,.\label{e138}
\end{align}
Moreover, by taking in addition $\e$ such that $0<\e<p$ (which ensures that $q>2p$), and
since we are assuming \eqref{e138}, we can check that condition
\eqref{e136} is satisfied if and only if
\begin{align}
&\rho\geq \frac{q-p}{q-2p}\,(k-2)\,-\,\frac{pq-2(q-p)}{q-2p}
=\frac{p}{p-\e}\,(k-2)-p\,\frac{p+\e-2}{p-\e}\,.\label{e139}
\end{align}
By comparing the slopes of the affine in $k$ functions on the right hand sides of
 \eqref{e137} and of \eqref{e139}, we see that $(q-p)/(q-2p)= p/(p-\e)\geq q/(q-2)$,
since we are taking $0<\e<p$.
We now compare the slope $p/(p-\e)$ with $1/\alpha$ (always for $0<\e<p$), namely with the slope
of $k \mapsto k/\alpha$.
Let us choose $\e$ such that $0<\e\leq p(1-\alpha)$, which
ensures that $p/(p-\e)\leq 1/\alpha$.
Hence, the admissible region $\mathcal{R}$ for $k, \rho$ (namely, the set
of all $[k,\rho]$ such that $\eqref{e137}$ and \eqref{e139}, together with
conditions $\rho>k/\alpha$ and $k>2$, are satisfied) turns out to be
\begin{align}
&\mathcal{R}=\Big\{[k,\rho]\in [0,\infty)\times [1,\infty)\,:\,k>2\,,\,\,\rho>\frac{k}{\alpha}\,\Big\}\,.
\label{e140}
\end{align}
Computing the infimum of $\beta$ over $\mathcal{R}$, it is not difficult to find
that\footnote{This infimum can be computed by observing that
$$\inf_{\mathcal{R}}\beta=\lim_{R\to\infty}\,\min_{\overline{\mathcal{R}}\cap\{\rho\leq R\}} \beta\,.$$
The minimum of $\beta$ on the compact set $\overline{\mathcal{R}}\cap\{\rho\leq R\}$, with $R> 2/\alpha$,
is attained in only one point at the boundary, namely in $[\alpha R,R]$, which does not belong to $\mathcal{R}$.}
\begin{align}
&\beta_\ast=\beta_\ast(\e):=\inf_{[k,\rho]\in\mathcal{R}}\,\beta(k,\rho)=\frac{p}{q}\,
\frac{(1-\alpha)q-2}{(2-\alpha)p-2}=\frac{1}{p+\e}\,\frac{(1-\alpha)\,p\,(p+\e)-2\e}{(2-\alpha)\,p-2}\,,\label{e141}
\end{align}
for $0<\e\leq p(1-\alpha)$ (this infimum is not attained). Now, owing to \eqref{str-reg-3}, and writing $\beta=\beta_\ast^+$, we infer that
the time integrability exponent of the second term on the right hand side of \eqref{e143} is
given by
\begin{align}
&\sigma_p:=2(1-\beta)\,\frac{p+\e}{p+\e-2}=\Big(\frac{2p}{(2-\alpha)p-2}\Big)^-\,.\label{e145}
\end{align}
Note that $\sigma_p$ does not depend on $\e$, if $0<\e\leq p(1-\alpha)$.
Let us consider also the case $p(1-\alpha)<\e<p$. For this case, the admissible region $\mathcal{R}$
for $k,\rho$ becomes
\begin{align}
&\mathcal{R}=\Big\{[k,\rho]\in [0,\infty)\times [1,\infty)\,:\,k>2\,,\,\,\rho>\frac{k}{\alpha}\,,\,\,
\rho\geq\frac{p}{p-\e}\,(k-2)-p\,\frac{p+\e-2}{p-\e}\,\Big\}\,.\label{e144}
\end{align}
Let us now compute the infimum of $\beta$ over this new region $\mathcal{R}$. Denoting the affine function
on the right hand side of \eqref{e139}, for simplicity, by $g(k)$, we have that $g(k)=k/\alpha$ for
$$k=k^\ast:=\frac{\alpha\, p\,(p+\e)}{\e-(1-\alpha)p}\,,$$
and we can check that
$$\beta(k,g(k))=\frac{p}{q}\,\frac{k-q}{k-2p}\,,\qquad\forall k\geq k^\ast\,,$$
with $k\mapsto \beta(k,g(k))$ (strictly) increasing on $[k^\ast,\infty)$.
By taking the geometry of $\mathcal{R}$ into account we can thus see that
\begin{align}
&\beta_\ast:=\inf_\mathcal{R}\beta=\beta(k^\ast,g(k^\ast))=\beta(k^\ast,k^\ast/\alpha)=\frac{1-\alpha}{2-\alpha}
\label{e178}
\end{align}
and this infimum is not attained; notice that, in this case, $\beta_\ast$ does not depend on $\e\in(p(1-\alpha),p)$.
Therefore, still writing $\beta=\beta_\ast^+$, we infer that
the time integrability exponent of the second term on the right hand side of \eqref{e143} is now
given by
\begin{align}
&2(1-\beta)\,\frac{p+\e}{p+\e-2}=\Big(\frac{2}{2-\alpha}\Big)^-\,\frac{p+\e}{p-2+\e}\,,
\end{align}
and since the right hand side is decreasing with respect to $\e$, we choose $\e=(p(1-\alpha))^+$ to get the best
time integrability exponent. In doing so we obtain the same $\sigma_p$ as in \eqref{e145}.
We thus conclude that the analysis of the case $p(1-\alpha)<\e<p$ does not improve $\sigma_p$.

Let us now estimate the second term on the right hand side of \eqref{e52}, still assuming that
$2\leq p<\infty$. By relying on Lemma \ref{admiss},
we obtain (see \eqref{SOURCE})
\begin{align}
&\Vert\mathcal{F}_2\Vert_{L^p(\Omega)}\leq C_p+(1+\zeta_\infty)\,b\,\Vert\nabla\varphi\Vert_{L^p(\Omega)^2}
\leq C+C\Vert\nabla\varphi\Vert^{\frac{2}{p}}\Vert\varphi\Vert_{H^2(\Omega)}^{1-\frac{2}{p}}\,,\label{e146}
\end{align}
where $\zeta_\infty:=\Vert\zeta\Vert_{L^\infty(-1,1)}$.
By taking the $L^\infty(0,T;V)$ regularity of $\varphi$ into account (cf. \eqref{str-reg-3}), this leads to
\begin{align}
&\mathcal{F}_2\in L^{\hat{p}}(0,T;L^p(\Omega))\,,
\label{e147}
\end{align}
where $\hat{p}:=2p/(p-2)$, if $2<p<\infty$, and $\hat{p}:=\infty$, if $p=2$. Notice that $\hat{p}>\sigma_p$, with $\sigma_p$ given by \eqref{e145}.
As far as the third term on the right hand side of \eqref{e52} is concerned,
by means of Lemma \eqref{trace-product}
we have that (see \eqref{BDRY})
\begin{align}
\Vert\mathcal{G}\Vert_{W^{1-1/p,p}(\Gamma)}&\leq\Vert\varphi\Vert_{L^\infty(\Gamma)}
\Big\Vert\frac{\partial}{\partial\nvec}(J\ast\varphi)\Big\Vert_{W^{1-1/p,p}(\Gamma)}
+\Big\Vert\frac{\partial}{\partial\nvec}(J\ast\varphi)\Big\Vert_{L^\infty(\Gamma)}\Vert\varphi\Vert_{W^{1-1/p,p}(\Gamma)}
\notag\\[1mm]
&\leq C\Vert J\ast\varphi\Vert_{W^{2,p}(\Omega)}
+C\Vert\nabla J\ast\varphi\Vert_{W^{1,3}(\Omega)^2}\Vert\varphi\Vert_{W^{1,p}(\Omega)}\notag\\[1mm]
&\leq C_p(1+\Vert\nabla\varphi\Vert_{L^p(\Omega)^2})\,,\label{e148}
\end{align}
where also Lemma \ref{admiss} has been employed. Therefore, arguing in the same fashion as in \eqref{e146}, we get
\begin{align}
&\mathcal{G}\in L^{\hat{p}}(0,T;W^{1-1/p,p}(\Gamma))\,.
\label{e149}
\end{align}
By collecting \eqref{e143}, \eqref{e147}, \eqref{e148},
from \eqref{e52} it follows that
\begin{align}
&\pi\in L^{\sigma_p}(0,T;W^{2,p}(\Omega))\,,\qquad\sigma_p:=\Big(\frac{2p}{(2-\alpha)p-2}\Big)^-\,.
\label{e156}
\end{align}
Moreover, from \eqref{e134}, setting $\hat{\sigma}_p:=\sigma_p/\beta_\ast$, with $\beta_\ast$
given by \eqref{e141} and depending on $q$, with $p(2-\alpha)/(1-\alpha)\leq q<\infty$
(recall that $0<\e\leq p(1-\alpha)$ in \eqref{e141}), we also deduce
the following regularity
\begin{align}
&\pi\in L^{\hat{\sigma}_p}(0,T;W^{1,q}(\Omega))\,,\quad
\hat{\sigma}_p:=\Big(\frac{2q}{(1-\alpha)q-2}\Big)^-\,,\quad p\,\frac{2-\alpha}{1-\alpha}\leq q<\infty\,.
\label{e150}
\end{align}

(ii) {\it Case $1<p<2$.} We handle this case by exploiting the regularity
\eqref{e150} obtained above. Namely, we employ \eqref{e150} with $p=2$ and with $q=2(2-\alpha)/(1-\alpha)$, together with
the following interpolation inequality
$$\Vert\nabla\pi\Vert_{L^q(\Omega)^2}\leq \Vert\nabla\pi\Vert^{\frac{2(2-\alpha)}{q}-(1-\alpha)}
\Vert\nabla\pi\Vert_{L^{\frac{2(2-\alpha)}{1-\alpha}}(\Omega)^2}^{(2-\alpha)\frac{q-2}{q}}\,,
\quad 2\leq q\leq 2\,\frac{2-\alpha}{1-\alpha}\,,$$
and with the $L^\infty(V_0)-$regularity for $\pi$ (cf. \eqref{str-reg-2}), to get
\begin{align}
&\pi\in L^{\hat{\sigma}_p}(0,T;W^{1,q}(\Omega))\,,\quad \hat{\sigma}_p:=
\frac{1}{(1-\alpha)^2}\,\Big(\frac{2q}{q-2}\Big)^-\,,
\quad 2\leq q\leq 2\,\frac{2-\alpha}{1-\alpha}\,.\label{e151}
\end{align}
This completes the regularity \eqref{e150}. Next, we go back to \eqref{e133}, which we can write equivalently as (see \eqref{SOURCE})
\begin{align}
&\Vert\mathcal{F}_1\Vert_{L^p(\Omega)}\leq \zeta_\infty\Vert\nabla\varphi\Vert_{L^{\frac{pq}{q-p}}(\Omega)^2}\Vert\nabla\pi\Vert_{L^{q}(\Omega)^2}\,,
\label{e152}
\quad q>p\,.
\end{align}
Observe that $pq/(q-p)\leq 2$ if and only if $q\geq 2p/(2-p)$. Since $q$ is now taken
in the interval $[2,2(2-\alpha)/(1-\alpha)]$ (see \eqref{e151}), we can then distinguish two cases.
Assume first that $2p/(2-p)\leq 2(2-\alpha)/(1-\alpha)$, namely that $p\leq p_\alpha$, where
$p_\alpha:=2(2-\alpha)/(3-2\alpha)$. Then, supposing $2p/(2-p)\leq q\leq 2(2-\alpha)/(1-\alpha)$,
from \eqref{e151}, \eqref{e152}, accounting for the $L^\infty(0,T;V_0)-$regularity for $\pi$, we get
\begin{align}
&\mathcal{F}_1\in L^{\hat{\sigma}_p}(0,T;L^p(\Omega))\,,\label{e153}
\end{align}
with $\hat{\sigma}_p$ the same as in \eqref{e151}. Since $\hat{\sigma}_p$ is decreasing with respect to $q$, we take
$q=2p/(2-p)$ in \eqref{e153} (i.e. the left endpoint of the admissible interval for $q$) to get the best time integrability exponent.
This yields
\begin{align}
&\mathcal{F}_1\in L^{\sigma_p}(0,T;L^p(\Omega))\,,\quad \sigma_p:=\Big(\frac{p^\prime}{(1-\alpha)^2}\Big)^-\,,
\quad 1<p\leq p_\alpha:=\frac{2(2-\alpha)}{3-2\alpha}\,.\label{e154}
\end{align}
On the other hand, if $p_\alpha<p<2$ (which means that $2p/(2-p)\geq 2(2-\alpha)/(1-\alpha)$),
then we have $q\leq 2p/(2-p)$ for all $2\leq q\leq 2(2-\alpha)/(1-\alpha)$, and hence $pq/(q-p)\geq 2$.
The norm in $\nabla\varphi$ on the right hand side of \eqref{e152} will then
be estimated through the Gagliardo-Nirenberg inequality and this gives
\begin{align}
&\Vert\mathcal{F}_1\Vert_{L^p(\Omega)}\leq C\,\Vert\varphi\Vert_{H^2(\Omega)}^{\frac{pq-2(q-p)}{pq}}
\,\Vert\nabla\pi\Vert_{L^{q}(\Omega)^2}\,.\label{e157}
\end{align}
Using \eqref{e151} once more, we can easily see that the best time integrability exponent
for $\mathcal{F}_1$ with values in $L^p$ is reached by taking $q=2(2-\alpha)/(1-\alpha)$
in the admissible interval for $q$. Therefore we find
\begin{align}
&\mathcal{F}_1\in L^{\sigma_p}(0,T;L^p(\Omega))\,,\quad\sigma_p:=\Big(\frac{2p}{(2-\alpha)p-2}\Big)^-\,,
\quad p_\alpha\leq p<2\,.\label{e155}
\end{align}

As far as the second and the third norm on the right hand side of \eqref{e52} are concerned, from \eqref{e146} and \eqref{e148},
on account of \eqref{str-reg-3}, we get
\begin{align*}
&\mathcal{F}_2\in L^\infty(0,T;L^p(\Omega))\,,\quad\mathcal{G}\in L^\infty(0,T;W^{1-1/p,p}(\Gamma))\,.
\end{align*}
From \eqref{e52} we then deduce that
\begin{align}
&\pi\in L^{\sigma_p}(0,T;W^{2,p}(\Omega))\,,
\end{align}
with $\sigma_p$ given by \eqref{e154}, or by \eqref{e155}, according with the value of $p$
in the interval $(1,2)$.

Let us now analyze the regularity of $\uvec$ for both cases (i) and (ii).

\noindent
By taking the spatial derivatives
$\partial_j$ of \eqref{pb3-eq} , we get
\begin{align}
\partial_j u_k&=-\frac{1}{\eta(\varphi)}\,\partial^2_{jk}\pi+\frac{\zeta(\varphi)}{\eta(\varphi)}\,\partial_j\varphi\,\partial_k\pi
+\frac{1-\varphi\,\zeta(\varphi)}{\eta(\varphi)}\,(\partial_k J\ast\varphi)\,\partial_j\varphi+\frac{1}{\eta(\varphi)}\partial_j(\partial_k J\ast\varphi)\varphi\,,
\label{e64}
\end{align}
and the term to be estimated in $L^p(\Omega)$ in a less straightforward way is the second one on the right hand side of \eqref{e64}.

Let us consider the case (i), namely, $2\leq p<\infty$. It is immediate to see that the term $(\zeta(\varphi)/\eta(\varphi))\,\partial_j\varphi\,\partial_k\pi$ can be estimated
as in \eqref{e157}, for all $q$ such that $p(2-\alpha)/(1-\alpha)\leq q<\infty$.
Therefore, by means of \eqref{str-reg-3} and \eqref{e150}, we can easily check that
the time integrability exponent of the right hand side of \eqref{e157}
(and hence of the second term in \eqref{e64}) is  $\sigma_p$ given by \eqref{e156}.
Consider now the last two terms on the right hand side of \eqref{e64}. It is easy to realize that
\begin{align}
&\frac{1}{\eta(\varphi)}\partial_j(\partial_k J\ast\varphi)\varphi\in L^\infty(0,T;L^p(\Omega))\,,\quad
\frac{1-\varphi\,\zeta(\varphi)}{\eta(\varphi)}\,(\partial_k J\ast\varphi)\,\partial_j\varphi\in L^{\hat{p}}(0,T;L^p(\Omega))\,,
\label{e158}
\end{align}
where $\hat{p}=2p/(p-2)$, if $2<p<\infty$, and $\hat{p}=\infty$, if $p=2$.
Therefore, from \eqref{e64} we deduce that (note that $\hat{p}>\sigma_p$)
\begin{align}
&\uvec\in L^{\sigma_p}(0,T;W^{1,p}(\Omega)^2)\cap L^{\hat{\sigma}_p}(0,T;L^{q}(\Omega)^2)\,.\label{e159}
\end{align}
The case $1<p<2$ can be handled similarly. Therefore, for both cases (i) and (ii) we find that
\eqref{e159} holds with $\sigma_p$, $\hat{\sigma}_p$ and $q$ given by \eqref{e156} and by \eqref{e150}, respectively,
if $2\leq p<\infty$, or by \eqref{e154}-\eqref{e155} and \eqref{e151}, respectively, if $1<p<2$.

\medskip

$\blacktriangleright\, d=2\,,\; J\in W^{3,1}_{loc}(\mathbb{R}^2)$.

\medskip

This stronger assumption on the kernel $J$ allows to deduce an $H^1(\Omega)-$regularity for
$\mathcal{F}(\varphi,\nabla\varphi,\nabla\pi)$ (see \eqref{ell-1}).
Indeed we have
\begin{align}
&\partial_i\mathcal{F}_1(\varphi,\nabla\varphi,\nabla\pi)=\partial_i(\zeta(\varphi)\,\partial_k\varphi\,\partial_k\pi)
\nonumber\\
&=
\zeta(\varphi)\,\partial^2_{ik}\,\varphi\,\partial_k\pi+\zeta(\varphi)\,\partial_k\varphi\,\partial_{ik}^2\pi+
\zeta^\prime(\varphi)\,\partial_i\varphi\,\partial_k\varphi\,\partial_k\pi\,.\label{e69}
\end{align}
Recall first the Gagliardo-Nirenberg inequality (see Proposition \ref{Bre-Mir})
\begin{align}
&\Vert\pi\Vert_{W^{1,\infty}(\Omega)}\leq
C_q\Vert\pi\Vert_{W^{1,q}(\Omega)}^{\frac{q}{q+2}}\Vert\pi\Vert_{H^3(\Omega)}^{\frac{2}{q+2}}\,,\quad 2<q<\infty\,.
\label{e81}
\end{align}
Then we get
\begin{align}
\Vert\zeta(\varphi)\,\partial^2_{ik}\,\varphi\,\partial_k\pi\Vert&\leq\zeta_\infty\Vert\nabla^2\varphi\Vert
\Vert\nabla\pi\Vert_{L^\infty(\Omega)^2}
\leq C_q\Vert\varphi\Vert_{H^2(\Omega)}\Vert\pi\Vert_{W^{1,q}(\Omega)}^{\frac{q}{q+2}}
\Vert\pi\Vert_{H^3(\Omega)}^{\frac{2}{q+2}}\notag\\[1mm]
&\leq\delta\Vert\pi\Vert_{H^3(\Omega)}+C_{q,\delta}\,\Vert\varphi\Vert_{H^2(\Omega)}^{\frac{q+2}{q}}
\Vert\pi\Vert_{W^{1,q}(\Omega)}\,,\label{e70}
\end{align}
with $2<q<\infty$ arbitrarily large, and $\delta>0$ to be fixed later.
Using \eqref{str-reg-3}, \eqref{e175} with $p=2$ and with $2(2-\alpha)/(1-\alpha)\leq q<\infty$ (cf. \eqref{e176})
so that $q$ can be chosen arbitrarily large, we can easily see that the time integrability
exponent of the second term in the right hand side of the last inequality of \eqref{e70} is
given by $\sigma_\infty:=(2/(2-\alpha))^-$.
Moreover, the $L^2$-norm of the second term on the right hand side of \eqref{e69} can be estimated as follows
\begin{align*}
&\Vert\zeta(\varphi)\,\partial_k\varphi\,\partial_{ik}^2\pi\Vert\leq
\zeta_\infty\Vert\nabla\varphi\Vert_{L^{\frac{2p}{p-2}}(\Omega)^2}
\Vert\nabla^2\pi\Vert_{L^{p}(\Omega)^{2\times 2}}\leq C \Vert\varphi\Vert_{H^2(\Omega)}^{\frac{2}{p}}
\Vert\pi\Vert_{W^{2,p}(\Omega)}\,,
\end{align*}
with $2<p<\infty$. Thus, using \eqref{e175} and \eqref{e176}, we can easily get
\begin{align}
&\zeta(\varphi)\,\partial_k\varphi\,\partial_{ik}^2\pi\in L^{\sigma_\infty}(0,T;H)\,.
\label{e73}
\end{align}
As far as the third term on the right hand side of \eqref{e69} is concerned,
we have that
\begin{align}
&\Vert\zeta^\prime(\varphi)\,\partial_i\varphi\,\partial_k\varphi\,\partial_k\pi\Vert\leq
\zeta^\prime_\infty\,\Vert\nabla\varphi\Vert_{L^4(\Omega)^2}^2\Vert\nabla\pi\Vert_{L^\infty(\Omega)^2}
\leq C\Vert\varphi\Vert_{H^2(\Omega)}\Vert\nabla\pi\Vert_{L^\infty(\Omega)^2}\,.\label{e78}
\end{align}
Hence this term can be handled as in \eqref{e70}. 
We also need an estimate for the
$L^2$-norm of $\mathcal{F}_1$. To this aim, we observe that
\begin{align}
&\Vert\zeta(\varphi)\,\partial_k\varphi\,\partial_k\pi\Vert\leq\zeta_\infty\Vert\nabla\varphi\Vert\Vert\nabla\pi\Vert_{L^\infty(\Omega)^2}
\leq C\Vert\nabla\varphi\Vert\Vert\pi\Vert_{W^{2,p}(\Omega)}\,,\label{e79}
\end{align}
with $p=2^+$. Then, on account of \eqref{e175}, \eqref{e176}, this yields  $\zeta(\varphi)\,\partial_k\varphi\,\partial_k\pi\in L^{\sigma_2}(0,T;H)$ (note that
$\sigma_2=(2/(1-\alpha))^->\sigma_\infty$). 

Let us now consider the term $\mathcal{F}_2(\varphi,\nabla\varphi)$. We have that (see \eqref{SOURCE})
\begin{align}
\partial_i\mathcal{F}_2(\varphi,\nabla\varphi)&=\partial_i\varphi\,\partial_k(\partial_k J\ast\varphi)
+\varphi\,\partial_i(\partial_k(\partial_k J\ast\varphi))
+(1-\varphi\,\zeta(\varphi))\,\partial_{ik}^2\varphi\,(\partial_k J\ast\varphi)\notag\\[1mm]
&+(1-\varphi\,\zeta(\varphi))\,\partial_k\varphi\,\partial_i(\partial_k J\ast\varphi)
-(\zeta(\varphi)+\varphi\,\zeta^\prime(\varphi))\,\partial_i\varphi\,(\partial_k J\ast\varphi)\,\partial_k\varphi\,.
\end{align}
By estimating the $L^2$-norms of the terms on the right hand side one by one, we then get
\begin{align}
\Vert\nabla\mathcal{F}_2(\varphi,\nabla\varphi)\Vert &\leq\Vert\mbox{div}(\nabla J\ast\varphi)\Vert_{L^\infty(\Omega)}
\Vert\nabla\varphi\Vert+\Vert\nabla\mbox{div}(\nabla J\ast\varphi)\Vert\notag\\[1mm]
&+b\,(1+\zeta_\infty)\,\Vert\nabla^2\varphi\Vert+(1+\zeta_\infty)\,\Vert\nabla(\nabla J\ast\varphi)\Vert_{L^\infty(\Omega)^{2\times 2}}
\Vert\nabla\varphi\Vert\notag\\[1mm]
&\leq C(1+\Vert\varphi\Vert_{H^2(\Omega)})\,.\label{e74}
\end{align}
As far as the first, second and fourth terms on the right hand side of the first
inequality in \eqref{e74} are concerned, these
have been estimated by relying on the assumption that
$J\in W^{3,1}_{loc}(\mathbb{R}^2)$.
Moreover, $b$ is the constant appearing in \textbf{(H2)} and \eqref{GN-spe} with $p=4$ has been used.
We thus immediately deduce
\begin{align}
&\mathcal{F}_2(\varphi,\nabla\varphi)\in L^2(0,T;V)\,.\label{e76}
\end{align}

There now remains to address the boundary term $\mathcal{G}(\varphi)$ (see \eqref{ell-2}). Notice first that $J\ast\varphi\in H^3(\Omega)$. Thus we have that
$\partial_{\nvec}(J\ast\varphi)=(\nabla J\ast\varphi)\cdot\nvec\in H^{3/2}(\Gamma)$ so that
$\mathcal{G}(\varphi)\in H^{3/2}(\Gamma)$. Invoking now Lemma \ref{trace-product}
(which can be easily generalized to the case $s\geq 1$), the $H^{3/2}(\Gamma)$-norm of $\mathcal{G}$
can be estimated in the following way (see also \eqref{e4})
\begin{align*}
\Vert\mathcal{G}(\varphi)\Vert_{H^{3/2}(\Gamma)}&=\Vert(\nabla J\ast\varphi)\varphi\cdot\nvec\Vert_{H^{3/2}(\Gamma)}\notag\\[1mm]
&\leq\Vert\varphi\Vert_{L^\infty(\Gamma)}\Big\Vert\frac{\partial}{\partial\nvec}(J\ast\varphi)\Big\Vert_{H^{3/2}(\Gamma)}
+\Vert\varphi\Vert_{H^{3/2}(\Gamma)}\Big\Vert\frac{\partial}{\partial\nvec}(J\ast\varphi)\Big\Vert_{L^\infty(\Gamma)}\notag\\[1mm]
&\leq C\Vert J\ast\varphi\Vert_{H^3(\Omega)}
+C\Vert\varphi\Vert_{H^2(\Omega)}\Vert\nabla J\ast\varphi\Vert_{W^{1,4}(\Omega)^2}\notag\\[1mm]
&\leq C\,\Vert\varphi\Vert+C\Vert\varphi\Vert_{H^2(\Omega)}\Vert\varphi\Vert_{L^4(\Omega)}
\leq C\big(1+\Vert\varphi\Vert_{H^2(\Omega)}\big)\,.
\end{align*}
Hence we infer that
\begin{align}
&\mathcal{G}(\varphi)\in L^2(0,T;H^{3/2}(\Gamma))\,.\label{e77}
\end{align}
We now recall the well-known elliptic estimate (see \eqref{ell-1}-\eqref{ell-2})
\begin{align}
&\Vert\pi\Vert_{H^3(\Omega)}\leq C\big(\Vert\mathcal{F}_1\Vert_V+\Vert\mathcal{F}_2\Vert_V
+\Vert\mathcal{G}\Vert_{H^{3/2}(\Gamma)}\big)\,,
\end{align}
and, by collecting \eqref{e70}-\eqref{e79}, \eqref{e76}, \eqref{e77}
we can conclude that
\begin{align}\label{e82}
&\pi\in L^{\sigma_\infty}(0,T;H^3(\Omega))\,,\quad\sigma_\infty:=\Big(\frac{2}{2-\alpha}\Big)^-\,.
\end{align}

Let us complete this case by analyzing the corresponding regularity of $\uvec$.
We first take the spatial derivative $\partial_i$ of \eqref{e64} and get
\begin{align}
\partial_{ij}^2 u_k&=-\frac{1}{\eta(\varphi)}\,\partial_{ijk}^3\pi+
\frac{\zeta(\varphi)}{\eta(\varphi)}\,\big(\partial_i\varphi\,\partial_{jk}^2\pi
+\partial_j\varphi\,\partial_{ik}^2\pi\big)
+\frac{\zeta(\varphi)}{\eta(\varphi)}\,\partial_{ij}^2\varphi\,\partial_k\pi\notag\\[1mm]
&+\Big(\frac{\zeta}{\eta}\Big)^\prime(\varphi)\,\partial_i\varphi\,\partial_j\varphi\,\partial_k\pi
+\frac{1-\varphi\,\zeta(\varphi)}{\eta(\varphi)}\,(\partial_k J\ast\varphi)\,\partial_{ij}^2\varphi
+\frac{1-\varphi\,\zeta(\varphi)}{\eta(\varphi)}\,\partial_i(\partial_k J\ast\varphi)\,\partial_j\varphi\notag\\[1mm]
&-2\,\frac{\zeta(\varphi)}{\eta(\varphi)}\,\partial_i\varphi\,(\partial_k J\ast\varphi)\,\partial_j\varphi
+\frac{1-\varphi\,\zeta(\varphi)}{\eta(\varphi)}\,\partial_j(\partial_k J\ast\varphi)\,\partial_i\varphi
+\frac{1}{\eta(\varphi)}\,\partial_i\partial_j(\partial_k J\ast\varphi)\,\varphi\notag\\[1mm]
&-\Big(\frac{\zeta}{\eta}\Big)^\prime(\varphi)\,\partial_i\varphi\,\partial_j\varphi\,(\partial_k J\ast\varphi)\,\varphi\,.
\label{e80}
\end{align}
We now proceed to estimate the $H$-norms of the ten terms on the right hand side of \eqref{e80}.
For the sake of simplicity, we denote these norms by $\mathcal{I}_1,\dots\mathcal{I}_{10}$
(preserving the same order as in \eqref{e80}). We have that
\begin{align*}
&\mathcal{I}_2\leq 2\,\frac{\zeta_\infty}{\eta_1}\,\Vert\nabla\varphi\Vert_{L^{\frac{2p}{p-2}}(\Omega)^2}
\Vert\nabla^2\pi\Vert_{L^{p}(\Omega)^{2\times 2}}\leq C_q\,\Vert\varphi\Vert_{H^2(\Omega)}^{\frac{2}{p}}
\Vert\pi\Vert_{W^{2,p}(\Omega)}\,,
\end{align*}
with $2<p<\infty$. 
By means of \eqref{e175} and \eqref{e176} (see also \eqref{str-reg-3}), we
infer that the time integrability exponent for $\mathcal{I}_2$ is still given by $\sigma_\infty$,
namely we get $\mathcal{I}_2\in L^{\sigma_\infty}(0,T)$.
Next, on account of \eqref{e81}, we have that
\begin{align*}
&\mathcal{I}_3\leq \frac{\zeta_\infty}{\eta_1}\,\Vert\nabla^2\varphi\Vert\Vert\nabla\pi\Vert_{L^\infty(\Omega)^2}
\leq C_q\,\Vert\varphi\Vert_{H^2(\Omega)}\Vert\pi\Vert_{W^{1,q}(\Omega)}^{\frac{q}{q+2}}\Vert\pi\Vert_{H^3(\Omega)}^{\frac{2}{q+2}}\,,
\end{align*}
with $2<q<\infty$. From \eqref{e175} and \eqref{e176} we then get $\sigma_\infty$ as
 time integrability exponent for $\mathcal{I}_3$. 
Next, observe that
\begin{align*}
&\mathcal{I}_4\leq\Big(\frac{\zeta}{\eta}\Big)^\prime_\infty\,\Vert\nabla\varphi\Vert_{L^4(\Omega)^2}^2
\Vert\nabla\pi\Vert_{L^\infty(\Omega)^2}\leq C\,\Vert\varphi\Vert_{H^2(\Omega)}\Vert\nabla\pi\Vert_{L^\infty(\Omega)^2}\,,
\end{align*}
where $(\zeta/\eta)^\prime_\infty:=\Vert(\zeta/\eta)^\prime\Vert_{L^\infty(-1,1)}$. Then, arguing as for $\mathcal{I}_3$, we again get
$\mathcal{I}_4\in L^{\sigma_\infty}(0,T)$.
The estimates of the terms from $\mathcal{I}_5$ to $\mathcal{I}_{10}$ are straightforward recalling that
$J\in W^{3,1}_{loc}(\mathbb{R}^2)$. The details are left to the reader. In particular, we can easily find that
$\sum_{l=5}^{10}\mathcal{I}_l\in L^2(0,T)$. Summing up, employing \eqref{e82} to estimate $\mathcal{I}_1$,
we conclude that
\begin{align}
&\uvec\in L^{\sigma_\infty}(0,T;H^2(\Omega)^2)\,.
\end{align}

\medskip

$\blacktriangleright\, d=3,\,\; J\, admissible$.

\medskip

From \eqref{str-reg-2}-\eqref{str-reg-3}, we have that $\nabla\pi,\nabla\varphi\in L^{r}(\Omega)^3$,
  for all $2< r\leq 6$
which entails that $\mathcal{F}\in L^p(\Omega)$, for all $1<p\leq 3$.
As far as the boundary term $\mathcal{G}$ is concerned, by arguing as at the beginning of the discussion
of the case $d=2$, we can deduce that
$\mathcal{G}\in W^{1-1/p,p}(\Gamma)\cap L^\infty(\Gamma)$, for $1< p\leq 6$.
Thanks to elliptic regularity, a two step bootstrap argument allows us to deduce that
$\nabla\pi\in L^\infty(\Omega)^3$ (while $\nabla\varphi$ has a spatial integrability exponent
at most equal to $6$). As a consequence, for $1< p\leq 6$, we have that $\pi\in W^{2,p}(\Omega)$, and that estimate \eqref{e52} holds true.

Before addressing the terms in estimate \eqref{e52}, let us first point out how to control
the $L^p(\Omega)^3$-norm of $\nabla\varphi$, for $2\leq p\leq 6$, by the $H^2(\Omega)$-norm of $\varphi$
in a convenient way, i.e., keeping the exponent in the $H^2(\Omega)$-norm as low as possible.
If $4\leq p\leq 6$, we can use \eqref{GN-spe} by relying on the boundedness of $\varphi$, and find
\begin{align}
&\Vert\nabla\varphi\Vert_{L^p(\Omega)^3}\leq C\Vert\varphi\Vert_{H^2(\Omega)}^{2(1-\frac{3}{p})}\,,\quad 4\leq p\leq 6\,.
\label{e101}
\end{align}
If $2\leq p<4$, the interpolation inequality \eqref{GN-spe} cannot be directly applied. However, we can first proceed
by means of an elementary interpolation inequality and then apply \eqref{GN-spe}, namely,
\begin{align}
&\Vert\nabla\varphi\Vert_{L^p(\Omega)^3}\leq\Vert\nabla\varphi\Vert^{\frac{4}{p}-1}
\Vert\nabla\varphi\Vert_{L^4(\Omega)^3}^{2-\frac{4}{p}}\leq C\Vert\nabla\varphi\Vert^{\frac{4}{p}-1}
\Vert\varphi\Vert_{L^\infty(\Omega)}^{1-\frac{2}{p}}\Vert\varphi\Vert_{H^2(\Omega)}^{1-\frac{2}{p}}\,,
\label{e170}
\end{align}
which, on account of the boundedness of $\varphi$ and of its $L^\infty(0,T;V)-$regularity, gives
\begin{align}
&\Vert\nabla\varphi\Vert_{L^p(\Omega)^3}\leq C\Vert\varphi\Vert_{H^2(\Omega)}^{1-\frac{2}{p}}\,,\quad 2\leq p\leq 4\,.
\label{e102}
\end{align}

We can now proceed to estimate the three norms on the right hand side of \eqref{e52}.
As far as the first norm is concerned, we have (see \eqref{SOURCE})
\begin{align}
&\Vert\mathcal{F}_1\Vert_{L^p(\Omega)}\leq \zeta_\infty\,\Vert\nabla\varphi\Vert_{L^{p+\e}(\Omega)^3}\Vert\nabla\pi\Vert_{L^{q}(\Omega)^3}\,,\quad
q:=p\Big(1+\frac{p}{\e}\Big)\,,\label{e100}
\end{align}
where $\e>0$ is such that $p+\e\leq 6$ and will be conveniently chosen later.
We then take advantage of the $\alpha-$H\"{o}lder continuity property of $\pi$ and of Proposition
\ref{Bre-Mir} to estimate the $L^{q}$-norm of $\nabla\pi$ as follows
\begin{align}
&\Vert\nabla\pi\Vert_{L^{q}(\Omega)^3}\leq C\Vert\pi\Vert_{W^{\frac{k}{\rho},\rho}(\Omega)}^{1-\beta}\Vert\pi\Vert_{W^{2,p}(\Omega)}^\beta\,,
\label{e95}
\end{align}
for some $\beta\in(0,1)$, $k>0$ and $\rho\geq 1$, with $\rho>k/\alpha$, so that the injection $C^\alpha(\overline{\Omega})\hookrightarrow
W^{\frac{k}{\rho},\rho}(\Omega)$ holds true and allows to control the $W^{\frac{k}{\rho},\rho}(\Omega)-$norm of $\pi$ by a constant
which only depends on structural parameters (cf. \eqref{e53}).

By combining \eqref{e100} with \eqref{e95}, and by taking \eqref{e101}, \eqref{e102} into account, we have that
\begin{align}
&\Vert\mathcal{F}_1\Vert_{L^p(\Omega)}\leq  C\Vert\nabla\varphi\Vert_{L^{p+\e}(\Omega)^3}\Vert\pi\Vert_{W^{2,p}(\Omega)}^\beta
\leq \delta\Vert\pi\Vert_{W^{2,p}(\Omega)}+C_\delta\Vert\nabla\varphi\Vert_{L^{p+\e}(\Omega)^3}^{\frac{1}{1-\beta}}\notag\\[1mm]
&\leq \left\{\begin{array}{ll}
\delta\Vert\pi\Vert_{W^{2,p}(\Omega)}
+C_\delta\Vert\varphi\Vert_{H^2(\Omega)}^{\frac{1}{1-\beta}(1-\frac{2}{p+\e})},\qquad\mbox{if }\,\,2\leq p+\e\leq 4,\\[2mm]
\delta\Vert\pi\Vert_{W^{2,p}(\Omega)}
+C_\delta\Vert\varphi\Vert_{H^2(\Omega)}^{\frac{2}{1-\beta}(1-\frac{3}{p+\e})},\qquad\mbox{if }\,\,4\leq p+\e\leq 6.
\end{array}\right.\label{e103}
\end{align}

On account of \eqref{Bre-Mir-cond}, the interpolation inequality \eqref{e95} holds, provided that
$\beta\in(0,1)$ is given by
\begin{align}
&\frac{1}{q}=\Big(\frac{1-\beta}{\rho}+\frac{\beta}{p}\Big)-\frac{s-1}{3}\,,\quad s:=(1-\beta)\frac{k}{\rho}+2\beta\,,
\label{e96}
\end{align}
where $k>0$ and $\rho\geq 1$, with $\rho>k/\alpha$ satisfying the following condition
\begin{align}
&(1-\beta)\frac{k}{\rho}+2\beta\geq 1\,.\label{e97}
\end{align}

It is now convenient to distinguish the following cases in the analysis, according to the values of $p\in[2,6]$.

\medskip

(i) {\itshape Case $2\leq p<3$.}

\medskip

Let us take $3-p\leq \e< p^2/(3-p)$. This ensures that $3<q\leq 3p/(3-p)$ so that
$q/(q-3)\geq p/(2p-3)$. Assuming then $k>3$ (this is not restrictive)\footnote{Indeed, if $0\leq k\leq 3$, we have that
$\beta\geq p(q-3)/q(2p-3)$ (cf. \eqref{e113}), and we can see that $p(q-3)/q(2p-3)>\beta_\ast$,
for $\beta_\ast$ given by \eqref{e112} below (and for the admissible $\e$ and $\alpha$ considered in
\eqref{e112}, namely $\e\in [3-p,p(1-\alpha)]$, and $\alpha\in (0,(2p-3)/p]$).
Moreover, we have also that $p(q-3)/q(2p-3)>1/2>\beta_\ast$,
for $\beta_\ast$ given by \eqref{e117} below, since the condition $\e<p$ implies $q>2p$.
We argue similarly also for the cases $3\leq p<4$, and $4\leq p\leq 6$.},
we can then see that \eqref{e96} admits a solution $\beta\in (0,1)$ if and only if
\begin{align}
&\rho>\frac{q}{q-3}\,(k-3)\,,\label{e98}
\end{align}
with $\beta$ given by
\begin{align}
&\beta=\beta(k,\rho):=\frac{p}{q}\,\frac{(q-3)\rho-q(k-3)}{(2p-3)\rho-p(k-3)}\,.\label{e113}
\end{align}
Moreover, by taking in addition
$\e$ such that $3-p\leq \e<p$ (note that $\e<p$ ensures that $q>2p$),
since we are assuming that \eqref{e98} is satisfied, we can check that condition
\eqref{e97} is satisfied if and only if
\begin{align}
&\rho\geq \frac{q-p}{q-2p}\,(k-3)\,-\,\frac{pq-3(q-p)}{q-2p}
=\frac{p}{p-\e}\,(k-3)-p\,\frac{p+\e-3}{p-\e}\,.\label{e99}
\end{align}
By comparing the slopes of the affine functions on the right hand sides of
 \eqref{e98} and \eqref{e99} we see that $q/(q-3)\leq p/(p-\e)$ since $3-p\leq \e<p$.
The slope $p/(p-\e)$ has now to be compared with $1/\alpha$, namely with the slope
of $k \mapsto \rho>k/\alpha$.
Let us assume that $0<\alpha\leq (2p-3)/p$, that is, $3-p\leq p(1-\alpha)$.
If $\alpha$ satisfies this condition, then we can choose $\e$ such that
$3-p\leq \e\leq p(1-\alpha)$, and this ensures that $p/(p-\e)\leq 1/\alpha$.
Hence, the admissible region turns out to be
\begin{align}
&\mathcal{R}=\Big\{[k,\rho]\in [0,\infty)\times [1,\infty)\,: \,k>3\,,\,\,\rho>\frac{k}{\alpha}\,\Big\}\,.
\label{e114}
\end{align}
Computing the infimum of $\beta$ over $\mathcal{R}$, it is not difficult to find that
\begin{align}
&\beta_\ast=\beta_\ast(\e) :=\inf_{[k,\rho]\in\mathcal{R}}\,\beta(k,\rho)=\frac{p}{q}\,
\frac{(1-\alpha)q-3}{(2-\alpha)p-3}=\frac{1}{p+\e}\,\frac{(1-\alpha)\,p\,(p+\e)-3\e}{(2-\alpha)\,p-3}\,.\label{e112}
\end{align}
Also this infimum is not attained. If $p+\e\leq 4$, namely if
$3-p\leq \e \leq \min(4-p,p(1-\alpha))$, owing to \eqref{str-reg-3}, and writing $\beta=\beta_\ast^+$, we infer that
the time integrability exponent of the second term on the right hand side of \eqref{e103} is
given by
\begin{align}
&\mu_p:=2(1-\beta)\,\frac{p+\e}{p+\e-2}=\Big(\frac{2p}{(2-\alpha)p-3}\Big)^-\,\frac{\e+p-3}{\e+p-2}\,.\label{e104}
\end{align}
Observe that the right hand side in \eqref{e104} is (strictly) increasing in $\e$. Then it is convenient
to choose the greatest admissible value for $\e$ to get the best time integrability exponent $\mu_p$.
Hence, if $\alpha \geq (2p-4)/p$ (i.e., $4-p\geq p(1-\alpha)$), we take $\e=p(1-\alpha)$
getting
\begin{align}
&\mu_p=\Big(\frac{2p}{(2-\alpha)p-2}\Big)^-\,,\quad \frac{2p-4}{p}\leq \alpha\leq\frac{2p-3}{p}\,.
\label{e107}
\end{align}
Moreover, since $3-p\leq\e\leq p(1-\alpha)$, then $q$ satisfies $p(2-\alpha)/(1-\alpha)\leq q\leq 3p/(3-p)$,
and, setting $\hat{\mu}_p:=\mu_p/\beta_\ast$, from \eqref{e112} we have that
\begin{align}
&\hat{\mu}_p=\Big(\frac{(2-\alpha)p-3}{(2-\alpha)p-2}\Big)^-\,\frac{2q}{(1-\alpha)q-3}\,,\quad
p\,\frac{2-\alpha}{1-\alpha}\leq q\leq\frac{3p}{3-p}\,,\quad\frac{2p-4}{p}\leq \alpha\leq\frac{2p-3}{p}\,.
\label{e122}
\end{align}
If $p+\e\geq 4$, still with $\e\leq p(1-\alpha)$, namely if $4-p\leq\e\leq p(1-\alpha)$ (so that
$0<\alpha<(2p-4)/p$), then, still invoking \eqref{str-reg-3} we deduce that
the time integrability exponent of the second term on the right hand side of \eqref{e103} is now
given by
\begin{align}
&\mu_p:=(1-\beta)\,\frac{p+\e}{p+\e-3}=\Big(\frac{p}{(2-\alpha)p-3}\Big)^-\,,\qquad
 0< \alpha\leq\frac{2p-4}{p}\,.
 \label{e108}
\end{align}
Thus we find that $\mu_p$ does not depend on the choice of $\e$, if $4-p\leq\e\leq p(1-\alpha)$.
With $4-p\leq\e\leq p(1-\alpha)$ we have that $q$ satisfies $p(2-\alpha)/(1-\alpha)\leq q\leq 4p/(4-p)$,
and, from \eqref{e112}, for $\hat{\mu}_p:=\mu_p/\beta_\ast$ we obtain
\begin{align}
&\hat{\mu}_p=\Big(\frac{q}{(1-\alpha)q-3}\Big)^-\,,\qquad
p\,\frac{2-\alpha}{1-\alpha}\leq q\leq\frac{4p}{4-p}\,,\qquad 0< \alpha\leq\frac{2p-4}{p}\,.\label{e123}
\end{align}




We are left to discuss the case $(2p-3)/p<\alpha<1$, that is, $p(1-\alpha)< 3-p$.
We have that $p/(p-\e)>1/\alpha$ (namely, $\e>p(1-\alpha)$) for all $\e$ such that
$3-p\leq\e<p$. Then it is not difficult to see that the admissible region becomes
\begin{align}
&\mathcal{R}=\Big\{[k,\rho]\in [0,\infty)\times [1,\infty)\,:\, k>3\,,\,\,\rho>\frac{k}{\alpha}\,,\,\,
\rho\geq\frac{p}{p-\e}\,(k-3)-p\,\frac{p+\e-3}{p-\e}\,\Big\}\,.
\label{e119}
\end{align}
Let us compute the infimum of $\beta$ on $\mathcal{R}$. Denoting the affine function
on the right hand side of \eqref{e99} by $g(k)$, we have that $g(k)=k/\alpha$ for
$$k=k^\ast:=\frac{\alpha\, p\,(p+\e)}{\e-(1-\alpha)p}\,,$$
and we can check that
$$\beta(k,g(k))=\frac{p}{q}\,\frac{k-q}{k-2p}\,,\qquad\forall k\geq k^\ast\,,$$
with $k\mapsto \beta(k,g(k))$ (strictly) increasing on $[k^\ast,\infty)$.
By carefully addressing the geometry of $\mathcal{R}$
(notice, in particular, that $\beta(3,\rho)>1/2$, since $\e<p$ implies $q>2p$)
we find
\begin{align}
&\beta_\ast:=\inf_\mathcal{R}\beta=\beta(k^\ast,g(k^\ast))=\frac{1-\alpha}{2-\alpha}\,,\qquad\frac{2p-3}{p}<\alpha<1\,,
\label{e117}
\end{align}
and this infimum is not attained. Notice that, in this case, $\beta_\ast$ does not depend on $\e\in[3-p,p)$.
Moreover, since the exponent $q$ is decreasing with respect to $\e$, we can take $\e=3-p$ to get the best $q$, i.e., $q=3p/(3-p)$.
Now, if $3-p\leq\e\leq 4-p$, owing to \eqref{str-reg-3} and setting $\beta=\beta_\ast^+$, we infer that
the time integrability exponent of the second term on the right hand side of \eqref{e103} is
given by
\begin{align}
&2(1-\beta)\,\frac{p+\e}{p+\e-2}=\Big(\frac{2}{2-\alpha}\Big)^-\,\frac{p+\e}{p-2+\e}\,,\label{e105}
\end{align}
while, if $4-p\leq\e\leq p$, then the time integrability exponent of the second term on the right hand side of \eqref{e103} is
\begin{align}
&(1-\beta)\,\frac{p+\e}{p+\e-3}=\Big(\frac{1}{2-\alpha}\Big)^-\,\frac{p+\e}{p-3+\e}\,.\label{e106}
\end{align}
Observe that the right hand sides of both \eqref{e105} and \eqref{e106} are decreasing in $\e$ on
the intervals $[3-p,4-p]$ and $[4-p,p]$, respectively. Hence, in order to get the best
time integrability exponent for the second term on the right hand side of \eqref{e103} in both cases, it is convenient to take
$\e=3-p$ in \eqref{e105} and $\e=4-p$ in \eqref{e106}. By comparing the two values thus obtained, we get
\begin{align}
&\mu_p=\Big(\frac{6}{2-\alpha}\Big)^-\,,\qquad\frac{2p-3}{p}<\alpha<1\,,
\label{e109}
\end{align}
while, for $\hat{\mu}_p:=\mu_p/\beta_\ast$, and $q$ we have
\begin{align}
&\hat{\mu}_p=\Big(\frac{6}{1-\alpha}\Big)^-\,,\qquad
q=\frac{3p}{3-p}\,,\qquad \frac{2p-3}{p} < \alpha<1\,.\label{e124}
\end{align}

Regarding the second term on the right hand side of \eqref{e52}, on account of Lemma \ref{admiss},
and taking \eqref{e102} into account,
we have that (see \eqref{SOURCE})
\begin{align*}
&\Vert\mathcal{F}_2\Vert_{L^p(\Omega)}\leq C_p+(1+\zeta_\infty)\,b\,\Vert\nabla\varphi\Vert_{L^p(\Omega)^3}
\leq C+C\Vert\varphi\Vert_{H^2(\Omega)}^{1-\frac{2}{p}}\,.
\end{align*}
Invoking \eqref{str-reg-3}, this yields
\begin{align}
&\mathcal{F}_2\in L^{\frac{2p}{p-2}}(0,T;L^p(\Omega))\,,
\label{e110}
\end{align}
and we can check that $2p/(p-2)>\mu_p$, in all the three cases where $\mu_p$ is defined
(see \eqref{e107}, \eqref{e108}, and \eqref{e109}), according with
the value of $\alpha$.

The boundary term \eqref{BDRY} can be handled similarly as for the case $d=2$,
by again obtaining \eqref{e148}, whence we have now that
\begin{align}
&\mathcal{G}\in L^{\frac{2p}{p-2}}(0,T;W^{1-1/p,p}(\Gamma))\,.
\label{e111}
\end{align}
By means of \eqref{e103}, \eqref{e110}, \eqref{e111}, and by fixing $\delta>0$ small enough, estimate \eqref{e52} then yields
\begin{align}
&\pi\in L^{\mu_p}(0,T;W^{2,p}(\Omega))\,,
\label{e115}
\end{align}
with $\mu_p$
given by \eqref{e107} (or \eqref{e108} or \eqref{e109}), according with
the value of $\alpha$.
Moreover, from \eqref{e95} we deduce that
\begin{align}
&\pi\in L^{\hat{\mu}_p}(0,T;W^{1,q}(\Omega))\,,\label{e116}
\end{align}
where $\hat{\mu}_p:=\mu_p/\beta$, and $q$ are given by \eqref{e122} (or \eqref{e123} or \eqref{e124})
according with the value of $\alpha\in(0,1)$.

\medskip

(ii) {\itshape Case $3\leq p<4$.}

\medskip

We argue as at the beginning of the case $2\leq p<3$, taking now $0<\e< 6-p$. Notice that $q>3$, and $q/(q-3)>p/(2p-3)$,
since $p\geq 3$, and $\e>0$. We can thus again see that \eqref{e96} admits a solution $\beta\in (0,1)$ if and only if \eqref{e98}
is satisfied with $\beta$ given by \eqref{e113} (we can again assume that $k>3$). Since $0<\e<6-p\leq p$ (so that $2q>p$, being $\e<p$),
we obtain once more that \eqref{e99} ensures \eqref{e97}. Thus we observe that the slopes
of the affine functions on the right hand sides of \eqref{e98} and of \eqref{e99}
still satisfy $q/(q-3)< p/(p-\e)$. Let us now take $\e$ satisfying, in addition, the condition
$0<\e\leq p(1-\alpha)$ (hence, $p/(p-\e)\leq 1/\alpha$). Then the admissible region
is still given by \eqref{e114}, with the (not attained) infimum of $\beta$ over $\mathcal{R}$
still given by \eqref{e112}. We now distinguish two cases. If $(2p-4)/p\leq \alpha<1$,
then $p(1-\alpha)\leq 4-p$, and, on account of $0<\e\leq p(1-\alpha)\leq 4-p$, we get
that the time integrability exponent of the second term on the right hand side of \eqref{e103}
is given by \eqref{e104}. We again choose $\e=p(1-\alpha)$ to get the best $\mu_p$,
which is given by
\begin{align}
&\mu_p=\Big(\frac{2p}{(2-\alpha)p-2}\Big)^-\,,\quad\frac{2p-4}{p}\leq \alpha<1\,.
\label{e125}
\end{align}
Moreover, since $0<\e\leq p(1-\alpha)$, then $p(2-\alpha)/(1-\alpha)\leq q<\infty$, and $\hat{\mu}_p=\mu_p/\beta_\ast$
is given by
\begin{align}
&\hat{\mu}_p=\Big(\frac{(2-\alpha)p-3}{(2-\alpha)p-2}\Big)^-\,\frac{2q}{(1-\alpha)q-3}\,,\qquad
p\,\frac{2-\alpha}{1-\alpha}\leq q<\infty\,,\qquad\frac{2p-4}{p}\leq \alpha<1\,.\label{e126}
\end{align}
If, on the other hand, $0<\alpha\leq(2p-4)/p$ (i.e. $p(1-\alpha)\geq 4-p$),
 then the time integrability exponent of the second term on the right hand side of \eqref{e103}
is given by \eqref{e104}, if $0<\e\leq 4-p$, or by \eqref{e108}, if $4-p\leq\e\leq p(1-\alpha)$.
We hence infer that the best $\mu_p$ we get for this case is given by \eqref{e108}.
{However, differently from the case $2\leq p<3$, we now need an additional
condition which guarantees that $\mu_p>1$, namely that
$p/((2-\alpha)p-3)>1$ or $\alpha>(p-3)/p$. Therefore, we have
\begin{align}
&\mu_p=\Big(\frac{p}{(2-\alpha)p-3}\Big)^-\,,\qquad
 \frac{p-3}{p}< \alpha\leq\frac{2p-4}{p}\,.
 \label{e127}
\end{align}
Moreover, for $\hat{\mu}_p=\mu_p/\beta_\ast$ and $q$, we get
\begin{align}
&\hat{\mu}_p=\Big(\frac{q}{(1-\alpha)q-3}\Big)^-\,,\qquad
p\,\frac{2-\alpha}{1-\alpha}\leq q<\infty\,,\qquad \frac{p-3}{p}< \alpha\leq\frac{2p-4}{p}\,.
\label{e128}
\end{align}
In conclusion, for $3\leq p<4$, the interval $(0,1)$ is not entirely admissible for $\alpha$
(unless $p=3$), and
 we distinguish two cases instead of three, namely
\eqref{e115} and \eqref{e116} hold with $\mu_p$ given by\footnote{We can check that
$\mu_p < 2p/(p-2)>$, for both cases of $\mu_p$ given by \eqref{e125}, \eqref{e127} (see \eqref{e110}-\eqref{e111}).}
 \eqref{e125} or by \eqref{e127}, and $\hat{\mu}_p=\mu_p/\beta_\ast$,
 $q$ given by \eqref{e126} or by \eqref{e128}, according with $\alpha\in((p-3)/p,1)$.

\medskip

(iii) {\itshape Case $4\leq p<6$.}

\medskip

We again argue as at the beginning of the previous cases $2\leq p<3$ and $3\leq p<4$, taking now $0<\e\leq 6-p<p$.
Notice that, since $p\geq 4$, then $p+\e>4$ and hence only the second line on the right hand side of
\eqref{e103} can be employed to estimate the $L^p$- norm of $\mathcal{F}_1$ to get
the time integrability exponent $\mu_p$ in \eqref{e115}. Let us begin to take also $0<\e\leq p(1-\alpha)$,
namely $0<\e\leq \min(6-p,p(1-\alpha))$. As we saw in the discussion for the case $2\leq p<3$,
with this choice of $\e$ we have that the admissible region $\mathcal{R}$ is given by
\eqref{e114}, with the (not attained) infimum $\beta_\ast$ of $\beta$ over $\mathcal{R}$
again given by \eqref{e112}. By combining \eqref{e112} with the exponent in the second term in
the second line on the right hand side of \eqref{e103}, we thus get (see \eqref{e108})
\begin{align}
&\mu_p:=(1-\beta)\,\frac{p+\e}{p+\e-3}=\Big(\frac{p}{(2-\alpha)p-3}\Big)^-\,.\label{e118}
\end{align}
Then $\mu_p$ is independent of $\e$. We have that $\mu_p>1$ for $(p-3)/p<\alpha<1$.
We now distinguish the following cases. If $0<\alpha\leq 2(p-3)/p$, then $6-p\leq p(1-\alpha)$. So that
$0<\e\leq 6-p$ implies $6p/(6-p)\leq q<\infty$, with $\beta_\ast$ given by \eqref{e112} as a function of $q$.
We then obtain that $\hat{\mu}_p:=\mu_p/\beta_\ast$ is given by
\begin{align}
&\hat{\mu}_p=\frac{q}{(1-\alpha)q-3}\,,\quad\frac{6p}{6-p}\leq q<\infty\,,\quad\frac{p-3}{p}<\alpha\leq\frac{2(p-3)}{p}\,.
\label{e121}
\end{align}
If, on the other hand, $2(p-3)/p\leq \alpha<1$ then $p(1-\alpha)\leq 6-p$. In this case, if $0<\e\leq p(1-\alpha)$ then
$\beta_\ast$ is still given by \eqref{e112} yielding $\mu_p$ as given by \eqref{e118}. Moreover,
we have $p(2-\alpha)/(1-\alpha)\leq q<\infty$ and $\hat{\mu}_p:=\mu_p/\beta_\ast$ again given
by\footnote{Still under the condition $2(p-3)/p\leq \alpha<1$, if we also consider the case $p(1-\alpha)<\e\leq 6-p$, then, recalling the discussion
carried out for the case $2\leq p<3$, the admissible region $\mathcal{R}$ now becomes \eqref{e119},
with the (not attained) infimum $\beta_\ast$ of $\beta$ over $\mathcal{R}$ given by $\beta_\ast=(1-\alpha)/(2-\alpha)$
(cf. \eqref{e117}). Hence, for $\mu_p$ we get the same as in \eqref{e106} (which is decreasing in $\e$),
and we choose $\e=(p(1-\alpha))^+$ to get the best $\mu_p$, getting the same $\mu_p$ as in \eqref{e118}.
Moreover, for $\hat{\mu}_p:=\mu_p/\beta_\ast$ we get
\begin{align*}
&\hat{\mu}_p=\frac{2-\alpha}{1-\alpha}\,\Big(\frac{p}{(2-\alpha)p-3}\Big)^-\,,
\end{align*}
and for $q$ we can take the best exponent for $p(1-\alpha)<\e\leq 6-p$, namely $q=(p(2-\alpha)/(1-\alpha))^-$.
Comparing with \eqref{e120} (take $q=p(2-\alpha)/(1-\alpha)$), we thus conclude that addressing the case
$p(1-\alpha)<\e\leq 6-p$ does not improve $\mu_p$.}
\begin{align}
&\hat{\mu}_p=\frac{q}{(1-\alpha)q-3}\,,\quad p\,\frac{2-\alpha}{1-\alpha}\leq q<\infty\,,\quad\frac{2(p-3)}{p}\leq \alpha<1\,.
\label{e120}
\end{align}

Summing up, in the case $4\leq p<6$, for $\mu_p$, $\hat{\mu}_p$, $q$ in \eqref{e115}, \eqref{e116}
 we have obtained the corresponding values\footnote{We can check that $\mu_p<2p/(p-2)$,(see \eqref{e110}-\eqref{e111})}
 \begin{align}
&\mu_p=\Big(\frac{p}{(2-\alpha)p-3}\Big)^-,\qquad\mbox{if }\,\,\frac{p-3}{p}<\alpha<1\,,\label{e131}
\end{align}
and $\hat{\mu}_p$, $q$ given by \eqref{e121}, \eqref{e120}, according with the value of
$\alpha\in((p-3)/p,1)$. As for the case $3\leq p<4$, we observe that the interval $(0,1)$ is again not entirely admissible for $\alpha$.

\medskip

(iv) {\itshape Case $p=6$.}

\medskip

In this case we can only take $\e=0$ and $q=\infty$ in estimates \eqref{e100} and \eqref{e95}, since the maximum spatial integrability exponent for $\nabla\varphi$ is $6$. Doing so (and arguing as for the case (i)) we get
$\beta(k,\rho)=(2\rho-2(k-3))/(3\rho-2(k-3))$, with conditions \eqref{e98} and \eqref{e99} that are now equivalent
to $\rho>k-3$ (still taking $k>3$). The admissible region $\mathcal{R}$ is still \eqref{e114}
and the infimum $\beta_\ast$ of $\beta$ over $\mathcal{R}$ is now
\begin{align*}
&\beta_\ast=\frac{2(1-\alpha)}{3-2\alpha}\,.
\end{align*}
The time integrability exponent of the second term on the right hand side of \eqref{e103} (written for $p=6$ and $\e=0$)
is then given by\footnote{We can check that $\mu_6 <3$ (see \eqref{e110}-\eqref{e111} for $p=6$
so that $2p/(p-2)=3$).}
\begin{align}
&\mu_6=\Big(\frac{2}{3-2\alpha}\Big)^-\,,\label{e129}
\end{align}
and $\mu_6>1$ provided that $1/2<\alpha<1$. For $\hat{\mu}_6:=\mu_6/\beta_\ast$ we have
\begin{align}
&\hat{\mu}_6=\Big(\frac{1}{1-\alpha}\Big)^-\,,\qquad \frac{1}{2}<\alpha<1\,.\label{e130}
\end{align}
Hence, \eqref{e115} and \eqref{e116} hold with $p=6$, with $\mu_6$, $\hat{\mu}_6$ given by
\eqref{e129} and \eqref{e130}, respectively, and with $q=\infty$.
Notice that the case $p=6$ can be considered as the limit case of (iii) for $p\to 6^-$ and $q\to\infty$.

\medskip

Regarding the regularity of $\uvec$ in all cases (i)-(iv) considered above, we use \eqref{e64} and we focus
on the second term on the right hand side which is the less obvious.
This term can be estimated similarly as $\mathcal{F}_1$ (cf. \eqref{e103}), namely
\begin{align}
&\Big\Vert\frac{\zeta(\varphi)}{\eta(\varphi)}\partial_j\varphi\,\partial_k\pi\Big\Vert_{L^p(\Omega)}
\leq \frac{\zeta_\infty}{\eta_1}\Vert\nabla\varphi\Vert_{L^{p+\e}(\Omega)^3}\Vert\nabla\pi\Vert_{L^q(\Omega)^3}
\leq C\Vert\nabla\varphi\Vert_{L^{p+\e}(\Omega)^3}\Vert\pi\Vert_{W^{2.p}(\Omega)}^\beta\,.\label{e132}
\end{align}
Invoking \eqref{e103}, where $\delta$ is supposed to be fixed small enough, we see that the time integrability exponent of the right hand side of \eqref{e132} coincides with the time integrability exponent of the second term on the right hand side
of \eqref{e103}, which is $\mu_p$ for all choices of $p\in[2,6]$ and $\alpha$ considered in the above cases (i)-(iv).
The estimates in $L^p$
of the last two terms on the right hand side of \eqref{e64} is straightforward
(see \eqref{e158})
Therefore, noting that we always have $\hat{p}>\mu_p$, from \eqref{e64} we deduce that}
\begin{align}
&\uvec\in L^{\mu_p}(0,T;W^{1,p}(\Omega)^3)\cap L^{\hat{\mu}_p}(0,T;L^{q}(\Omega)^3)\,,\qquad 2\leq p\leq 6\,,
\end{align}
where $\mu_p$, $\hat{\mu}_p$, and $q$ are given in terms of $p$ and $\alpha$ by the relations
and constraints deduced in the discussion carried out in the above cases (i)-(iv).

Finally, if $\eta$ is a positive constant, for both cases $d=2,3$,
the regularity analysis of the elliptic system \eqref{ell-1}-\eqref{ell-2}, as well as of equation \eqref{e64} for $\nabla\uvec$,
gets much simpler.
Indeed, we have that $\mathcal{F}_1=0$ (see \eqref{SOURCE0} and \eqref{SOURCE}) so the only terms
which survive in the elliptic estimate \eqref{e52} are the norms of $\mathcal{F}_2$ and of
$\mathcal{G}$. If $d=2$, \eqref{e146}-\eqref{e149}, and \eqref{e64} immediately yield that
\begin{align}
&\pi\in L^{\hat{p}}(0,T;W^{2,p}(\Omega))\,,\qquad\uvec\in L^{\hat{p}}(0,T;W^{1,p}(\Omega)^2)\,,\qquad
2\leq p<\infty\,,\label{e184}
\end{align}
where $\hat{p}:=2p/(p-2)$, if $2<p<\infty$, and $\hat{p}=\infty$, if $p=2$.
Assume now that
$d=3$. Since the $L^p(\Omega)$-norm of $\mathcal{F}_2$ and the $W^{1-1/p,p}(\Gamma)$-norm
of $\mathcal{G}$ can be both  controlled by $\Vert\nabla\varphi\Vert_{L^p(\Omega)^3}$
(cf. \eqref{e146} and \eqref{e148}), then, by employing \eqref{e101} and \eqref{e102}, we obtain
\begin{align}
&\pi\in L^{\hat{p}}(0,T;W^{2,p}(\Omega))\,,\quad\uvec\in L^{\hat{p}}(0,T;W^{1,p}(\Omega)^3)\,,
\label{e185}
\end{align}
if $2\leq p\leq 4$, and
\begin{align}
&\pi\in L^{\frac{p}{p-3}}(0,T;W^{2,p}(\Omega))\,,\quad\uvec\in L^{\frac{p}{p-3}}(0,T;W^{1,p}(\Omega)^3)\,,
\label{e186}
\end{align}
if $4\leq p\leq 6$. Observe that, for both $d=2,3$, we have $\pi\in L^\infty(0,T;H^2(\Omega))$.
Therefore \eqref{e181}-\eqref{e183} hold. The values of $\hat{\sigma}_p$, $\hat{\mu}_p$ and $q$ are obtained
in a straightforward fashion from \eqref{e184}-\eqref{e186} by using classical Sobolev embeddings.
The proof is finished.
\end{proof}
\begin{oss}\label{eta-cost-reg} If $\eta$ is a positive constant, the regularity properties for $\pi$ and $\uvec$ derived in
Theorem \ref{str-sols-fur-reg} hold true also for weak solutions. This is a direct consequence of Darcy's law and
of the properties of the Helmholtz projector operator $P_r$ from $L^r(\Omega)^d$ to $L^r_{div}(\Omega)^d$ associated with
the decomposition $L^r(\Omega)^d=L^r_{div}(\Omega)^d\oplus G_r$,
where $G_r:=\{\wvec\in L^r(\Omega)^d:\,\wvec=\nabla\pi\,\,\mbox{for some }\pi\in W^{1,r}(\Omega)\}$.
We recall that this decomposition is valid for $\Omega$ with locally Lipschitz boundary, if $r\neq 2$, and for all domains $\Omega$,
if $r=2$ (see \cite{FM}, see also \cite[Theorem III 1.2]{Galdi}). If $\uvec\in W^{m,r}(\Omega)^d$ ($m\geq 0$), then
$P_r\uvec\in W^{m,r}(\Omega)^d\cap L^r_{div}(\Omega)^d$, and
\begin{align}
&\Vert P_r\uvec\Vert_{W^{m,r}(\Omega)^d}\leq C_{m,r}\Vert\uvec\Vert_{W^{m,r}(\Omega)^d}\,,\label{e192}
\end{align}
with $C_{m,r}>0$ independent of $\uvec$ (cf. \cite[Lemma 3.3]{GMi}). Indeed, by applying Helmholtz projector operator $P_r$
to Darcy's law
\eqref{e26} with $\eta$ constant, and by taking \eqref{e192} into account, we get
\begin{align}
\Vert\uvec\Vert_{W^{m,r}(\Omega)^d}\leq C \Vert(\nabla J\ast\varphi)\varphi\Vert_{W^{m,r}(\Omega)^d}\,,\quad
m\geq 0\,,\,\,\,\,1<r<\infty\,.\label{e193}
\end{align}
Hence, \eqref{e175}, \eqref{e179}, \eqref{e180}, \eqref{e165}, together with \eqref{e181}-\eqref{e183} follow
from \eqref{e193} also for a weak solution.
\end{oss}

\section{Weak-strong uniqueness}
\setcounter{equation}{0}
\label{sec:unique}

In two dimensions we can prove a continuous weak-strong dependence estimate which entails weak-strong uniqueness.

\begin{thm}
\label{weak-strong-2D}
Let $d=2$. Suppose that \textbf{(H1)}-\textbf{(H8)}
are satisfied and that $J\in W_{loc}^{2,1}(\mathbb{R}^{2})$ or $J$ is
admissible. Let $\varphi _{01}\in L^{\infty }(\Omega )$ and $\varphi _{02}\in V\cap L^{\infty }(\Omega )$,
with $M(\varphi _{01}), M(\varphi_{02})\in L^{1}(\Omega )$, where $M$ is defined as in Theorem \ref{existweak}.
For any given $T>0$, denote by $[\uvec_1,\pi_1,\varphi_1]$ be a weak solution and by $[\uvec_2,\pi_2,\varphi_2]$ be a strong solution
to problem \eqref{Sy01}-\eqref{Sy04}, \eqref{sy5}, \eqref{sy6} on $[0,T]$, corresponding to $\varphi_{01}$ and to $\varphi_{02}$,
and given by Theorem \ref{existweak} and by Theorem \ref{strong-sols},
respectively. %
Then, the following estimate holds
\begin{align}
&\Vert\uvec_2-\uvec_1\Vert_{L^2(0,t;G_{div})}+\Vert\varphi_2-\varphi_1\Vert_{L^{\infty}(0,t;H)\cap L^2(0,t;V)}
+\Vert\pi_2-\pi_1\Vert_{L^2(0,t;V_0)} \nonumber\\
&\leq \tilde\Lambda(t)\Vert\varphi_{02}-\varphi_{01}\Vert\,, \label{stab-est}
\end{align}
for all $t\in[0,T]$, where $\tilde\Lambda$ is a continuous function which depends on some norms of the strong solution.
\end{thm}

\begin{proof}
Let us first take the difference between the two identities \eqref{e26} written for the weak and the strong solutions,
multiply it by $\uvec:=\uvec_2-\uvec_1$ and integrate over $\Omega$. Then, setting $\varphi:=\varphi_2-\varphi_1$, we get
\begin{align}
&\big((\eta(\varphi_2)-\eta(\varphi_1))\uvec_2,\uvec\big)
+(\eta(\varphi_1)\uvec,\uvec)
=\big((\nabla J\ast\varphi)\varphi_2,\uvec\big)
+\big((\nabla J\ast\varphi_1)\varphi,\uvec\big)\,.\label{e160}
\end{align}
From this identity, on account of \textbf{(H1)}, we have that
\begin{align}
\eta_1\Vert\uvec\Vert^2&\leq C\Vert\varphi\Vert_{L^4(\Omega)}\Vert\uvec_2\Vert_{L^4(\Omega)^2}
\Vert\uvec\Vert+C\Vert\varphi\Vert\Vert\uvec\Vert\nonumber\\[1mm]
&\leq C(\Vert\varphi\Vert+\Vert\varphi\Vert^{1/2}\Vert\nabla\varphi\Vert^{1/2})
\Vert\uvec_2\Vert_{L^4(\Omega)^2}
\Vert\uvec\Vert+C\Vert\varphi\Vert\Vert\uvec\Vert\nonumber\\[1mm]
&\leq \frac{\eta_1}{2}\Vert\uvec\Vert^2+\delta\Vert\nabla\varphi\Vert^2
+C_{\delta}(1+\Vert\uvec_2\Vert_{L^4(\Omega)^2}^4)\Vert\varphi\Vert^2\,,\label{e161}
\end{align}
which gives
\begin{align}
&\eta_1\Vert\uvec\Vert^2\leq 2\delta\Vert\nabla\varphi\Vert^2
+C_{\delta}(1+\Vert\uvec_2\Vert_{L^4(\Omega)^2}^4)\Vert\varphi\Vert^2\,,
\label{e46}
\end{align}
with $\delta>0$ to be fixed later. We recall that, here and in the sequel of this section,
$C$ stands for a generic positive constant which only depends
on main constants of the problem (see (\textbf{H1})-(\textbf{H8})) and on $\Omega$ at most. Any other dependency
will be explicitly pointed out.

We now take the difference of \eqref{E2} written for the weak and the strong solutions (see also \eqref{primitive}).
Taking then $\varphi$ as test function, we obtain
\begin{align}
& \frac{1}{2}\frac{d}{dt}\Vert \varphi \Vert ^{2}+\big(\nabla (B(\varphi
_{2})-B(\varphi _{1})),\nabla \varphi \big)+(\uvec\cdot\nabla\varphi_2,\varphi)\notag \\[1mm]
& =\big((m(\varphi _{2})-m(\varphi _{1}))(\nabla J\ast \varphi _{2}),\nabla
\varphi \big)+\big(m(\varphi _{1})(\nabla J\ast \varphi ),\nabla \varphi 
\big)\,.  \label{e45}
\end{align}%
Thanks to \textbf{(H7)}, we have that
\begin{equation}
\big(\nabla (B(\varphi _{2})-B(\varphi _{1})),\nabla \varphi \big)\geq
\alpha _{0}\Vert \nabla \varphi \Vert ^{2}+\big((\lambda (\varphi
_{2})-\lambda (\varphi _{1}))\nabla \varphi _{2},\nabla \varphi \big)\,,
\label{e44}
\end{equation}%
and, in view of the regularity \eqref{str-reg-3} for $\varphi _{2}$, the second term on the right hand side of \eqref{e44} can be estimated as in
\cite[Proof of Theorem 6.1, Part (c)]{FGGS} by means of the Gagliardo-Nirenberg inequality, namely as
\begin{align}
\big|\big((\lambda (\varphi _{2})-\lambda (\varphi _{1}))\nabla \varphi
_{2},\nabla \varphi \big)\big|&\leq C\Vert\varphi\Vert_{L^4(\Omega)}\Vert\nabla\varphi_2\Vert_{L^4(\Omega)^2}
\Vert\nabla\varphi\Vert\notag \\[1mm]
& \leq C(\Vert \varphi \Vert +\Vert \varphi \Vert ^{1/2}\Vert \nabla \varphi
\Vert ^{1/2})\Vert \varphi _{2}\Vert
_{H^{2}(\Omega )}^{1/2}\Vert \nabla \varphi \Vert\notag \\[1mm]
& \leq \delta^\prime\Vert \nabla \varphi \Vert
^{2}+C_{\delta^\prime}(1+\Vert \varphi _{2}\Vert _{H^{2}(\Omega )}^{2})\Vert \varphi \Vert
^{2}\,,\label{e47}
\end{align}%
with $\delta^\prime>0$ to be fixed later.

As far as the third term on the left hand side of \eqref{e45} is concerned, we have
\begin{align}
 |(\uvec\cdot\nabla\varphi_2,\varphi)|&\leq
\Vert\uvec\Vert\Vert\nabla\varphi_2\Vert_{L^4(\Omega)^2}\Vert\varphi\Vert_{L^4(\Omega)}
\leq  (\Vert\varphi\Vert+\Vert\varphi\Vert^{1/2}\Vert\nabla\varphi\Vert^{1/2})\Vert\varphi_2\Vert_{H^2(\Omega)}^{1/2}
\Vert\uvec\Vert\nonumber\\[1mm]
&\leq \delta\Vert\uvec\Vert^2+\delta^\prime\Vert\nabla\varphi\Vert^2
+C_{\delta,\delta^\prime}(1+\Vert\varphi_2\Vert_{H^2(\Omega)}^2)\Vert\varphi\Vert^2\,.\label{e48}
\end{align}
On the other hands, the two terms on the right hand side of \eqref{e45} can be controlled as follows
\begin{align}
&\big|\big((m(\varphi _{2})-m(\varphi _{1}))(\nabla J\ast \varphi _{2}),\nabla
\varphi \big)\big|+\big|\big(m(\varphi _{1})(\nabla J\ast \varphi ),\nabla \varphi\big)\big|
\leq \delta^\prime\Vert\nabla\varphi\Vert^2+C_{\delta^\prime}\Vert\varphi\Vert^2\,.\label{e49}
\end{align}
Hence, adding together \eqref{e46} with \eqref{e45}, taking \eqref{e44}-
\eqref{e49} into account, and choosing $\delta,\delta^\prime$ suitably small, we find
\begin{align*}
&\frac{d}{dt}\Vert\varphi\Vert^2+\eta_1\Vert\uvec\Vert^2+\alpha_0\Vert\nabla\varphi\Vert^2
\leq C\big(1+\Vert\uvec_2\Vert_{L^4(\Omega)^2}^4+\Vert\varphi_2\Vert_{H^2(\Omega)}^2\big)\Vert\varphi\Vert^2\,.
\end{align*}
Thus an application of the Gronwall lemma and an integration in time yield
\begin{align}
&\Vert\varphi(t)\Vert^2+\eta_1\int_{0}^{t}\Vert\uvec(\tau)\Vert^2\,d\tau+%
\alpha_0\int_{0}^{t}\Vert\nabla\varphi(\tau)\Vert^2\,d\tau
\leq \Lambda(t)\Vert\varphi_{02}-\varphi_{01}\Vert^2\,,\label{e50}
\end{align}
where the continuous function $\Lambda$ depends on norms of the strong solution. More precisely, we can take $\Lambda(t)=1+\int_{0}^{t}\alpha(\tau)e^{\int_{0}^{\tau}\alpha(s)ds}\,d\tau$,
with $\alpha(t):=C\big(1+\Vert\uvec_2(t)\Vert_{L^4(\Omega)^2}^4+\Vert\varphi_2(t)\Vert_{H^2(\Omega)}^2\big)$.

Concerning the pressure, setting $\pi:=\pi_2-\pi_1$, from \eqref{e26} we have that
\begin{align}
&\nabla\pi=-\big(\eta(\varphi_2)-\eta(\varphi_1)\big)\uvec_2-\eta(\varphi_1)\uvec+(\nabla J\ast\varphi)\varphi_2+%
(\nabla J\ast\varphi_1)\varphi\,.\label{e169}
\end{align}
Therefore we get
\begin{align}
\Vert\nabla\pi\Vert_{L^2(0,t;L^2(\Omega)^2)}
&\leq C\Vert\uvec_2\Vert_{L^4(0,t;L^4(\Omega)^2)}\Vert\varphi\Vert_{L^4(0,t;L^4(\Omega))}
+C\Vert\uvec\Vert_{L^2(0,t;G_{div})}+C\Vert\varphi\Vert_{L^2(0,t;H)}\notag\\[1mm]
&\leq C\left(\Vert\uvec_2\Vert_{L^4(0,t;L^4(\Omega)^2)}\Vert\varphi\Vert_{L^\infty(0,t;H)\cap L^2(0,t;V)}%
+\Vert\uvec\Vert_{L^2(0,t;G_{div})}+\Vert\varphi\Vert_{L^2(0,t;H)}\right)\,.\label{e51}
\end{align}
Estimate \eqref{stab-est} follows from \eqref{e50} and \eqref{e51}.
\end{proof}

The above result can be extended to the case $d=3$ provided that $\lambda$ is constant which is nonetheless the reference case (see Remark \ref{lambdaconst}).
This extension is conditional since we need to require that the pressure of the strong solution has a spatial H\"{o}lder continuity exponent $\alpha\in(1/5,1)$. Recall that $\pi_2$ satisfies the
elliptic problem \eqref{pb1-eq-bis}-\eqref{pb2-eq-bis} (with $\varphi_2$ in place of $\varphi$).
Notice that, since $|\varphi_2|\leq 1$, and $\eta(\varphi_2)$ is bounded from below and above
by positive constants, Proposition \ref{DiBenedetto} only
ensures that $\alpha$ depends on $\eta_1,\eta_\infty, b, d,\Omega$, and on the
geometrical properties of $\Gamma$, but it does not depend on the (unknown) form of $\varphi_2$
(this motivates the notation $\alpha$, instead of $\alpha_2$).
Therefore, although the result we are going to prove is conditional, $\alpha$
depends on structural constants of the problem only. In this case the key tool for the proof is Theorem \ref{str-sols-fur-reg}.
However, if $\eta$ is constant then uniqueness of weak solutions holds.

We have

\begin{thm}\label{weak-strong-3D}
Let $d=3$.  Suppose that $[\uvec_1,\pi_1,\varphi_1]$ and $[\uvec_2,\pi_2,\varphi_2]$ are solutions corresponding, respectively,
to initial data $\varphi_{01}$ and $\varphi_{02}$ as in Theorem \ref{weak-strong-2D}.
In addition assume that $\lambda$ is constant and that
the spatial H\"{o}lder continuity exponent $\alpha$
of $\pi_2$ is such that $\alpha\in (1/5,1)$.
Then \eqref{stab-est} still holds.

Suppose now that $\lambda$ satisfies (\textbf{H4}) and $\eta$ is constant. If $[\uvec_1,\pi_1,\varphi_1]$ and $[\uvec_2,\pi_2,\varphi_2]$ are
weak solutions corresponding, respectively, to initial data $\varphi_{01}$ and $\varphi_{02}$ as in Theorem \ref{existweak},
then the following stability estimate holds
\begin{align}
&\Vert\uvec_2-\uvec_1\Vert_{L^2(0,t;G_{div})}+\Vert\varphi_2-\varphi_1\Vert_{L^{\infty}(0,t;V^\prime)\cap L^2(0,t;H)}
+\Vert\pi_2-\pi_1\Vert_{L^2(0,t;V_0)} \nonumber\\
&\leq \hat{\Lambda}(t)\Vert\varphi_{02}-\varphi_{01}\Vert_{V^\prime}\,, \label{stab-est-2}
\end{align}
for all $t\in[0,T]$, where $\hat{\Lambda}$ is a continuous function which depends on the norms of one of the weak solutions.

\end{thm}

\begin{proof}
Suppose $\eta$ not constant first.
Consider \eqref{e160} and observe that $\eta(\varphi_2)-\eta(\varphi_1)$ will be estimated differently. Namely, instead of \eqref{e161}, now, by employing
\eqref{GN-3D}, we have that
\begin{align}
\eta_1\Vert\uvec\Vert^2 &\leq
C\Vert\varphi\Vert_{L^{2r^\prime}(\Omega)}\Vert\uvec_2\Vert_{L^{2r}(\Omega)^3}\Vert\uvec\Vert
+C\Vert\varphi\Vert\Vert\uvec\Vert\nonumber\\[1mm]
&\leq C\big(\Vert\varphi\Vert+\Vert\varphi\Vert^{\frac{2r-3}{2r}}\Vert\nabla\varphi\Vert^{\frac{3}{2r}}\big)
\Vert\uvec_2\Vert_{L^{2r}(\Omega)^3}\Vert\uvec\Vert
+C\Vert\varphi\Vert\Vert\uvec\Vert\nonumber\\[1mm]
&\leq \frac{\eta_1}{2}\Vert\uvec\Vert^2+ C(1+\Vert\uvec_2\Vert_{L^{2r}(\Omega)^3}^2)\Vert\varphi\Vert^2
+C\Vert\uvec_2\Vert_{L^{2r}(\Omega)^3}^2\Vert\varphi\Vert^{\frac{2r-3}{r}}\Vert\nabla\varphi\Vert^{\frac{3}{r}}\nonumber\\[1mm]
&\leq \frac{\eta_1}{2}\Vert\uvec\Vert^2+\delta\Vert\nabla\varphi\Vert^2
+C_\delta(1+\Vert\uvec_2\Vert_{L^{2r}(\Omega)^3}^\frac{4r}{2r-3})\Vert\varphi\Vert^2\,,
\label{e163}
\end{align}
where $3/2<r\leq 3$, and $\delta>0$ to be fixed later.

Consider now \eqref{e45}. On account of the fact that $\lambda$ is now constant, using \eqref{e49} and noting that
$(\uvec\cdot\nabla\varphi_2,\varphi)=-(\varphi_2\uvec,\nabla\varphi)$, we get
\begin{align}
\frac{1}{2}\frac{d}{dt}\Vert \varphi \Vert ^{2}+\frac{\alpha_0}{2}\Vert\nabla\varphi\Vert^2
&\leq |(\varphi_2\uvec,\nabla\varphi)|+C\Vert\varphi\Vert^2\,.
\label{e162}
\end{align}%
Let us multiply \eqref{e162} by a positive coefficient $\gamma$ to be fixed later,
and sum the resulting inequality with \eqref{e163}, where $\delta=\gamma^2/\eta_1$.
This gives
\begin{align}
\frac{\gamma}{2}\frac{d}{dt}\Vert \varphi \Vert ^{2}&+
\frac{\eta_1}{2}\Vert\uvec\Vert^2+\frac{\alpha_0\gamma}{2}\Vert\nabla\varphi\Vert^2\nonumber\\
&\leq
\gamma|(\varphi_2\uvec,\nabla\varphi)|
+\gamma C\Vert\varphi\Vert^2
+\delta\Vert\nabla\varphi\Vert^2 +C_\delta(1+\Vert\uvec_2\Vert_{L^{2r}(\Omega)^3}^\frac{4r}{2r-3})\Vert\varphi\Vert^2 \nonumber\\
&\leq \frac{\eta_1}{4}\Vert\uvec\Vert^2+\frac{2\gamma^2}{\eta_1}\Vert\nabla\varphi\Vert^2+
C_{\gamma}\,(1+\Vert\uvec_2\Vert_{L^{2r}(\Omega)^3}^\frac{4r}{2r-3})\Vert\varphi\Vert^2\,.
\end{align}
Fixing now $\gamma>0$ such that $\gamma<\alpha_0\eta_1/4$ (e.g., choosing $\gamma=\alpha_0\eta_1/8$), we then find
\begin{align}
&\frac{d}{dt}\Vert \varphi \Vert ^{2}+
\frac{8}{\alpha_0}\Vert\uvec\Vert^2+\frac{\alpha_0}{2}\Vert\nabla\varphi\Vert^2\leq
C\,(1+\Vert\uvec_2\Vert_{L^{2r}(\Omega)^3}^\frac{4r}{2r-3})\Vert\varphi\Vert^2\,.\label{e168}
\end{align}
Therefore, in order to apply the Gronwall lemma we need
\begin{align}
&\uvec_2\in L^{\frac{4r}{2r-3}}(0,T;L^{2r}(\Omega)^3)\,,\label{e164}
\end{align}
for some $r\in(3/2,3]$. We now exploit the regularity properties for
$\uvec$ established in Theorem \ref{str-sols-fur-reg}.
If $\eta$ is a positive constant then condition \eqref{e164}
is immediately satisfied. Indeed, take, e.g., $p=2$ in \eqref{e183} and get
$\uvec_2\in L^\infty(0,T;L^6(\Omega)^3)$, which fulfills \eqref{e164} with $r=3$.
If $\eta$ is not constant, then we employ the regularity properties for $\uvec_2$
expressed by Theorem \ref{str-sols-fur-reg} in terms
of the H\"{o}lder continuity exponent $\alpha\in(0,1)$ of $\pi_2$. Namely, we
aim to find a condition on $\alpha$ ensuring that \eqref{e164} holds for some $r\in(3/2,3]$.
Let us take $r$ such that $4r/(2r-3)=2r$, which means $r=5/2$, and let us 
look for a lower bound on $\alpha$ ensuring that $\uvec_2\in L^5(0,T;L^5(\Omega)^3)$. 
To this purpose, we
consider the case $2\leq p<3$ in Theorem \ref{str-sols-fur-reg} and look for $p\in [2,3)$ and
$\alpha$ such that (see \eqref{e165})
\begin{align}
&\hat{\mu}_p=q\geq 5\,,\qquad \frac{2p-4}{p}\leq\alpha\leq\frac{2p-3}{p}\,.\label{e167}
\end{align}
By means of the second line in \eqref{e166}, taking $q=p(2-\alpha)/(1-\alpha)$, we have that
$\hat{\mu}_p=q$ if and only if
$$\Big(\frac{2}{(2-\alpha)p-2}\Big)^-=1\,,$$
which holds if and only if $(2-\alpha)p=4^-$, that is, if and only if $\alpha=((2p-4)/p)^+$, which
is acceptable (see \eqref{e167}).
For this value of $\alpha$ we find that $q=(4p/(4-p))^+$. Thus the first condition
in \eqref{e167} is satisfied by taking $p=20/9$. This gives $\alpha=(1/5)^+$.
Therefore, we conclude that\footnote{Addressing the other intervals
for $p$ considered in Theorem \ref{str-sols-fur-reg}, 
to require that $\uvec_2\in L^5(0,T;L^5(\Omega)^3)$,
does not improve the lower bound $1/5$. The details are left to the reader.}
if $\alpha>1/5$ then $\uvec_2\in L^5(0,T;L^5(\Omega)^3)$. We can now apply the Gronwall lemma to \eqref{e168}
and we find \eqref{e50} with $\Lambda$ suitably modified.

We are left to estimate $\pi$. Arguing as for the case $d=2$ and writing
\eqref{e169} for $\nabla\pi$, the only term which is handled differently
is the first one on the right hand side, which is now estimated in $L^2$ as follows
\begin{align}
\big\Vert\big(\eta(\varphi_2)-\eta(\varphi_1)\big)\uvec_2\big\Vert_{L^2(0,t;H^3)}
&\leq C\Vert\varphi\Vert_{L^{10/3}(0,t;L^{10/3}(\Omega))}
\Vert\uvec_2\Vert_{L^5(0,t;L^5(\Omega)^3)}\notag\\[1mm]
&\leq C \Vert\uvec_2\Vert_{L^5(0,t;L^5(\Omega)^3)}
\Vert\varphi\Vert_{L^\infty(0,t;H)\cap L^2(0,t;V)}\,,
\end{align}
where we have used the embedding $L^\infty(0,t;H)\cap L^2(0,t;V)\hookrightarrow L^{10/3}(0,t;L^{10/3}(\Omega))$,
which is a consequence of Gagliardo-Nirenberg inequality.
By means of this estimate, recalling that $\uvec_2\in L^5(0,T;L^5(\Omega)^3)$, we recover the $L^2(0,t;V_0)-$control of
$\pi$ (similarly to \eqref{e51}).
Hence we again get \eqref{stab-est}. 

If $\eta$ is a positive constant, we can argue in a simpler fashion. Indeed, we first observe that \eqref{e160} immediately
yields that
\begin{align}
&\eta_1\Vert\uvec\Vert^2\leq C\Vert\varphi\Vert^2\,.\label{e188}
\end{align}
On the other hand, the difference of \eqref{E2} is now tested by $\mathcal{N}\varphi$ (rather than by $\varphi$) to give (cf. also
\cite[Proof of Thm.4]{FGG1})
\begin{align}
&\frac{1}{2}\frac{d}{dt}\Vert\mathcal{N}^{1/2}\varphi\Vert^2+\big(B(\varphi_2)-B(\varphi_1),\varphi\big)
+(\uvec_2\cdot\nabla\varphi,\mathcal{N}\varphi)+(\uvec\cdot\nabla\varphi_1,\mathcal{N}\varphi)\notag\\[1mm]
&=\big((m(\varphi _{2})-m(\varphi _{1}))(\nabla J\ast \varphi _{2}),\nabla
\mathcal{N}\varphi \big)+\big(m(\varphi _{1})(\nabla J\ast \varphi ),\nabla\mathcal{N} \varphi 
\big)\,. \label{e187}
\end{align}
Thanks to \textbf{(H7)} and to the Gagliardo-Nirenberg inequality \eqref{GN-3D}, we have that
\begin{align}
\big(B(\varphi_2)-B(\varphi_1),\varphi\big)&\geq \alpha_0\Vert\varphi\Vert^2\,,\label{e189}\\[1mm]
|(\uvec\cdot\nabla\varphi_1,\mathcal{N}\varphi)|&=|(\uvec\,\varphi_1,\nabla\mathcal{N}\varphi)|\leq \Vert\uvec\Vert
\Vert\nabla\mathcal{N}\varphi\Vert\leq\delta\Vert\uvec\Vert^2
+C_\delta\Vert\nabla\mathcal{N}\varphi\Vert^2\,,\\[1mm]
|(\uvec_2\cdot\nabla\varphi,\mathcal{N}\varphi)|&=|(\uvec_2\,\varphi,\nabla\mathcal{N}\varphi)|\leq\Vert\uvec_2\Vert_{L^6(\Omega)^3}
\Vert\varphi\Vert\Vert\nabla\mathcal{N}\varphi\Vert_{L^3(\Omega)^3}\notag\\[1mm]
&\leq C\Vert\uvec_2\Vert_{L^6(\Omega)^3}
\Vert\varphi\Vert^{3/2}\Vert\nabla\mathcal{N}\varphi\Vert^{1/2}\notag\\[1mm]
&\leq \delta\Vert\varphi\Vert^2+C_\delta\Vert\uvec_2\Vert_{L^6(\Omega)^3}^4\Vert\nabla\mathcal{N}\varphi\Vert^2\,.
\label{e190}
\end{align}
The estimates for the two terms on the right hand side of \eqref{e187} are
straightforward. Adding now \eqref{e188}, multiplied by some $\delta^\prime>0$,
together with \eqref{e187} and taking \eqref{e189}-\eqref{e190} into account, we get,
for $\delta,\delta^\prime>0$ small enough,
\begin{align}
&\frac{d}{dt}\Vert\mathcal{N}^{1/2}\varphi\Vert^2+\eta_1\Vert\uvec\Vert^2+\alpha_0\Vert\varphi\Vert^2
\leq C(1+\Vert\uvec_2\Vert_{L^6(\Omega)^3}^4)\Vert\mathcal{N}^{1/2}\varphi\Vert^2\,.\label{e191}
\end{align}
We now observe that $\uvec_2\in L^4(0,T;L^6(\Omega)^3)$ holds, when $\eta$ is constant,
also for weak solutions (cf. Remark \ref{eta-cost-reg}). Therefore, from \eqref{e191}, by means of Gronwall lemma,
we immediately get \eqref{stab-est-2} (the estimate for $\pi$ follows
directly from Darcy's law). The proof is now complete.
\end{proof}

\begin{oss}
\label{holdexp}
A further relaxation of the lower threshold for the H\"{o}lder exponent of the pressure appears to be a major task. One idea could be to start from \eqref{e175} and \eqref{e180}, then use the classical embeddings of
$W^{2,p}(\Omega)$ into H\"{o}lder spaces (e.g., $H^2(\Omega)\hookrightarrow C^\gamma(\overline{\Omega})$,
for all $\gamma\in(0,1)$, if $d=2$) in order to improve the spatial H\"{o}lder exponent of $\pi$ (e.g., from some fixed $\alpha\in(0,1)$
to some $\gamma$ arbitrarily close to $1$). Then one can argue as in the proof of Theorem
\ref{str-sols-fur-reg} with the goal of obtaining
the same exponents $\sigma_p$ and $\mu_p$ of the case $\eta$ constant.
However, it seems hard to increase the time integrability exponent of the pressure at each step of this bootstrap procedure.
Recall indeed that at the beginning of the proof of Theorem \ref{str-sols-fur-reg},
the regularity $\pi\in L^\infty(0,T;C^\alpha(\overline{\Omega}))$ is taken into account.
\end{oss}

\section{The convective nonlocal Cahn-Hilliard equation}
\setcounter{equation}{0}
\label{cnCH}




\noindent
Here we report some improvements of former results contained in \cite{FGGS,FGR}.
These results are concerned with the existence of weak/strong solutions to the convective nonlocal Cahn-Hilliard equation
with a prescribed divergence-free velocity field and their uniqueness. These results are used in Section
\ref{sec:strongex}.
\begin{thm}
\label{reg-thm-bis}
Suppose that $d=2$ or $d=3$. Let assumptions \textbf{(H2)}-\textbf{(H6)} be satisfied
and suppose $\varphi _{0}\in L^{\infty }(\Omega )$ such that 
$M(\varphi _{0})\in L^{1}(\Omega )$, where $M$
is defined as in Theorem \ref{existweak}. If $\uvec\in L^2(0,T; G_{div})$,
for a given $T>0$, then there exists a (weak) solution $\varphi$ to \eqref{E2}-\eqref{I1} such that
\begin{align}
& \varphi \in L^{\infty }(0,T;L^p(\Omega))\cap H^1(0,T;V^\prime)\,,
\quad \varphi\in L^{2}(0,T;V)\,,\quad \forall\,p\in [2,\infty)\,, \\
& \varphi \in L^{\infty }(Q_{T})\,,\quad |\varphi (x,t)|\leq 1\quad \mbox{ for a.e. }(x,t)\in Q_{T}\,.
\label{regNCH2}
\end{align}
In addition to \textbf{(H2)}-\textbf{(H6)}, assume that \textbf{(H7)}-\textbf{(H8)} hold
and suppose that  $J\in W_{loc}^{2,1}(\mathbb{R}^{d})$ or that $J$ is admissible.
Let $\varphi _{0}\in V\cap L^{\infty }(\Omega )$ with 
$M(\varphi _{0})\in L^{1}(\Omega )$. If $\uvec$ satisfies
\begin{align}
&\uvec\in L^{\beta_r}(0,T;L_{div}^{r}(\Omega )^{d})\,,\quad\mbox{where}\,
\beta_r=\left\{\begin{array}{lll}
\frac{2r}{r-2}\,,\quad\mbox{with }\,\,2<r\leq \infty\,,\quad \mbox{if }\,\, d=2\,,\\[1mm]
\frac{r}{r-3}\,,\quad\mbox{with }\,\,3<r\leq 4\,,\quad \mbox{if }\,\, d=3\,,\\[1mm]
\frac{2r}{r-2}\,,\quad\mbox{with }\,\,4<r\leq \infty\,,\quad \mbox{if }\,\, d=3\,,
\end{array}\right.
\label{regvel-bis}
\end{align}%
for some given $T>0$, then there exists a  strong solution $\varphi$ to \eqref{e25}, \eqref{BCbis}, \eqref{I1} which fulfils \eqref{regNCH2} and
\begin{align}
& \varphi \in L^{\infty }(0,T;V)\cap H^{1}(0,T;H)\,,\qquad
\varphi \in L^{2}(0,T;H^{2}(\Omega ))\,.\label{regNCH1}
\end{align}%

Let \textbf{(H2)}-\textbf{(H4)}, and \textbf{(H7)} hold.
If $\lambda$ is a positive constant or if $\uvec$ satisfies
  \begin{align}
   &\uvec\in L^{\gamma_r}(0,T;L_{div}^{r}(\Omega )^{d})\,,\quad\mbox{ where }\quad
   \gamma_r=\frac{2r}{r-d}\,,\quad d<r\leq \infty\,,
  \label{regvel-tris}
  \end{align}
then weak solutions are unique.
Moreover, if $\varphi$ is a strong solution then the following differential identity holds
\begin{align}
& \frac{1}{2}\frac{d\Phi }{dt}+\Vert\sqrt{\lambda (\varphi )}\varphi_{t}\Vert^{2}
+\big(\uvec\cdot \nabla \varphi ,\lambda (\varphi )\varphi
_{t}\big) \notag  \\
&\qquad =-\big(m^{\prime }(\varphi )\varphi _{t}(\nabla J\ast \varphi ),\lambda
(\varphi )\nabla \varphi \big)-\big(m(\varphi )(\nabla J\ast \varphi
_{t}),\lambda (\varphi )\nabla \varphi \big)\,,  \label{diffid}
\end{align}%
where
\begin{equation*}
\Phi :=\Vert \nabla B(\varphi )\Vert ^{2}-2\big(m(\varphi )(\nabla J\ast
\varphi ),\lambda (\varphi )\nabla \varphi \big)\,.
\end{equation*}%
\end{thm}

\begin{proof}
We use the arguments of \cite{FGGS,FGR}. Therefore we will focus on the points where the results of \cite{FGGS,FGR} are improved.
To prove existence of weak solutions,
the approximation scheme follows the lines of the proofs of \cite[Thms.1, 2, 4]{FGR},
using a regularization of the degenerate mobility and singular potential combined with a Galerkin scheme (see also Section \ref{sec:weakex}).
The assumption on $\uvec$ is more general than in \cite[Theorem 4]{FGR} and can be handled
by means of a suitable divergence-free regularization of $\uvec$ then passing to the limit
with respect to the regularization parameter.

The existence of a strong solution can be proven in the same fashion as
in the proof of \cite[Theorem 6.1]{FGGS}. However, the assumption on $\uvec$ (see \cite[(6.1)]{FGGS})
can be relaxed for the case $d=3$ by performing estimate \cite[(6.6)]{FGGS} in a slightly different way.
The difference is the handling of the contribution coming from the convective term in the time-discretization scheme. Indeed, using the same notation
as in \cite{FGGS}, instead of the Gagliardo-Nirenberg inequality, inequality \eqref{GN-spe} can be used to estimate
the norm of $\nabla B(\varphi_{k+1})$. We distinguish two cases. If $4\leq 2r/(r-2)<6$, namely, if $3<r\leq 4$,
we can write (use \eqref{e102} with $p=2r/(r-2)$)
\begin{align}
& \tau \sum_{k=0}^{n}\Vert \boldsymbol{U}_{k}\cdot \nabla B(\varphi_{k+1})\Vert ^{2}\leq \tau \sum_{k=0}^{n}
\Vert \boldsymbol{U}_{k}\Vert_{L^{r}(\Omega )^{3}}^{2}\Vert \nabla B(\varphi _{k+1})\Vert_{L^{2r/(r-2)}(\Omega )^{3}}^{2}
\notag\\[1mm]
&\leq \tau \sum_{k=0}^{n}
\Vert \boldsymbol{U}_{k}\Vert_{L^{r}(\Omega )^{3}}^{2}\Vert B(\varphi _{k+1})\Vert_{L^{\infty}(\Omega )}^{\frac{4(r-3)}{r}}
\Vert B(\varphi _{k+1})\Vert_{H^2(\Omega)}^{\frac{2(6-r)}{r}}\notag\\[1mm]
&\leq \delta \tau \sum_{k=0}^{n}\Vert B(\varphi _{k+1})\Vert _{H^{2}(\Omega)}^{2}
+C_{\delta }\,\tau \sum_{k=0}^{n}\Vert \boldsymbol{U}_{k}\Vert_{L^{r}(\Omega )^{3}}^{\frac{r}{r-3}}\,,
\label{e171}
\end{align}
while, if $2\leq 2r/(r-2)<4$, namely, if $4< r\leq\infty$, we can write (use \eqref{e170} with $p=2r/(r-2)$)
\begin{align}
& \tau \sum_{k=0}^{n}\Vert \boldsymbol{U}_{k}\cdot \nabla B(\varphi_{k+1})\Vert ^{2}\leq \tau \sum_{k=0}^{n}
\Vert \boldsymbol{U}_{k}\Vert_{L^{r}(\Omega )^{3}}^{2}\Vert \nabla B(\varphi _{k+1})\Vert_{L^{2r/(r-2)}(\Omega )^{3}}^{2}
\notag\\[1mm]
&\leq \tau \sum_{k=0}^{n}
\Vert \boldsymbol{U}_{k}\Vert_{L^{r}(\Omega )^{3}}^{2}\Vert\nabla B(\varphi_{k+1})\Vert^{2(1-\frac{4}{r})}
\Vert B(\varphi _{k+1})\Vert_{L^{\infty}(\Omega )}^{\frac{4}{r}}
\Vert B(\varphi _{k+1})\Vert_{H^2(\Omega)}^{\frac{4}{r}}\notag\\[1mm]
&\leq \delta \tau \sum_{k=0}^{n}\Vert B(\varphi _{k+1})\Vert _{H^{2}(\Omega)}^{2}
+C_{\delta }\,\tau \sum_{k=0}^{n}\Vert \boldsymbol{U}_{k}\Vert_{L^{r}(\Omega )^{3}}^{\frac{2r}{r-2}}
\Vert\nabla B(\varphi_{k+1})\Vert^{2\frac{r-4}{r-2}}\notag\\
&\leq \delta \tau \sum_{k=0}^{n}\Vert B(\varphi _{k+1})\Vert _{H^{2}(\Omega)}^{2}
+C_{\delta }\,\tau \sum_{k=0}^{n}\Vert \boldsymbol{U}_{k}\Vert_{L^{r}(\Omega )^{3}}^{\frac{2r}{r-2}}
\,\big(\Vert\nabla B(\varphi_{k+1})\Vert^{2}+1\big)\,.
\label{e172}
\end{align}
Note that this last estimate also holds for $d=2$.
In both cases we have taken advantage of the uniform bound in $L^\infty(\Omega)$ for the time discrete solutions
$\varphi_{k+1}$ (see the proof of \cite[Theorem 6.1]{FGGS}).
Then we employ the estimate
\begin{equation}
\tau \sum_{k=0}^{n}\Vert \boldsymbol{U}_{k}\Vert _{L^{r}(\Omega )^{3}}^{\beta_r}
\leq \Vert \boldsymbol{u}\Vert _{L^{\beta_r}(0,T;L^{r}(\Omega )^{3})}^{\beta_r}\,,
\end{equation}
where $\beta_r=r/(r-3)$, or $\beta_r=2r/(r-2)$,
in \eqref{e171} or \eqref{e172}, respectively.
Thus we can conclude as in the proof of \cite[Theorem 6.1]{FGGS} by means of the discrete Gronwall lemma.

The uniqueness argument follows the lines of the proof
of \cite[Proposition 4]{FGR}, for weak solutions, and of Part (c) of the proof of \cite[Theorem 6.1]{FGGS},
for strong solutions in two dimensions. We point out that, in order to prove uniqueness, the available techniques are
essentially two. The first one consists in testing the identity resulting from the difference
of the convective nonlocal Cahn-Hilliard equation (written for each solution $\varphi_1,\varphi_2$)
 by $\varphi:=\varphi_1-\varphi_2$. Alternatively, we can test by $\mathcal{N}\varphi$.
The former choice has the advantage that we get rid
of the contribution of the convective term since $\uvec$ is divergence-free, but it leads us to deal with
the term $\lambda(\varphi_1)-\lambda(\varphi_2)$ (unless $\lambda$ is constant). For this reason, we need to
work with strong solutions and we can expect to prove only a weak-strong uniqueness result in dimension two.
On the other hand, testing by $\mathcal{N}\varphi$ has the advantage that we do not have to deal
with the above term. Therefore the argument also works for weak solutions as well as for non-constant $\lambda$. The drawback
is the convective term, namely $(\uvec\cdot\nabla\varphi,\mathcal{N}\varphi)$,
has to be handled. This forces us to make some stronger integrability assumption on the given velocity field $\uvec$.
In particular, we can suppose $\uvec\in L^2(0,T;L^\infty(\Omega)^d)$, with
$\mbox{div}(\uvec)=0$ (see \cite[Theorem 4]{FGR}). This condition can be relaxed
by estimating the term $(\uvec\cdot\nabla\varphi,\mathcal{N}\varphi)$ in a different fashion
(compare with \cite[(6.9)]{FGR}), namely,
\begin{align}
|(\uvec\cdot\nabla\varphi,\mathcal{N}\varphi)|&\leq |(\uvec\varphi,\nabla\mathcal{N}\varphi)|\leq
\Vert\uvec\Vert_{L^r(\Omega)^3}\Vert\varphi\Vert\Vert\nabla\mathcal{N}\varphi\Vert_{L^{2r/(r-2)}(\Omega)^3}\notag\\[1mm]
&\leq C \Vert\uvec\Vert_{L^r(\Omega)^3}\Vert\varphi\Vert\Vert\nabla\mathcal{N}\varphi\Vert^{\frac{r-3}{r}}
\Vert\nabla\mathcal{N}\varphi\Vert_{V^3}^{\frac{3}{r}}
\leq C \Vert\uvec\Vert_{L^r(\Omega)^3}\Vert\varphi\Vert^{\frac{r+3}{r}}\Vert\nabla\mathcal{N}\varphi\Vert^{\frac{r-3}{r}}
\notag\\[1mm]
&\leq \delta\Vert\varphi\Vert^2+C_\delta\Vert\uvec\Vert_{L^r(\Omega)^3}^{\frac{2r}{r-3}}\Vert\nabla\mathcal{N}\varphi\Vert^2\,,
\label{e173}
\end{align}
where $3<r\leq \infty$. Here the Gagliardo-Nirenberg inequality in dimension three has been used (in dimension two one can argue in a similar way).
On account of \eqref{e173}, we can proceed as in the proof of \cite[Proposition 4]{FGR} and deduce that uniqueness
of weak solutions holds under the assumption \eqref{regvel-tris}.
Observe that, if $d=2$ then we have that $\gamma_r=\beta_r$ (for all $2<r\leq\infty$). Thus the condition ensuring
existence of a strong solution also guarantees its uniqueness. Instead, if $d=3$, we have that $\gamma_r>\beta_r$
(unless $r=\infty$). Therefore, in order to ensure uniqueness of the strong solution we need a stronger assumption
on $\uvec$ than the one which only guarantees its existence. We recall that $\uvec\in L^2(0,T;L^\infty(\Omega)^3)$
is the only assumption which ensures both existence and uniqueness of the strong solution.

If $\lambda$ is a positive constant, we can test the difference of the nonlocal Cahn-Hilliard equation
by $\varphi$. Hence we do not have to consider the contribution of the convective term so that assumption \eqref{regvel-tris} is no longer needed. For this
reason \textbf{(H2)}-\textbf{(H4)}, \textbf{(H7)} are enough for establishing uniqueness of weak solutions.

Finally, the differential identity \eqref{diffid} for strong solutions
 can be formally deduced by taking $\psi=B(\varphi)_t$
in the variational formulation \eqref{e174}. This choice of test function is just formal but
it can be made rigorous, for instance, by means of a regularization
procedure which employs time convolutions and by passing to the limit (using strong convergences) with respect to
the convolution regularization parameter (see \cite[Chap.II, Lemma 4.1]{Temam}.
\end{proof}

\section{Concluding remarks}
\setcounter{equation}{0}
\label{sec:conrem}
It would be nice to remove (or improve) the condition $\alpha>1/5$ on the H\"{o}lder exponent of the pressure in the weak-strong uniqueness in dimension three with $\eta$ variable, but this does not seem easy (see Remark \ref{holdexp}).

Our results suggest that optimal control problems like
the one studied in \cite{FGS} can also be analyzed in three dimensions if $\eta$ is constant and in two dimensions if $\eta$ is variable. In this spirit, one can try to extend the present analysis to a system with sources (see \cite{GaLa,JWZ} and their references) and to formulate and study appropriate optimal control problems also in this case (see \cite{SW}).

In the context of tumor growth models, another challenging issue could be the analysis of multi-species non-local systems (see, for instance, \cite{DFRSS} and references therein for the local Cahn-Hilliard-Darcy system). More precisely, the goal is to formulate and study multi-component nonlocal Cahn-Hilliard equations with sources governed by suitable reaction-diffusion equations. We believe that, on account of the results obtained in this paper, we could go beyond the mere existence of a weak solution. It is worth observing that nonlocal models for tumor growth have been recently considered in \cite{FLOW,FLNOW} from a theoretical and numerical viewpoint.

\section{Appendix: Gagliardo-Nirenberg inequalities}
\setcounter{equation}{0}
\label{sec:GNineq}
For the reader's convenience, we report here below a generalization of the Gagliardo-Nirenberg inequality for fractional Sobolev spaces
given by \cite[Theorem 1]{BM1} and by \cite[Theorem 1]{BM2} which is used in the previous sections.
\begin{prop}\label{Bre-Mir}
Let $\Omega\subset\mathbb{R}^d$ be a Lipschitz bounded domain. Let $s_1,p_1,s_2,p_2,r,q,\theta$ and $d$ satisfy
\begin{equation}\label{Bre-Mir-cond}
\begin{aligned}
&0\leq s_1\leq s_2\,,\,\,\,r\geq 0\,,\,\,\,1\leq p_1,p_2,q\leq\infty\,,\,\,\,(s_1,p_1)\neq (s_2,p_2)\,,\,\,\,\theta\in(0,1)\,,\\[1mm]
&\frac{1}{q}=\Big(\frac{\theta}{p_1}+\frac{1-\theta}{p_2}\Big)-\frac{s-r}{d}\,,\,\,\,s:=\theta s_1+(1-\theta) s_2\,,\,\,\,r<s\,.
\end{aligned}
\end{equation}
Then the following Gagliardo-Nirenberg-Sobolev inequality holds
\begin{align}\label{GNS}
&\Vert u\Vert_{W^{r,q}(\Omega)}\leq C\Vert u\Vert_{W^{s_1,p_1}(\Omega)}^\theta\Vert u\Vert_{W^{s_2,p_2}(\Omega)}^{1-\theta}\,,\quad
\forall u\in W^{s_1,p_1}(\Omega)\cap W^{s_2,p_2}(\Omega)\,,
\end{align}
with the following exceptions, when it fails,
\begin{enumerate}
  \item $d=1$, $s_2$ is an integer $\geq 1$, $1<p_1\leq\infty$, $p_2=1$, $s_1=s_2-1+\frac{1}{p_1}$,\\
  $[1<p_1<\infty\,,r=s_2-1]$ or $\Big[s_2+\frac{\theta}{p_1}-1<r<s_2+\frac{\theta}{p_1}-\theta\Big]$\,;
  \item $d\geq 1$, $s_1<s_2$, $s_1-\frac{d}{p_1}=s_2-\frac{d}{p_2}=r$ is an integer, $q=\infty$, $(p_1,p_2)\neq (\infty,1)$
  (for every $\theta\in(0,1)$).
\end{enumerate}
Moreover, if in \eqref{Bre-Mir-cond} we have $r=s$, then \eqref{GNS} still holds if and only if the following condition fails
\begin{align}
&s_2\mbox{ is an integer }\geq 1,\,\,p_2=1\mbox{ and }0<s_2-s_1\leq 1-\frac{1}{p_1}.
\label{GN-ex}
\end{align}

\end{prop}
\begin{oss}\label{Sobolev-emb}
If $(s_1,p_1)=(s_2,p_2)$ and $0\leq r<s=s_1=s_2$ in \eqref{Bre-Mir-cond}, then estimate \eqref{GNS} is equivalent to the embedding
(see \cite[Theorem B]{BM2})
\begin{align}\label{Sob-emb}
&W^{s,p}(\Omega)\hookrightarrow W^{r,q}(\Omega)\,,
\end{align}
which holds provided that $1\leq p< q\leq\infty$ and
$$r-\frac{d}{q}=s-\frac{d}{p}\,,$$
with the following exceptions, when \eqref{Sob-emb} fails,
\begin{enumerate}
  \item $d=1$, $s$ is an integer $\geq 1$, $p=1$, $1<q<\infty$ and $r=s-1+\frac{1}{q}$\,;
  \item $d\geq 1$, $1<p<\infty$, $q=\infty$ and $s-\frac{d}{p}=r\geq 0$ is an integer.
\end{enumerate}
\end{oss}
\begin{oss}\label{Sobolev-scases}
The following special case of the Gagliardo-Nirenberg-Sobolev inequality \eqref{GNS}, that holds true
for a bounded smooth domain $\Omega\subset\mathbb{R}^d$, $d=2,3$, is useful as well
\begin{align}\label{GN-spe}
&\Vert \nabla u\Vert_{L^p(\Omega)^d}\leq
C\Vert u\Vert_{L^\infty(\Omega)}^{1-\hat{\alpha}}\Vert u\Vert_{H^2(\Omega)}^{\hat{\alpha}}\,,\qquad\forall u\in H^2(\Omega)\,,
\end{align}
where
$$\hat{\alpha}=\frac{2}{p}\,\frac{p-d}{4-d}\,\quad\mbox{ and }\,\left\{\begin{array}{ll}
4\leq p<\infty\,,\qquad \mbox{if } d=2\,,\\
4\leq p\leq 6\,,\qquad \mbox{if } d=3\,.
\end{array}\right.
$$
Finally, we also recall other special cases of \eqref{GNS}
\begin{align}
&\Vert u\Vert_{L^p(\Omega)}\leq C\Vert u\Vert^{\frac{2}{p}}\Vert u\Vert_V^{1-\frac{2}{p}}\,,\quad\forall u\in V\,,
\quad 2\leq p<\infty\,,\quad d=2\,,\label{GN-2D}\\[1mm]
&\Vert u\Vert_{L^p(\Omega)}\leq C\Vert u\Vert^{\frac{6-p}{2p}}\Vert u\Vert_V^{\frac{3(p-2)}{2p}}\,,\quad\forall u\in V\,,
\quad 2\leq p\leq 6\,,\quad d=3\,.\label{GN-3D}
\end{align}
\end{oss}

\bigskip 
\noindent
{\bf Acknowledgment.} The authors are members of Gruppo Nazionale per l'Analisi Ma\-te\-ma\-ti\-ca, la Probabilit\`{a} e le loro Applicazioni (GNAMPA), Istituto Nazionale di Alta Matematica (INdAM).

\bigskip 
\noindent
{\bf Conflict of interest statement.} On behalf of all authors, the corresponding author states that there is no conflict of interest.


\begin{thebibliography}{30}

\bibitem{Am} H. Amann, {\it  Global existence for semilinear parabolic systems}, J. Reine Angew. Math. \textbf{360} (1985),
47-83.

\bibitem{BRB} J. Bedrossian, N. Rodr\'{\i}guez, A. Bertozzi, {\itshape Local and global well-posedness for an aggregation equation and
Patlak-Keller-Segel models with degenerate diffusion}, Nonlinearity \textbf{24} (2011), 1683-1714.


\bibitem{BF} F. Boyer, P. Fabrie, Mathematical Tools for the Study of the Incompressible Navier-Stokes Equations and Related Models,
Appl. Math. Sci. \textbf{183}, Springer, New York, 2013.


\bibitem{BM1} H. Br\'{e}zis, P. Mironescu, {\it Where Sobolev interacts with Gagliardo-Nirenberg}, J. Funct. Anal. \textbf{277}
(2019), 2839-2864.


\bibitem{BM2} H. Br\'{e}zis, P. Mironescu, {\it Gagliardo-Nirenberg inequalities and non-inequalities: The full story},
Ann. Inst. H. Poincar\'{e} Anal. Non Lin\'{e}aire \textbf{35} (2018), 1355-1376.


\bibitem{CHOT} A. Cheskidov, D.D. Holm, E. Olson, E.S. Titi, {\it On a Leray-$\alpha$ model of turbolence},
Proc. Roy. Soc. A \textbf{461} (2005), 629-649.


\bibitem{CFG} P. Colli, S. Frigeri, M. Grasselli, {\itshape Global existence of weak solutions to a nonlocal Cahn-Hilliard-Navier-Stokes system}, J. Math. Anal. Appl.
\textbf{386} (2012), 428-444.

\bibitem{CG} M. Conti, A. Giorgini, {\itshape Well-posedness for the Brinkman-Cahn-Hilliard system with unmatched viscosities}, J. Differential Equations \textbf{268} (2020),
6350-6384.

\bibitem{DFRSS} M. Dai, E. Feireisl, E. Rocca, G. Schimperna, M.E. Schonbek, {\itshape Analysis of a diffuse interface model of
multispecies tumor growth}, Nonlinearity \textbf{30} (2017), 1639-1658.

\bibitem{DGL} L. Ded\'{e}, H. Garcke, K.F. Lam, {\itshape A Hele-Shaw-Cahn-Hilliard model for incompressible two-phase flows with different
densities}, J. Math. Fluid Mech. \textbf{20} (2018), 531-567.

\bibitem{DPGG} F. Della Porta, A. Giorgini, M. Grasselli, {\itshape The nonlocal Cahn-Hilliard-Hele-Shaw system with logarithmic potential},
Nonlinearity \textbf{31} (2018), 4851-4881.

\bibitem{DiBen} E. DiBenedetto, Partial Differential Equations, Second Edition, Birkh\"{a}user, Boston, 2010.

\bibitem{DiB} E. DiBenedetto, Real Analysis, Birkh\"{a}user, Boston, 2002.

\bibitem{EV} L.C. Evans, Partial Differential Equations, Grad. Stud. Math. \textbf{19}, American Mathematical Society, Providence (RI), 1998.

\bibitem{Fr} S. Frigeri, {\itshape On a nonlocal Cahn-Hilliard/Navier-Stokes system with degenerate mobility and singular potential for incompressible fluids with different densities}, Ann. Inst. H. Poincar\'{e} Anal. Non Lin\'{e}aire \textbf{38} (2020), 647-687.

\bibitem{FGG1} S. Frigeri, C.G. Gal, M. Grasselli, {\itshape On nonlocal Cahn-Hilliard-Navier-Stokes systems in two dimensions,}
J. Nonlinear Sci. \textbf{26} (2016), 847-893.

\bibitem{FGG2} S. Frigeri, C.G. Gal, M. Grasselli, {\itshape Regularity results for the nonlocal Cahn-Hilliard equation with singular potential and degenerate mobility,}
J. Differential Equations \textbf{287} (2021), 295-328.

\bibitem{FGGS} S. Frigeri, C.G. Gal, M. Grasselli, J. Sprekels,
{\itshape  Two-dimensional nonlocal Cahn-Hilliard-Navier-Stokes systems with variable viscosity, degenerate mobility and singular potential}, Nonlinearity \textbf{32} (2019), 678-727.

\bibitem{FGR} S. Frigeri, M. Grasselli, E. Rocca, {\itshape A diffusive interface model for two-phase incompressible flows
with nonlocal interactions and nonconstant mobility}, Nonlinearity \textbf{28} (2015), 1257-1293.


\bibitem{FGS} S. Frigeri, M. Grasselli, J. Sprekels, {\itshape Optimal distributed control of two-dimensional nonlocal Cahn-Hilliard-Navier-Stokes systems with
degenerate mobility and singular potential}, Appl. Math. Optim. \textbf{81} (2020), 899-931.

\bibitem{FLOW} M. Fritz, E.A.B.F. Lima, J.T. Oden, B. Wohlmuth, {\itshape On the unsteady Darcy-Forchheimer-Brinkman equation in local and nonlocal tumor growth models}, Math. Models Methods Appl. Sci. \textbf{29} (2019), 1691-1731.

\bibitem{FLNOW} M. Fritz, E.A.B.F. Lima, V. Nikoli\'{c}, J.T. Oden, B. Wohlmuth, {\itshape Local and nonlocal phase-field models of tumor growth and invasion due to ECM degradation}, Math. Models Methods Appl. Sci. \textbf{29} (2019), 2433-2468.

\bibitem{FM} D. Fujiwara and H. Morimoto, {\it An $L_r$-theorem of the Helmholtz decomposition of vector fields,}
J. Fac. Sci. Univ. Tokyo, Sec. I \textbf{24} (1977), 685-700.

\bibitem{GZ} H. Gajewski, K. Zacharias, {\itshape On a nonlocal phase separation model}, J. Math. Anal. Appl. \textbf{286} (2003), 11-31.

\bibitem{GGG} C.G. Gal, A. Giorgini, M. Grasselli, {\itshape The nonlocal Cahn-Hilliard equation with singular potential:
well-posedness, regularity and strict separation property}, J. Differential Equations  \textbf{263} (2017), 5253-5297.

\bibitem{Galdi} G.P. Galdi, An Introduction to the Mathematical Theory of the Navier-Stokes Equations, Steady-State Problems,
Springer Monogr. Math., Springer, New York, Second Edition, 2011.

\bibitem{GaLa} H. Garcke, K.F. Lam, {\itshape Global weak solutions and asymptotic limits of a Cahn-Hilliard-Darcy system modelling tumour growth}, AIMS Math. \textbf{1} (2016), 318-260.

\bibitem{GL1} G. Giacomin, J.L. Lebowitz, {\itshape Exact macroscopic description of phase segregation in model
alloys with long range interactions},  Phys. Rev. Lett. \textbf{76} (1996), 1094-1097.

\bibitem{GL2} G. Giacomin, J.L. Lebowitz, {\itshape Phase segregation dynamics in particle systems with long
range interactions. I. Macroscopic limits}, J. Stat. Phys. \textbf{87} (1997), 37-61.

\bibitem{GL3} G. Giacomin, J.L. Lebowitz, {\itshape Phase segregation dynamics in particle systems with long
range interactions. I. Phase motion}, SIAM J. Appl. Math. \textbf{58} (1998), 1707-1729.

\bibitem{Giga1} Y. Giga, {\itshape Analyticity of the semigroup generated by the Stokes operator in $L_r$ spaces},
Math. Z. \textbf{178} (1981), 297-329.

\bibitem{Giga2} Y. Giga, {\itshape The Stokes operator in $L_r$ spaces}, Proc. Japan Acad. \textbf{57} (1981), 85-89.

\bibitem{GMi} Y. Giga and T. Miyakawa, {\it Solutions in $L_r$ of the Navier-Stokes initial value problem},
Arch. Rational Mech. Anal. \textbf{89} (1985), 267-281.

\bibitem{Gio} A. Giorgini, {\itshape Well-posedness for a diffuse interface model for two-phase Hele-Shaw flows}, J. Math. Fluid Mech., in press.

\bibitem{GGW} A. Giorgini, M. Grasselli, H. Wu, {\itshape The Cahn-Hilliard-Hele-Shaw system with singular potential}, Ann. Inst. H.
Poincar\'{e} Anal. Non Lin\'{e}aire \textbf{35} (2018), 1079-1118.

\bibitem{JWZ} J. Jiang, H. Wu, S. Zheng, {\itshape Well-posedness and long-time behavior of a non-autonomous Cahn-Hilliard-Darcy system with mass source modeling tumor growth},
 J. Differential Equations \textbf{259} (2015), 3032-3077.

\bibitem{LLG1} H.-G. Lee, J.S. Lowengrub, J. Goodman, {\itshape Modeling pinch-off and reconnection in a Hele-Shaw cell. I. The models
and their calibration}, Phys. Fluids \textbf{14} (2002), 492-512.

\bibitem{LLG2} H.-G. Lee, J.S. Lowengrub, J. Goodman, {\itshape Modeling pinch-off and reconnection in a Hele-Shaw cell. II. Analysis and
simulation in the nonlinear regime}, Phys. Fluids \textbf{14} (2002), 514-545.



\bibitem{SW} J. Sprekels, H. Wu, {\itshape Optimal distributed control of a Cahn-Hilliard-Darcy
System with mass sources}, Appl. Math. Optim. \textbf{83} (2021), 489-530.

\bibitem{Temam} R. Temam, Infinite-dimensional dynamical systems in mechanics and physics, 2d ed., Springer-Verlag, New-York, 1997.

\bibitem{T} R. Temam, Navier-Stokes Equations. Theory and Numerical Analysis, Third Edition, North-Holland, Oxford, 1984.

\bibitem{WW} X. Wang, H. Wu, {\itshape Long-time behavior for the Hele-Shaw-Cahn-Hilliard system}, Asymptot. Anal. \textbf{78} (2012), 217-245.

\bibitem{WZ} X. Wang, Z. Zhang, {\itshape Well-posedness of the Hele-Shaw-Cahn-Hilliard system}, Ann. Inst. H. Poincar\'{e} Anal. Non
Lin\'{e}aire \textbf{30} (2013), 367-384.



\end{thebibliography}
\end{document}